\documentclass[11pt,a4paper]{article}
\usepackage[utf8]{inputenc}

\usepackage{styling}
 
\hypersetup{
  pdftitle={Mean-Biased Processes for Balanced Allocations},
  pdfauthor={Dimitrios Los, Thomas Sauerwald, John Sylvester}
  }

\allowdisplaybreaks

\title{\textbf{Mean-Biased Processes for Balanced Allocations}\footnote{Some results from the paper were presented at SODA 2022~\cite{OurSODA}.}}

\author[1]{Dimitrios Los}  
\author[1]{Thomas Sauerwald\thanks{T.S. was supported by the ERC Starting Grant 679660 (DYNAMIC MARCH).}}
\author[2]{John Sylvester\thanks{J.S. was supported by EPSRC project EP/T004878/1: Multilayer Algorithmics to Leverage Graph Structure.}}
\affil[1]{Department of Computer Science \& Technology, University of Cambridge, UK\\ \texttt{firstname.lastname@cl.cam.ac.uk}} 
\affil[2]{Department of Computer Science, University of Liverpool, UK\\ \texttt{john.sylvester@liverpool.ac.uk}}

\newcommand{\WOne}{{\hyperlink{w1}{\ensuremath{\mathcal{W}_1}}}\xspace}
\newcommand{\WTwo}{{\hyperlink{w2}{\ensuremath{\mathcal{W}_2}}}\xspace}

\newcommand{\POne}{{\hyperlink{p1}{\ensuremath{\mathcal{P}_1}}}\xspace}
\newcommand{\PTwo}{{\hyperlink{p2}{\ensuremath{\mathcal{P}_2}}}\xspace}
\newcommand{\PThree}{{\hyperlink{p3}{\ensuremath{\mathcal{P}_3}}}\xspace}
\date{\vspace{-5ex}}
\begin{document}
 
\maketitle

\begin{abstract}

 We introduce a new class of balanced allocation processes which bias towards underloaded bins (those with load below the mean load) either by skewing the probability by which a bin is chosen for an allocation (\textit{probability bias}), or alternatively, by adding more balls to an  underloaded bin (\textit{weight bias}). A prototypical process satisfying the \textit{probability bias} condition is \textsc{Mean-Thinning}: At each round, we sample one bin and if it is underloaded, we allocate one ball; otherwise, we allocate one ball to a {\em second} bin sample. Versions of this process have been in use since at least 1986.
 An example of a process, introduced by us, which satisfies the \textit{weight bias} condition is \textsc{Twinning}: At each round, we only sample one bin. If the bin is underloaded, then we allocate two balls; otherwise, we allocate only one ball. 
 
 Our main result is that for any process with a probability or weight bias, with high probability the gap between maximum and minimum load is logarithmic in the number of bins. This result holds for any number of allocated balls (heavily loaded case), covers many natural processes that relax the \textsc{Two-Choice} process, and we also prove it is tight for many such processes, including \textsc{Mean-Thinning} and \textsc{Twinning}. 
 
 Our analysis employs a delicate interplay between  linear, quadratic and exponential potential functions.
 It also hinges on a phenomenon we call ``mean quantile stabilization'', which holds in greater generality than our framework and may be of independent interest. 

	\medskip
	
	\noindent \textbf{\textit{Keywords---}}  Balls-into-bins, balanced allocations, potential functions, heavily loaded, gap bounds, maximum load, thinning, two-choices, weighted balls. \\
	\textbf{\textit{AMS MSC 2010---}} 68W20, 68W27, 68W40, 60C05
	
\end{abstract}

\clearpage

\clearpage
\tableofcontents
~
\clearpage
\section{Introduction}

We consider the sequential allocation of $m$ balls (jobs or data items) to $n$ bins (servers or memory cells), by allowing each ball to choose from a set of randomly sampled bins. The goal is to allocate balls efficiently, while also keeping the load vector balanced. The \textit{balls-into-bins} framework has found numerous applications in hashing, load balancing, routing, but also has connections to more theoretical topics such as randomized rounding or pseudorandomness (we refer to the surveys~\cite{MRS01} and~\cite{W17} for more details).

A classical algorithm is the \DChoice process introduced by Azar, Broder, Karlin, Upfal~\cite{ABKU99} and Karp, Luby, and Meyer auf der Heide~\cite{KLM96}, where at each round, we sample $d \geq 1$ bins uniformly at random and then allocate the ball to the least loaded of the $d$ sampled bins. It is well-known that for the \OneChoice process ($d=1$), the \textit{gap} between the maximum and mean load is $\Theta\bigr( \sqrt{ (m/n) \cdot \log n } \bigr)$ for $m \geq n \log n$. In particular, this gap grows significantly as $m/n \rightarrow \infty$, which is called the {\em heavily loaded case}. For \TwoChoice ($d=2$), it was proved~\cite{ABKU99} that the gap is only $\log_2 \log n+\Oh(1)$ for $m=n$. This result was generalized by Berenbrink, Czumaj, Steger and V\"{o}cking~\cite{BCSV06} who proved that the same guarantee also holds for $m \geq n$, in other words, even as $m/n \rightarrow \infty$, the difference between the maximum and mean load remains bounded by a slowly growing function in $n$ that is independent of $m$.
This dramatic improvement of \TwoChoice over \OneChoice is widely known as the ``power of two choices'' paradigm.

The importance of this paradigm was recently recognized in the 2020 ACM Paris Kanellakis Theory and Practice award~\cite{ACM20}. However, in some real-world applications the original \TwoChoice process has been shown to underperform and so variants of \TwoChoice are being used instead (e.g.~\cite{OWZS13,LXKGLG11,WKKA23}). These variants relax the assumptions of \TwoChoice and can thus be regarded as more \textit{sample-efficient} and \textit{robust}.

As noted in~\cite{LXKGLG11}, in some real-word systems one important shortcoming of \TwoChoice are its communication requirements:
\begin{quote}
\textit{More importantly, the \textsc{[Two-Choice]} algorithm requires communication between dispatchers and processors at the time of job assignment. The communication time is on the critical path, hence contributes to the increase in response time.}
\end{quote}

An important family of processes, \TwoThinning attempts to address this issue. In this process, each ball first samples a bin uniformly at random and if its load is below some threshold, the ball is allocated to that sample. Otherwise, the ball is allocated to a {\em second} bin sampled uniformly at random, without inspecting its load. These processes have the advantage that they relax communication between the job dispatchers and servers (see \cref{fig:two_choice_vs_two_thinning}). 

\begin{figure}
  \centering
  \includegraphics[scale=0.6]{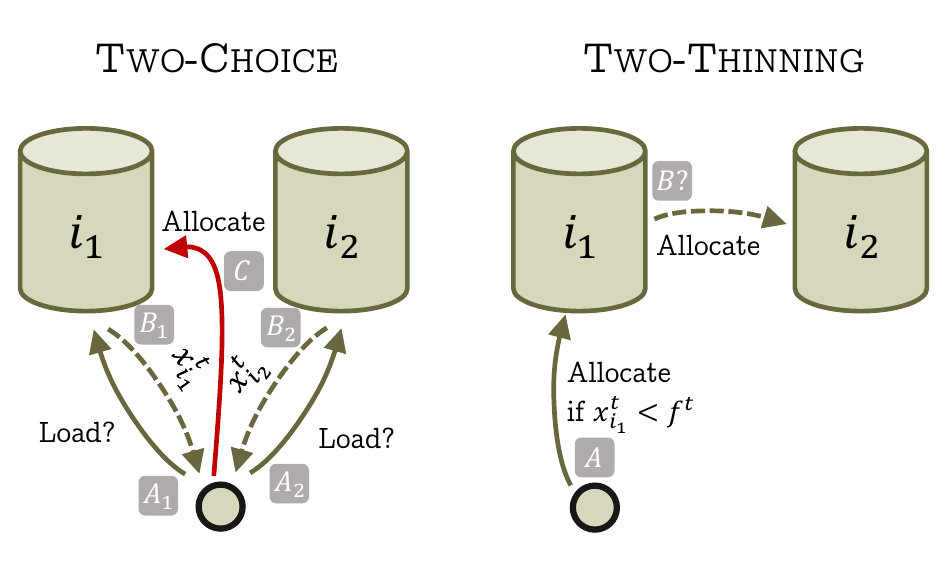}
  \caption{(\textbf{left}) The three stages $A, B, C$ for an allocation in round $t$ using \TwoChoice, which assumes that no other ball is being allocated to the two sampled bins (so the two bins are being ``held'' until both loads $x_{i_1}^t$ and $x_{i_2}^t$ are reported and the allocation is complete). (\textbf{right}) The (at most) two stages $A,B$ for an allocation using \TwoThinning process with threshold $f^t$. Here, none of the bins needs to be ``held'' and there are at most two stages.}
  \label{fig:two_choice_vs_two_thinning}
\end{figure}

\TwoThinning was used by Eager, Lazowska and Zahorjan~\cite{ELZ86} and was analyzed by Iwama and Kawachi~\cite{IK04} and Feldheim and Gurel-Gurevich~\cite{FG18} in the lightly-loaded case ($m = \Oh(n)$). In~\cite{FG18}, the authors proved that for $m=n$, there is a \TwoThinning process using a threshold on the number of times a bin has been used as a first sample, which achieves a gap of $\Oh\big({\scriptstyle \sqrt{ \frac{\log n}{\log \log n}}}\big)$. This process achieves a significant improvement over \OneChoice and also uses a total of $(1+o(1)) \cdot n$ samples, which is an improvement over \TwoChoice. Similar threshold processes have been studied in queuing~\cite{ELZ86} and discrepancy theory~\cite{DFG19}.
For values of $m$ sufficiently larger than $n$, some lower and upper bounds have been established for a more general class of \emph{adaptive} \TwoThinning processes \cite{FGL21,LS22Queries}. Here, adaptive means that the choice of the threshold \emph{may} depend on the entire load vector, which is not ideal from a practical perspective. Related to this line of research, \cite{LS22Queries} also analyzed a so-called \Quantile process, which is a version of \TwoThinning where a ball is allocated to a second sample only if the first bin has a load which is at least the median load.

In \cite{LS22Queries}, an extension of \TwoThinning was studied, the \KThreshold processes, where for each ball two bins are sampled uniformly at random and $k$ queries of the form ``is the load of the bin at least $x$?'' are sent to each. Then, the process allocates to the bin that was witnessed to be lesser loaded. 

Extensions of \TwoThinning where up to $d$ bins may be sampled have been studied in~\cite{LS22Queries,FL20,M99}. A similar class of adaptive sampling processes (where, depending on the loads of the samples so far, the ball may or may not sample another bin) was analyzed by Czumaj and Stemann~\cite{CS01} who proved tight results for $m=n$ and by Berenbrink, Khodamoradi, Sauerwald and Stauffer~\cite{BKSS13} for some specific instances for $m \geq n$.

Another example of a process that reduces the number of samples is the $(1+\beta)$-process introduced by Mitzenmacher~\cite{M99}, where each ball is allocated using \OneChoice with probability $1-\beta$ and otherwise is allocated using \TwoChoice. Peres, Talwar and Wieder~\cite{PTW15} proved that for any $\beta \in (0,1]$, the gap is only $\Oh\big(\frac{\log n}{\beta} \big)$. Hence, only a \textit{small} fraction of \TwoChoice rounds are enough to inherit the property of \TwoChoice that the gap is independent of $m$. It was recently shown that this process asymptotically outperforms \TwoChoice in the presence of outdated information~\cite{LS23HerdPhenomenon}. The $(1+\beta)$-process has been used to analyze \textit{population protocols}~\cite{AAG18,AGR21}, \textit{distributed data structures}~\cite{ABKLN18,AK0N17} and \textit{online carpooling}~\cite{GK0020}. Another sample-efficient variant of \TwoChoice is the $(k, d)\textsc{-Choice}$ process allocates to the $k$ lightest out of the $d$ sampled bins~\cite{G17}.

From a more technical perspective, apart from analyzing a large class of natural allocation processes, an important question is to understand how sensitive the gap is to changes in the probability allocation vector $p^t$ of the process, where $p_i^t$ gives the probability to allocate to the $i$-th heaviest bin in round $t$. To this end,~\cite{PTW15} formulated general conditions on the probability allocation vector, which, when satisfied in all rounds, imply a small gap bound. These were then applied not only to the $(1+\beta)$-process, but also to analyze \textit{graphical balanced allocations} where an edge of a graph is chosen uniformly at random and the ball is allocated to the least loaded of its two endpoints. Other works which study perturbations on the probability allocation vector are: allocations on \textit{hypergraphs}~\cite{G08}, balls-into-bins with \textit{correlated choices}~\cite{W07}; or balls-into-bins with \textit{hash functions}~\cite{CRSW11}.

\paragraph{Our Results} We introduce a general framework that allows us to deduce a small gap independent of $m$ for a large family of processes, but also opens the door for the study of novel allocation processes such as \Twinning (to be defined below). To ensure that an allocation process produces a balanced load vector, we could  bias the allocation towards underloaded bins (bins whose load is below the mean) by skewing the probability by which a bin is chosen for an allocation (\textit{probability bias}), or alternatively, we could add more balls to a bin if it is underloaded (\textit{weight bias}). We call both of these processes \MeanBiased.

Note that a small bias in the probability allocation vector can be achieved in various ways, for example, by taking a {\em second} bin sample: $(i)$ in each round (\TwoChoice), or $(ii)$ in each round with some probability $\beta$ ($(1+\beta)$-process), or $(iii)$ in each round dependent on the load of the first sample (\TwoThinning). Regarding the weight bias, as pointed out in~\cite{MPS02}, allocating consecutive balls to the same bin is very natural in applications like sticky routing~\cite{GKK88} or load balancing, where balls may often arrive in bursts. To this end, we study \Twinning which allocates two balls if the bin is underloaded and one ball if it is overloaded.

We now describe two general conditions, which state that the probability allocation vector is biased towards \textit{underloaded} bins, and we allocate fewer balls to \textit{overloaded} bins. Neglecting some technical details (see \cref{sec:framework}), for $\delta^t \in (0,1)$ being the quantile of the mean load in round $t$, meaning there are $\delta^t n$ overloaded bins, the two conditions are (roughly) as follows: 
\begin{enumerate}\itemsep-4pt
    \item[] \textbf{Condition \PTwo}: At each round, the probability of allocating to any fixed underloaded bin is $\frac{1}{n} \cdot (1 + \Omega(\delta^t))$, while to any fixed overloaded bin is  $\frac{1}{n} \cdot (1 - \Omega(1-\delta^t))$.
    \item[] \textbf{Condition \WTwo}: At each round, if an underloaded bin is chosen for allocation, allocate (a constant number) more balls than if an overloaded bin is chosen.
\end{enumerate}
As our main result we prove a gap bound of $\Oh(\log n)$ whenever a process satisfies \emph{at least one} of these two \textit{mean-biasing} conditions (in addition to some mild and general conditions). It turns out that \PTwo is satisfied by \MeanThinning (i.e., the \TwoThinning process with threshold being the mean load), while \WTwo is satisfied by \Twinning; thus for both of these processes a gap bound of $\Oh(\log n)$ follows immediately. The result for \Twinning gives an example of a process with gap $\Oh(\log n)$, which is also more sample-efficient than \OneChoice. 

The result for \MeanThinning extends to any $\RelativeThreshold(f(n))$ process which uses a threshold of $\frac{t}{n} + f(n)$ for the \textit{offset function} $f(n) \geq 0$, via a coupling argument, proving that \Whp~$\Gap(m) = \Oh(f(n) + \log n)$.
Finally, the framework also applies to the $(1+\beta)$-process with any constant $\beta \in (0, 1)$ and a relaxation of \MeanThinning which we call the $(1+\zeta)$-process.

To the best of our knowledge, most of the related work on balls-into-bins focuses on a small number (usually one or two) allocation processes, and analyzes their gap (with~\cite{PTW15,V99} being exceptions). Here we develop a framework that captures a variety of existing and new processes.

Our analysis, especially of \Twinning, bears some resemblance to the \textit{weighted balls-into-bins setting}~\cite{TW07,PTW15, BFHM08}, but there, the weights of the balls are drawn from some distribution {\em before} the bin is sampled and there is a time-invariant probability bias towards lightly loaded bins (e.g., as in $(1+\beta)$-process and \TwoChoice). However, in our model the weights \textit{depend} on the load of the sampled bin.

The \Twinning process can also be applied to the graphical balanced allocations on $d$-regular graphs for any $d \geq 2$ (see \cref{rem:twinning_on_graphs}). More specifically, the $\Oh(\log n)$ bound from the classical setting (corresponding to the complete graph) still applies even for sparse regular graphs, like the cycle where \TwoChoice is conjectured to have an $\Omega(\poly(n))$ gap~\cite{BF22,ANS22}. 

An extension of \Twinning, called \Packing,  was studied by the authors of this paper in~\cite{LSS22Filling}, where an underloaded bin is ``filled'' until it becomes overloaded, this process also achieves an $\Oh(\log n)$ gap.

\begin{table}[ht]	\resizebox{\textwidth}{!}{
\renewcommand{\arraystretch}{1.5}
		\centering
\begin{tabular}{|c|c|c|c|c|}
\hline 
\multirow{2}{*}{Process} & \multicolumn{2}{c|}{Lightly Loaded Case $m=\Oh(n)$} & \multicolumn{2}{c|}{Heavily Loaded Case $m = \omega(n)$} \\ \cline{2-5}
& Lower Bound & Upper Bound & Lower Bound & Upper Bound
\\ \hline 
$(1+\beta)$-process, const $\beta \in (0,1)$  & \multicolumn{2}{c|}{\phantom{\cite{PTW15}}$\qquad \quad\;\,\frac{\log n}{\log \log n}\qquad\quad\;\,$\cite{PTW15} } & \multicolumn{2}{c|}{\phantom{\cite{PTW15}}$\qquad \quad\;\,\cellcolor{Gray} \log n \qquad\quad\;\,$\cite{PTW15} } \\ 	\hline
\Twinning & \multicolumn{2}{c|}{ \cellcolor{Greenish}$ \frac{\log n}{\log \log n}$} & \multicolumn{2}{c|}{ \cellcolor{Greenish} $ \log n$} \\ \hhline{|-|-|-|-|-|}
\MeanThinning & \multicolumn{2}{c|}{\cellcolor{Greenish} $ \frac{\log n}{\log \log n}$} & \multicolumn{2}{c|}{\cellcolor{Greenish} $ \log n$} \\ \hhline{|-|-|-|-|-|} 
$\RelativeThreshold(f(n))$ &  \phantom{\cite{FL20}} \;\; $\sqrt{\frac{\log n}{\log \log n}}$ \;\; \cite{FL20} & \cellcolor{Greenish} $ \frac{\log n}{\log \log n}$ & \phantom{\;\; [11]} $\frac{\log n}{\log \log n}$ \;\; \cite{LS22Queries} & \cellcolor{Greenish} $f(n) + \log n$ \\ \hhline{|-|-|-|-|-|} 
(adaptive) \TwoThinning & \multicolumn{2}{c|}{ \phantom{[7]}$ \qquad\quad\sqrt{\frac{\log n}{\log \log n}}\qquad\quad$ \cite{FL20}} & \phantom{\;\; \cite{LS22Queries}} $\frac{\log n}{\log \log n}$ \;\; \cite{LS22Queries} & $ \frac{\log n}{\log \log n}$ \cite{FGL21} \\ 
\hline
\end{tabular}}
 
\caption{Overview of the gap achieved by different processes considered in this and previous works. All stated bounds hold asymptotically and with high probability; upper bounds hold for all values of $m$, while lower bounds may only hold for certain values of $m$. Cells shaded in \hlgreenish{\,Green\,} are new results and cells shaded \hlgray{\,Gray\,} are known results we re-prove. The upper bounds for \RelativeThreshold and \Twinning when $m=\Oh(n)$ follow immediately from \OneChoice.
}\label{tab:overview}

\end{table}

	\paragraph{Overview of Proof Techniques}  
 
	Parts of our analysis are based on the use of the exponential potential pioneered in~\cite{PTW15}. However, in~\cite{PTW15} time/load independent conditions on the processes ensure that the potential \textit{drops} in expectation in every round in every round, whereas for our more general conditions there exist load vectors where the potential may \textit{increase} in expectation, even when it is \textit{large} (\cref{clm:bad_configuration_lambda}). This precludes any standard application of the approach in \cite{PTW15} and presents our main challenge. To overcome this challenge we introduce two new ideas/techniques, which since the conference version of this paper \cite{OurSODA} have already found further application \cite{LS23Noise,LS23RBB}. The first is a series of quantitative relationships between the linear, quadratic, and exponential potential functions; this helps overcome the issue of bad configurations ``blocking'' progress. Our second innovation is the phenomenon of ``Mean Quantile Stabilization'' where we show that any round with $\mathcal{O}(n)$ linear potential is followed by a $\Theta(n)$ length interval in which for a constant fraction of rounds the ratio of overload/underload bins is bounded away from $0$ or $1$; this helps maintain progress once it is ``unblocked''.  
	
	 To see in more detail how these two components come together, we observe that in any round where the process has a \textit{stable} quantile, i.e., $\delta^t \in [\eps, 1 - \eps]$ for some constant $\eps \in \big( 0, \frac{1}{2} \big)$, the exponential potential decreases in expectation; and prove that quantile of the mean load \textit{self-stabilizes}. This is done by first proving that the \textit{quadratic potential} drops as long as the \textit{absolute value potential} is $\Omega(n)$ (\cref{lem:quadratic_absolute_relation}). From this, we can deduce that the absolute value potential has to be $\Oh(n)$ at some point. Then we prove that once this happens, the quantile of the mean load will be stable for sufficiently many rounds (\cref{lem:good_quantile}). The final ingredient is to prove that in rounds where the mean quantile is stable, the exponential potential drops by a sufficiently large factor in expectation, while in other rounds it increases by a smaller factor. A more detailed overview is given in \cref{sec:analysis_overview} and the details of the proof in \cref{sec:mean_quantile_stabilization,sec:potential_function_inequalities,sec:mean_biased_gap_completion}.

\paragraph{Organization} 
In \cref{sec:processes}, we define the aforementioned processes  more formally (in addition, an illustration of \Twinning and \MeanThinning can be found in \cref{fig:mean_thinning,fig:twinning}), and also introduce some basic mathematical notation needed for our analysis. In \cref{sec:framework}, we present our general framework for \MeanBiased processes. In this part, we also verify that a variety of processes fall into (or can be reduced to) the framework.  In \cref{sec:analysis_overview}, we outline the proof of the upper bound on the gap for \MeanBiased processes. We also present experiments (\cref{fig:experiments} on page~\pageref{fig:experiments})
illustrating the interplay between the different potential functions used in the analysis.
In \cref{sec:mean_quantile_stabilization}, we introduce the tools to show that for the \MeanBiased processes the mean quantile stabilizes. In \cref{sec:potential_function_inequalities}, we mainly derive the inequalities for the expectation of the exponential potential over one round. In \cref{sec:mean_biased_gap_completion}, we complete the proof for the upper bound on the gap of \MeanBiased processes.
In \cref{sec:lower_bounds} we derive lower bounds on the gap of several \MeanBiased processes, which are tight for \Twinning and \MeanThinning. In \cref{sec:sample_efficiency}, we show that \Twinning is more sample-efficient than \OneChoice and \MeanThinning is more sample-efficient than \TwoChoice. In \cref{sec:experiemental_results}, we empirically compare the gaps of the different processes. In \cref{sec:conclusions}, we conclude with some open problems.

 \section{Notation and Balanced Allocation Processes}\label{sec:processes}
 
We consider processes that sequentially allocate balls into $n$ bins, labeled $[n] := \{ 1, \ldots, n \}$. By $x^t$ we denote the \textit{load vector} of the $n$ bins at round $t=0,1,2,\ldots$ (the state after $t$ allocations have been completed). Initially, we start with the empty load vector $x^0:=(0,\ldots,0)$, unless we explicitly state otherwise. In many parts of the analysis, we will assume a labeling of the $n$ bins so that their loads at round $t$ are ordered non-increasingly, i.e., 
\[
 x_1^t \geq x_2^t \geq \cdots \geq x_n^t.
\]
Unlike many of the standard balls-into-bins processes, some of our processes may allocate more than one ball in a single allocation. To this end, we define $W^t := \sum_{i\in[n]} x_i^t$ as the total number of balls allocated in the first $t$ allocations (for \TwoThinning we allocate one ball per allocation, so $W^t=t$, however for \Twinning $W^t \geq t$). 
We will also use $w^{t+1}:=W^{t+1}-W^{t}$ to denote the number of balls allocated in the $(t+1)$-th allocation. 

We define the \textit{gap} as \[
\Gap(t):=\max_{i \in [n]} x_i^{t} - \frac{W^t}{n},\] which is the difference between the maximum load and mean load\footnote{It is common in the literature to focus on this difference, rather than the difference between maximum and minimum load; however, our results for \MeanBiased processes apply to both.}
at round $t$. When the process $\mathcal{Q}$ is not clear from the context, we write $\Gap_{\mathcal{Q}}(t)$ for the gap and $x_{\mathcal{Q}}^t$ for the load vector of the process.
Finally, for any round $t \geq 0$ we define the \textit{normalized load} of a bin $i \in [n]$ as:
\[
 y_i^{t} :=  x_i^t - \frac{W^t}{n}.
\]
Further, we define $\mathfrak{F}^{t}$ to be the \textit{filtration} corresponding to the first $t$ allocations of the process (so in particular, $\mathfrak{F}^{t}$ reveals $x^{t}$). Following \cite{PTW15}, let $p^t := p^t(\mathfrak{F}^t) \in \R^n$ be the \textit{probability allocation vector} of the process, where for every $i \in [n]$,  $p_i^t$ is the probability that $(t+1)$-th allocation is to $i$-th heaviest bin.  

Further, let $B_{+}^{t}:=\left\{ i \in [n] \colon y_i^{t} \geq 0 \right\}$ be the set of \textit{overloaded} bins, and $B_{-}^{t} := [n] \setminus B_{+}^{t}$ be the set of \textit{underloaded} bins. Let $\delta^t := |B_+^t|/n \in \big\{\frac{1}{n}, \ldots , 1\big\}$ be the \textit{mean quantile} corresponding to the mean load. Let $P_+^t := \sum_{i \in B_+^t} p_i^t$ and $P_-^t := \sum_{i \in B_-^t} p_i^t$, which are the probabilities of allocating to an overloaded or underloaded bin respectively.

For two vectors $p$ and $q$, we say that $p$ \textit{majorizes} $q$ if $
 \sum_{i=1}^{k} p_i \geq \sum_{i=1}^k q_i,
$ for all $k \in [n]$. We denote this by $p\succeq q$. Typically we will be considering the majorization ordering between two (partial) probability allocation vectors, or sorted load vectors. For random variables $Y, Z$ we say that $Y$ is \textit{stochastically smaller} than $Z$ if $ \Pro{ Y \geq x } \leq \Pro{ Z \geq x }$ for all real $x$.

For some specific allocation processes considered in this work, we will also analyze the \textit{sample efficiency} defined as $\eta^m := W^m/S^m$, where $W^m$ is the total number the balls allocated in the first $m$ allocations, and $S^m$ is the total number of uniform samples taken by the process in the first $m$ allocations.

We begin with the definition of the classical \OneChoice process.

\begin{framed}
\vspace{-.45em} \noindent
\underline{\OneChoice Process:} \\
\textsf{Iteration:} For each round $t \geq 0$, sample one bin $i \in [n]$ uniformly at random. Then, update: 
    \begin{equation*}
     x_{i}^{t+1} = x_{i}^{t} + 1.
 \end{equation*}\vspace{-1.5em}
\end{framed}

The \OneChoice process has a \textit{time-independent} probability allocation vector 
\[
 p := p^t = \left( \frac{1}{n}, \ldots, \frac{1}{n} \right),
\]
and it has a sample efficiency of $\eta^m = 1$. 
We proceed with the definition of the \TwoChoice process~\cite{ABKU99,KLM96}.  
\begin{framed}
\vspace{-.45em} \noindent
\underline{\TwoChoice Process:} \\ \textsf{Iteration:} For each round $t \geq 0$, sample two bins $i_1, i_2 \in [n]$, independently and uniformly at random. Let $i \in \{i_1, i_2 \}$ be such that $x_{i}^{t} = \min\{ x_{i_1}^t,x_{i_2}^t\}$, favoring bins with higher indices in case of a tie. Then, update:  
    \begin{equation*}
     x_{i}^{t+1} = x_{i}^{t} + 1.
 \end{equation*}\vspace{-1.5em}
\end{framed}
The \TwoChoice process also has a time-independent probability allocation vector $p := p^t$, where
\[
 p_i := \frac{2i - 1}{n^2}, \quad \text{for any }i \in [n],
\]
and it has a sample efficiency of $\eta^m = \frac{1}{2}$.

\medskip 

\noindent Mixing \OneChoice with \TwoChoice rounds at a rate $\beta$, one obtains the $(1+\beta)$-process~\cite{PTW15}:
\begin{framed}
\vspace{-.45em} \noindent
\underline{$(1+\beta)$-Process:}\\
\textsf{Parameter:} A \textit{mixing factor} $\beta \in (0,1]$.\\
\textsf{Iteration:} For each round $t \geq 0$, with probability $\beta$ allocate one ball via the \TwoChoice process, otherwise allocate one ball via the \OneChoice process. \vspace{-.5em}
\end{framed}
The probability allocation vector for the $(1+\beta)$-process is a convex combination of the \OneChoice and \TwoChoice probability allocation vectors given by
\[
 p_i := (1-\beta) \cdot \frac{1}{n} + \beta \cdot \frac{2i - 1}{n^2}, \quad \text{for any }i \in [n].
\]

The following process has been studied by several authors~\cite{IK04,FG18,FL20}. 

\begin{framed}
	\vspace{-.45em} \noindent
	\underline{$\TwoThinning(f(n, x^t))$ Process:}\\
	\textsf{Parameter:} A \textit{threshold function} $f(n, x^t) \geq 0$.\\
	\textsf{Iteration:} For each round $t \geq 0$, sample two uniform bins $i_1$ and $i_2$ independently, and update:  
	\begin{equation*}
		\begin{cases}
			x_{i_1}^{t+1} = x_{i_1}^{t} + 1 & \mbox{if $x_{i_1}^{t} < f(n, x^t)$}, \\
			x_{i_2}^{t+1} = x_{i_2}^{t} + 1 & \mbox{if $x_{i_1}^{t} \geq f(n, x^t)$}.
		\end{cases}
	\end{equation*}\vspace{-1em}
\end{framed}

In this work, we focus on processes that have a threshold that is \textit{relative} to the mean load, i.e., $f(n, x^t) = \frac{t}{n} + f(n)$. These have the attractive property that they can be implemented just having knowledge of the mean load (instead of the entire load vector). The prototypical such process is \MeanThinning (introduced in \cref{sec:probbias}), which has threshold equal to the mean load, i.e., $f(n) = 0$.

The probability allocation vector for \MeanThinning is given by \[
 p^t := \Bigg( \underbrace{\frac{\delta^t}{n}, \ldots, \frac{\delta^t}{n}}_{\delta^t n \text{ bins}}, \underbrace{\frac{1 + \delta^t}{n}, \ldots, \frac{1+\delta^t}{n}}_{(1-\delta^t) \cdot n \text{ bins}}\Bigg).
\] 

Throughout the paper, we often make use of statements and inequalities which hold only for sufficiently large $n$. For simplicity, we do not state this explicitly.

\NewConstant{quad_delta_drop}{c}
\NewConstant{quad_const_add}{c}
\NewConstant{good_quantile_mult}{c}
\NewConstant{bad_quantile_mult}{c}
\NewConstant{poly_n_gap}{c}

\NewConstantWithName{stab_time}{\ensuremath{c_s}}
\NewConstantWithName{rec_time}{\ensuremath{c_r}}
\NewConstantWithName{small_delta}{C}

\NewConstantWithName{lambda_bound}{c}

\section{\textnormal{\MeanBiased} Processes and Our Results} \label{sec:framework}

In this section, we begin by defining the conditions for \MeanBiased processes. Then, we present our main result for \MeanBiased processes. 

Following this we define some typical \MeanBiased processes, including \Twinning, \MeanThinning and the $(1+\zeta)$-process, and show that they fall under the probability and weight biased subclasses. We summarize our lower bounds and upper bounds for these processes, showing that these are tight, and also consider the sample efficiency of these processes. We also consider the $\RelativeThreshold(f(n))$ processes with $f(n) > 0$, which although are not \MeanBiased processes strictly speaking, we show that our general bound can be extended via a coupling to give tight gap bounds for any $f(n) \geq \log n$.

\subsection{General \MeanBiased Framework and Bounds}

First, we will start by defining some conditions on the probability allocation vector and on the weights of the balls allocated to underloaded and overloaded bins. Recall that for any $i \in [n]$,  $p_{i}^t$ is the probability of allocating to the $i$-th heaviest bin at round $t$. 
As before, $p^t$ may depend on the filtration $\mathfrak{F}^{t}$. Recall also that $B_+^t:=\{i\in [n] : y_-^t\geq 0 \} $, $\delta^t := |B_+^t|/n$,  and $B_-^t:= [n] \setminus B_+^t $.

\NewConstant{p1k1}{k} 
\NewConstant{p1k2}{k} 
\begin{itemize}\itemsep0pt
 	\item[] \textbf{Condition \hypertarget{p1}{$\mathcal{P}_1$}}: There exists an integer  constant $\C{p1k1}\geq 1$ and a constant $\C{p1k2} \in (0,1]$,  such that for each round $t \geq 0$, 
		\begin{itemize}
			\item  the overloaded bins satisfy 
			\begin{alignat*}{3}
				\max_{i \in B_{+}^t} p_i^t &\leq \frac{\C{p1k1}}{n}\qquad &&\text{and}\qquad \left( p_1^t, \ldots, p_{ \delta^t n}^t \right)   &&\preceq \left( \frac{1}{n}, \ldots, \frac{1}{n} \right), 
			\end{alignat*}
			\item and the underloaded bins satisfy \begin{alignat*}{3}
				\min_{ i \in  B_-^t} p_i^t &\geq \frac{\C{p1k2}}{n}\qquad &&\text{and}\qquad
				\left( p_{n}^t, \ldots, p_{\delta^tn + 1}^t \right) && \succeq \left( \frac{1}{n}, \ldots , \frac{1}{n} \right).
			\end{alignat*}\end{itemize}

\item[] \textbf{Condition \hypertarget{w1}{$\mathcal{W}_1$}}: For each round $t \geq 0$, when bin $i \in [n]$ is chosen for allocation,
\begin{itemize}
   \item if  $y_i^{t} < 0$, then allocate $w_{-}$ balls to bin $i$,
   \item if  $y_i^{t} \geq 0$, then allocate $w_{+}$ balls to bin $i$,
    \end{itemize}
   where $1 \leq w_{+} \leq w_{-}$ are constant integers\footnote{By scaling, this can be generalized to constant rational weights.}, i.e., independent of $t$ and $n$.
\end{itemize}

Both conditions \POne and \WOne are natural, but on their own they are not sufficient to establish a good bound on the gap, as the \OneChoice process satisfies both conditions with $\C{p1k1} := \C{p1k2} := 1$. Thus, we will require that processes satisfy at least one of the following two stronger versions of \POne and \WOne:

\NewConstant{p2k1}{k}
\NewConstant{p2k2}{k}
\begin{itemize}\itemsep0pt
    \item[]\textbf{Condition \hypertarget{p2}{$\mathcal{P}_2$}}:   There exists an integer  constant $\C{p1k1}\geq 1$ and constants $\C{p1k2}, \C{p2k1}, \C{p2k2} \in (0,1]$,  such that for each round $t \geq 0$, 
		\begin{itemize}
			\item  the overloaded bins satisfy 
			\begin{alignat*}{3}
				\max_{i \in B_{+}^t} p_i^t &\leq \frac{\C{p1k1}}{n} &&\quad\text{ and }\quad\left( p_1^t, \ldots, p_{ \delta^t n}^t \right)   && \preceq \left( \frac{1}{n} - \frac{\C{p2k1} \cdot (1-\delta^t)}{n}, \ldots, \frac{1}{n} - \frac{\C{p2k1} \cdot (1-\delta^t)}{n} \right),
			\end{alignat*}
			\item and the underloaded bins satisfy \begin{alignat*}{3}
				\min_{ i \in  B_-^t} p_i^t &\geq \frac{\C{p1k2}}{n} &&\quad\text{ and }\quad
    \left( p_{n}^t, \ldots, p_{\delta^tn + 1}^t \right)
\succeq     \left( \frac{1}{n} + \frac{\C{p2k2} \cdot \delta^t}{n}, \ldots , \frac{1}{n} + \frac{\C{p2k2} \cdot \delta^t}{n} \right).	\end{alignat*}\end{itemize}
    
    \item[] \textbf{Condition \hypertarget{w2}{$\mathcal{W}_2$}}: This is as Condition \WOne, but additionally we have the strict inequality: $w_{+} < w_{-}$. Also, we assume that for each round $t \geq 0$, 
    \[
    \left( p_{n}^t, \ldots, p_{\delta^t n + 1}^t \right) \succeq \left( \frac{P_-^t}{|B_-^t|}, \ldots , \frac{P_-^t}{|B_-^t|} \right), 
    \] where $P_-^t := \sum_{i \in B_-^t} p_i^t$ is the sum of allocation probabilities over the underloaded bins. 
\end{itemize}

We refer to a balanced allocation process as \MeanBiased if it satisfies $\PTwo \cap \WOne$  or $\POne \cap \WTwo$, that is it has a probability or weight bias respectively. Our main result proves a logarithmic upper bound on the gap for any \MeanBiased process in the heavily loaded case.

\def\maintechnical{
	Consider any $\PTwo \cap \WOne$-process or $\POne \cap \WTwo$-process. Then, there exists a constant $\kappa > 0$ such that for any round $m \geq 0$,
	\[
	\Pro{ \max_{i \in [n]} \left| x_i^m - \frac{W^m}{n} \right| \leq \kappa \log n } \geq 1-n^{-4};
	\]
	so in particular, $\Pro{ \Gap(m) \leq \kappa \log n } \geq 1-n^{-4}$. }

\begin{thm}[Main Theorem] \label{thm:main_technical} 
	\maintechnical
\end{thm}

To prove a matching lower bound we add the following condition: 
 
\NewConstant{p4k}{k}
\begin{itemize}
	\item \textbf{Condition \hypertarget{p3}{$\mathcal{P}_3$}}: For any $\eps \in \big(0, \frac{1}{2} \big)$ there exists a $\C{p4k} := \C{p4k}(\eps) \in (0, 1]$ such that for each round $t \geq 0$ with $\delta^t \in [\eps, 1 - \eps]$, the probability allocation vector $p^t$ satisfies, 
	\[ \min_{i \in [n]} p_i^t \geq \frac{\C{p4k}}{n}.\] 
\end{itemize}
This condition is satisfied for many natural processes, including \Twinning, \MeanThinning, and the $(1+\beta)$-process with constant $\beta \in (0, 1)$. Assuming condition \PThree in addition to the \MeanBiased conditions allows us to prove a tight lower bound.
	\def\mainlower{	Consider any $\PTwo \cap \WOne$-process or $\POne \cap \WTwo$-process, which also satisfies condition \PThree. Then there exists a constant $\kappa>0$ such that
		\[
		\Pro{\Gap(\kappa n \log n) \geq \kappa \log n} \geq 1 - n^{-1/2}.
		\]}
 
\begin{thm}\label{thm:main_lower} 
\mainlower
\end{thm}

In the next subsection we give some examples of processes for which our bounds hold, and some processes which are not \MeanBiased but are sufficiently close that we can cover them with our bounds. First we discuss the meaning behind the \MeanBiased conditions.

 The rationale behind condition \PTwo is that we wish to slightly bias the probability allocation vector $p^t$ towards underloaded bins at each round $t$. However, it is natural to assume that this influence is limited by a process that samples, say, at most two bins independently and uniformly, and then allocates balls to the least loaded of the two. Concretely, if a process takes \emph{two} independent and uniform bin samples at each round, the probability of picking two overloaded bins equals
$
\bigl( \frac{|B_+^t|}{n} \bigr)^2.
$
Hence by averaging, for each overloaded bin $i \in B_+^t$
\[
p_i^t \geq \left( \frac{|B_+^t|}{n} \right)^2 \cdot \frac{1}{|B_+^t|} = \frac{|B_+^t|}{n^2} = \frac{\delta^t}{n}.
\]
The relaxation of the first constraint in \PTwo by taking a strict convex combination of $\frac{1}{n}$ and $\frac{\delta^t}{n}$
ensures some slack, for instance, it allows the framework to cover the $(1+\zeta)$-process, where we perform a \OneChoice allocation with some constant probability $\eta \in (0, 1]$, and otherwise we perform an allocation following the \MeanThinning process (see \cref{lem:noisy_threshold} for details). 
Similarly, for any process which takes at most two uniform samples, by averaging for any underloaded bin $i \in B_{-}^t$
\begin{align*}
	p_{i}^t \leq \frac{1 - \frac{|B_+^t|^2}{n^2}}{|B_{-}^t|}
	= \frac{(n - |B_+^t|) \cdot (n + |B_+^t|)}{n^2|B_{-}^t|}
	= \frac{1}{n} +  \frac{|B_{+}^t|}{n^2} = \frac{1}{n} + \frac{\delta^t}{n}.
\end{align*}

Finally, we remark that \PTwo resembles the framework of \cite[Equation~2]{PTW15}, where
$p^t_i=p_i$ is non-decreasing in $i$ and $p_{n/3} \leq \frac{1-4\eps}{n}$ and $p_{2n/3} \geq \frac{1+4\eps}{n}$ holds for some $0 < \eps < 1/4$ (not necessarily constant). In contrast to that, for constants $\C{p2k1},\C{p2k2}>0$, the conditions in \PTwo are relaxed as they only imply such a bias if $\delta^t$ is bounded away from $0$ and $1$, which may not hold in all rounds. 

We note that the requirements on the probability allocation vectors in \POne and \PTwo are a relaxation of those in the conference version of this paper~\cite{OurSODA}. Previously we assumed that the probability allocation vector is non-decreasing, which implies the current majorization condition. Also more restrictive bounds of $\max_{i \in B_+^t} p_i^t \leq \frac{1}{n}$ and $\min_{i \in B_-^t} p_i^t \geq \frac{1}{n}$ were assumed. As we shall see this relaxation allows our new framework to include the $(1+\beta)$-process with constant $\beta \in (0, 1)$ (\cref{lem:opb_satisfies_p2_w1}) and the $\KRelativeThreshold(f_1, \ldots, f_k)$ processes with $f_k = 0$ (\cref{lem:k_relative_threshold_satisfies_w1_p2}). Further, the bound on the maximum probability to allocate to any single bin is satisfied by any process that takes $d$ uniform and independent samples in each round with $\C{p1k1} := d$. For instance, \TwoChoice, the $(1+\beta)$-process and \TwoThinning processes, all satisfy this with $\C{p1k1} := 2$.

\subsection{\textnormal{\MeanBiased} Processes Satisfying \texorpdfstring{$\mathcal{P}_1$}{P₁} and \texorpdfstring{$\mathcal{W}_2$}{W₂} (Weight Bias)} 

We begin with an example of a natural process which satisfies conditions $\POne$ and $\WTwo$ of the framework, that is this process allocates more balls when it samples an underloaded bin.  Specifically, we introduce a new process called \Twinning, which allocates \textit{one} ball to an overloaded bin and \textit{two} balls to an underloaded bin. This process is possibly the most similar to \OneChoice among our processes, however, we still prove that \Whp~it has an $\Oh(\log n)$ gap and that it allocates $1+\eps$ balls per sample for constant $\eps \in (0,1)$.

\begin{framed}
	\vspace{-.45em} \noindent
	\underline{\Twinning Process:}\\
	\textsf{Iteration:} For each round $t \geq 0$, sample a bin $i$ uniformly at random, and update:
	\begin{equation*} x_{i}^{t+1} =
		\begin{cases}
			x_{i}^{t} + 2 & \mbox{if $x_{i}^{t} <  \frac{W^{t}}{n}$}, \\
			x_{i}^{t} + 1 & \mbox{if $x_{i}^{t} \geq  \frac{W^{t}}{n}$}.
		\end{cases}
	\end{equation*}\vspace{-1.5em}
\end{framed}

\begin{figure}[H]
	\centering
	\includegraphics[scale=0.45]{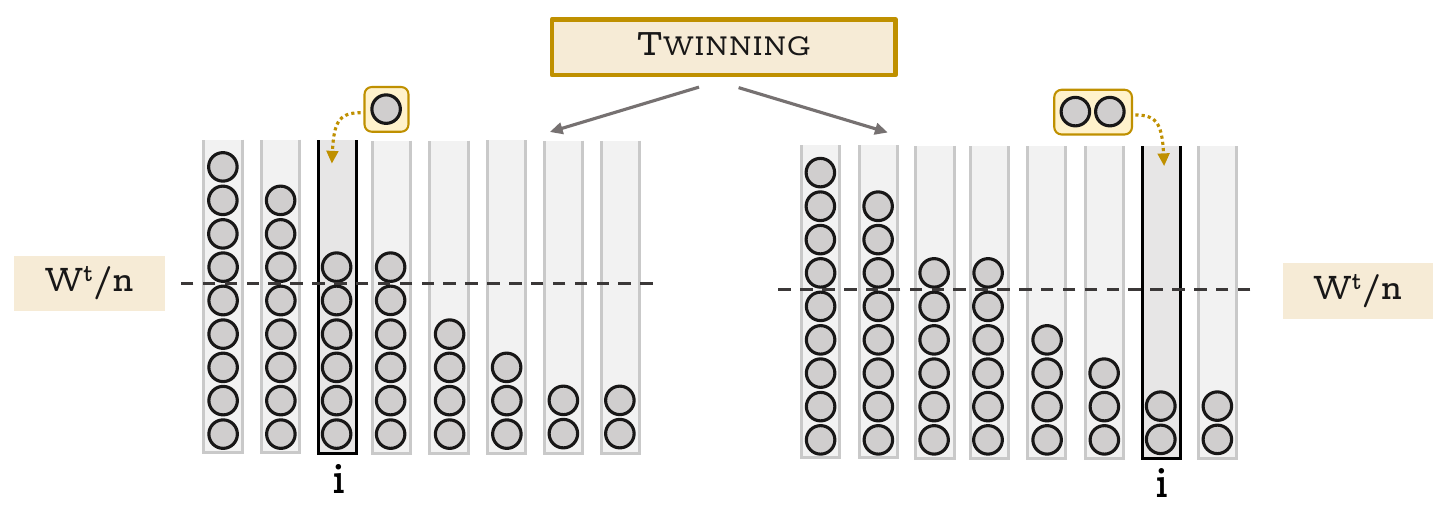}
 \vspace{-0.5cm}
	\caption{Two possibilities for \Twinning: $(i)$~allocate \textit{one} ball to an overloaded bin, or $(ii)$~allocate \textit{two} balls to an underloaded bin.}
	\label{fig:twinning}
\end{figure}

It is fairly clear that \Twinning fits into the \MeanBiased framework.

\begin{lem} %
The \Twinning process satisfies \POne and \WTwo. 
\end{lem}
\begin{proof}
The \Twinning process trivially satisfies \POne with $\C{p1k1} ,\C{p1k2} := 1$ since it samples a bin uniformly at random in each round, and \WTwo since $w_- := 2 > 1 =: w_+$.
\end{proof}

In \cref{sec:sample_efficiency}, we show that \Twinning is more sample efficient than the \OneChoice process.

\newcommand{\TwinningSampleEfficiency}{
There exists a constant $\rho> 0$, such that for any $m \geq \rho^{-1}\cdot  n \log n$, the \Twinning process satisfies
\[
\Ex{\eta^m} \geq 1 + \rho
\quad 
\text{ and }
\quad 
\Pro{\eta^m \geq 1 + \rho} \geq 1 - n^{-1/4}.
\]}

\begin{lem} \label{lem:twinning_sample_efficient}
\TwinningSampleEfficiency
\end{lem}

\begin{rem} \label{rem:twinning_on_graphs}
	The \Twinning process can be implemented in graphical balanced allocations on any $d$-regular graph (even on the cycle) in the following manner: $(i)$~sample an edge randomly, $(ii)$~sample one of its endpoints randomly and $(iii)$~allocate using \Twinning. The first two steps are equivalent to sampling a bin uniformly at random, so the process is equivalent to \Twinning. As we show in \cref{cor:meanthinning_twinning_upper_bound}, \Twinning achieves \Whp~an $\Oh(\log n)$ gap, which is much better than the observed $\poly(n)$ gap of \TwoChoice  (e.g., on the cycle)~\cite{ANS22,BF22}.
\end{rem}

\subsection{\textnormal{\MeanBiased} Processes Satisfying \texorpdfstring{$\mathcal{P}_2$}{P₂} and \texorpdfstring{$\mathcal{W}_1$}{W₁} (Probability Bias)} \label{sec:probbias}

We now turn our attention to introducing some processes which satisfy the conditions $\PTwo$ and $\WOne$ of the framework, that is they have a bias towards sampling underloaded bins.

We begin with a special case of \TwoThinning:

\begin{framed}
	\vspace{-.45em} \noindent
	\underline{\MeanThinning Process:}\\
	\textsf{Iteration:} For each round $t \geq 0$, sample two bins $i_1$ and $i_2$ independently and uniformly at random, and update:  
	\begin{equation*}
		\begin{cases}
			x_{i_1}^{t+1} = x_{i_1}^{t} + 1 & \mbox{if $x_{i_1}^{t} <  \frac{t}{n}$}, \\
			x_{i_2}^{t+1} = x_{i_2}^{t} + 1 & \mbox{if $x_{i_1}^{t} \geq  \frac{t}{n}$}.
		\end{cases}
	\end{equation*}\vspace{-1em}
\end{framed}

\begin{figure}[H]
	\centering
	\includegraphics[scale=0.45]{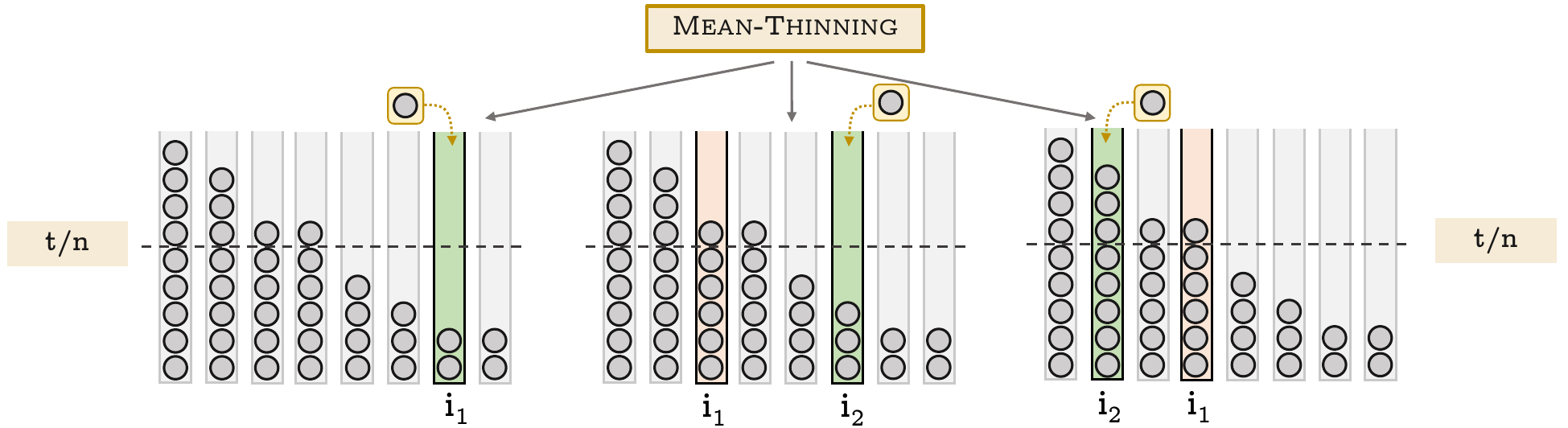}
	\caption{Three different possibilities for \MeanThinning: $(i)$~first bin is underloaded, so the ball is allocated there, $(ii)$~first bin is overloaded and the ball is allocated to an underloaded bin, or $(iii)$ the ball is allocated to an overloaded bin.}
	\label{fig:mean_thinning}
\end{figure}

\begin{lem}
	The \MeanThinning process satisfies \PTwo and \WOne.
\end{lem}
\begin{proof}
	The \MeanThinning process trivially satisfies \WOne with $w_-,w_+ := 1$. To verify \PTwo, recall that a ball is allocated to the first sample, if the bin is underloaded; otherwise, the ball is allocated to the second sample. Hence, for any overloaded bin $i \in B_+^t$.
	\[
	p_i^t = \delta^t \cdot \frac{1}{n} = \frac{1}{n} - \frac{1 \cdot (1 - \delta^t)}{n},
	\]
	that is we first sample an overloaded bin (with prob.\ $\delta^t$) and then sample the specific bin $i$ (with prob.\ $\frac{1}{n}$). In particular, $\max_{i \in B_+^t} p_i^t \leq \frac{1}{n}$. Similarly, any underloaded bin $i \in B_-^t$,
	\[
	p_i^t = \frac{1+\delta^t}{n} = \frac{1}{n} + \frac{1 \cdot \delta^t}{n},
	\]
    thus, $\min_{i \in B_-^t} p_i^t \geq \frac{1}{n}$. Therefore, \MeanThinning satisfies \PTwo with $\C{p1k1},\C{p1k2}, \C{p2k1},  \C{p2k2} := 1$.
\end{proof}

\newcommand{\MeanThinningEfficiency}{
There exists a constant $\rho> 0$, such that for any $m \geq \rho^{-1}\cdot  n \log n$, the \MeanThinning process satisfies
\[
\Ex{\eta^m} \geq \frac{1}{2} + \rho
\quad 
\text{ and }
\quad 
\Pro{\eta^m \geq \frac{1}{2} + \rho} \geq 1 - n^{-1/4}.
\]}
\begin{lem} \label{lem:mean_thinning_efficiency}
\MeanThinningEfficiency
\end{lem}

We now turn to the $(1+\beta)$-process.
\begin{lem} \label{lem:opb_satisfies_p2_w1}
 The $(1+\beta)$-process with any constant $\beta \in (0, 1)$ satisfies \PTwo and \WOne.
\end{lem}
\begin{proof}
The $(1+\beta)$-process trivially satisfies \WOne with $w_-, w_+ := 1$. To verify \PTwo, recall that 
\[
 p_i := (1-\beta) \cdot \frac{1}{n} + \beta \cdot \frac{2i-1}{n^2}.
\]
Hence, $\max_{i \in [n]} p_i \leq \frac{2}{n}$ and $\min_{i \in [n]} p_i \geq \frac{1-\beta}{n}$, and so $\C{p1k1} := 2$ and $\C{p1k2} := 1-\beta$. Now, for the majorization condition, note that the prefix sum is given by
\begin{align*}
 \sum_{i=1}^{\delta^t n} p_i &=
 \delta^t n \cdot (1-\beta) \cdot \frac{1}{n} + \beta \cdot \frac{(\delta^t n)^2}{n^2} 
 = \delta^t - \beta \cdot \delta^t \cdot ( 1 - \delta^t).
\end{align*}
This is the same prefix sum as for the uniform vector which has all entries being equal to $\frac{1}{n} - \beta \cdot (1-\delta^t)$. Since the probability allocation vector of the $(1+\beta)$-process is non-decreasing in $i \in [n]$, it follows that for $\C{p2k1} := \beta$, 
\[
 \left( p_1, \ldots, p_{\delta^t n} \right)  \preceq \left( \frac{1}{n} - \frac{\C{p2k1} \cdot (1-\delta^t)}{n}, \ldots, \frac{1}{n} - \frac{\C{p2k1} \cdot (1-\delta^t)}{n} \right).
\]
The argument for the suffix sum is analogous. We have,
\[
 \sum_{i=\delta^t n + 1}^n p_i = 1 - \delta^t + \beta \cdot \delta^t \cdot (1-\delta^t).
\]
This is the same suffix sum as for the uniform vector $\frac{1}{n} + \beta \cdot \delta^t$. Since the probability allocation vector $p_i$ is non-decreasing in $i \in [n]$, it follows that with $\C{p2k2} := \beta$,
\[
 \left( p_{n}, \ldots, p_{\delta^t n + 1} \right)   \succeq \left( \frac{1}{n} + \frac{\C{p2k2} \cdot \delta^t}{n}, \ldots , \frac{1}{n} + \frac{\C{p2k2} \cdot \delta^t}{n} \right). \qedhere
\]
\end{proof}
For the $(1+\beta)$-process with $\beta = 1-o(1)$ (which includes \TwoChoice), our framework implies an $\Oh(\log n)$ bound on the gap by majorization \cite[Section 3]{PTW15}.

Following the definition of $k$-\textsc{Threshold} processes~\cite{LS22Queries}, we define \KRelativeThreshold processes for $k \geq 1$.

\begin{framed}
	\vspace{-.45em} \noindent
	\underline{$\KRelativeThreshold(f_1(n), \ldots, f_k(n))$ Process:}\\
	\textsf{Parameter:} \textit{Threshold functions} $f_1, \ldots, f_{k}$, such that for every $n \in \N$,
	\[
	\infty = f_0(n) > f_1(n) > f_2(n) > \ldots > f_k(n).
	\]
	\textsf{Iteration:} For each round $t \geq 0$, sample two bins $i_1$ and $i_2$ independently and uniformly at random, and define for $i \in \{i_1, i_2\}$:
	\[
	\ell_i^t := \max \left\{ j \in [k]\cup \{ 0 \} : \frac{t}{n} + f_j(n) > x_{i}^t \right\}.
	\]
	Then, update
	\begin{equation*}
		\begin{cases}
			x_{i_1}^{t+1} = x_{i_1}^{t} + 1 & \mbox{if $\ell_{i_1}^t \geq \ell_{i_2}^t$}, \\
			x_{i_2}^{t+1} = x_{i_2}^{t} + 1 & \mbox{otherwise}.
		\end{cases}
	\end{equation*}\vspace{-0.8em}
\end{framed}
We now show that a special case of this process where the lowest threshold is equal to the mean load satisfies our \MeanBiased conditions. The \MeanBiased framework implies an $\Oh(\log n)$ bound on the gap and this may constitute a starting point for proving $o(\log n)$ bounds on the gap for $k \geq 2$.

\begin{lem} \label{lem:k_relative_threshold_satisfies_w1_p2}
	For any $k \geq 1$, any $\KRelativeThreshold(f_1, \ldots, f_k)$ process with $f_k = 0$ satisfies \PTwo and \WOne.
\end{lem} 
\begin{proof}
	The process trivially satisfies \WOne with $w_- ,w_+ := 1$.
	To verify \PTwo, let $\delta_{f_1}^t, \ldots, \delta_{f_k}^t$ be the quantiles of the thresholds $\frac{t}{n} + f_1, \ldots , \frac{t}{n} + f_k$. Then, the probability allocation vector $p^t$ of the process is given by\[
	p_i^t = \begin{cases}
		\frac{\delta_{f_1}^t}{n} & \text{if }i \leq  \delta_{f_1}^t n, \\
		\frac{\delta_{f_1}^t + \delta_{f_2}^t}{n} & \text{if }\delta_{f_1}^t n < i \leq \delta_{f_2}^t n, \\
		\quad \vdots & \\
		\frac{ \delta_{f_k}^t+1}{n} & \text{if }\delta_{f_k}^t n < i\leq n.
	\end{cases}
	\]
	Threshold $f_j$ affects the probability of bins with normalized load in $[f_{j-1}, f_{j+1}]$. In particular it ``moves'' probability from bins with normalized load in $[f_{j}, f_{j+1}]$ to bins with normalized load in $[f_{j-1}, f_j]$. Letting $p_{\mathcal{Q}_1}^t, \ldots , p_{\mathcal{Q}_{k}}^t$, be the probability allocation vectors corresponding to the processes $\mathcal{Q}_j = j\text{-}\RelativeThreshold(f_{k-j}, \ldots, f_k)$ for $j = 1, \ldots, k$, we have that \[
	p_{\mathcal{Q}_1}^t \preceq \ldots \preceq p_{\mathcal{Q}_{k}}^t,
	\]
	with $\mathcal{Q}_1$ corresponding to \MeanThinning and $\mathcal{Q}_{k}$ corresponding to the process in the statement. Since \MeanThinning satisfies the majorization conditions of \PTwo with $\C{p2k1} ,\C{p2k2} := 1$, by transitivity of majorisation so does $\mathcal{Q}_{k}$. Finally, $\max_{i \in [n]} p_i^t \leq \frac{2}{n}$ and $\min_{i \in B_-^t} p_i^t \geq \frac{1}{n}$ and so $\C{p1k1} := 2$ and $\C{p1k2} := 1$.
\end{proof}

Finally, we study a \textit{noisy} version of the \MeanThinning process, the $(1+\zeta)$-process.

\begin{framed}
	\vspace{-.45em} \noindent
	\underline{$(1+\zeta)$-Process:}\\
	\textsf{Parameter:} A \textit{mixing factor} $\zeta \in (0,1]$.\\
	\textsf{Iteration:} For each round $t \geq 0$, with probability $\zeta$ allocate one ball via the \MeanThinning process, otherwise allocate one ball via the \OneChoice process. \vspace{-.5em}
\end{framed}

\begin{lem}\label{lem:noisy_threshold}
For any constant $\zeta \in (0, 1]$, the $(1+\zeta)$-process satisfies \PTwo and \WOne.
\end{lem}
\begin{proof}
The $(1+\zeta)$-process trivially satisfies \WOne with $w_-,w_+ := 1$. Let $p^t$ be the probability allocation vector of the process. For any overloaded bin $i \in B_{+}^t$,
\begin{align*}
 p_{i}^t &= (1 - \zeta) \cdot \frac{1}{n} + \zeta \cdot \frac{\delta^t}{n}
 = \frac{1}{n} - \frac{\zeta \cdot (1-\delta^t)}{n}. 
\end{align*}
Similarly, for any underloaded bin $i \in B_{-}^t$,
\begin{align*}
 p_{i}^t &= (1 - \zeta) \cdot \frac{1}{n} + \zeta \cdot \frac{1+\delta^t}{n}
 = \frac{1}{n} + \frac{\zeta \cdot \delta^t}{n}.
\end{align*}
Thus, $\max_{i \in [n]} p_i^t \leq \frac{2}{n}$ and $\min_{i \in B_-^t} p_i^t \geq \frac{1}{n}$. So, we conclude that the $(1+\zeta)$-process satisfies \PTwo with $\C{p1k1} := 2$, $\C{p1k2} := 1$  and $\C{p2k1} ,\C{p2k2} := \zeta \in (0, 1]$.
\end{proof}
Having verified that the above processes are \MeanBiased, by \cref{thm:main_technical} we obtain the following upper bound on their gaps.
\begin{cor} \label{cor:meanthinning_twinning_upper_bound}
Consider any of the \Twinning, $\KRelativeThreshold(f_1, \ldots , f_k)$ for any $k \geq 1$ and $f_k = 0$ (including \MeanThinning) and the $(1+\beta)$-process with any constant $\beta \in (0,1]$. Then, there exists a constant $\kappa > 0$ (different for each process) such that for any round $m \geq 0$, \[
	\Pro{ \Gap(m) \leq \kappa \log n } \geq 1-n^{-4}.
\]
\end{cor}

\subsection{\texorpdfstring{$\textnormal{\RelativeThreshold}(f(n))$}{Relative-Threshold(f(n))} with \texorpdfstring{$f(n) \geq 0$}{f(n) ≥ 0}}

In the previous section, we showed that a special case of $\KRelativeThreshold(f_1, \ldots, f_k)$ process with $f_k=0$ is \MeanBiased. We now turn our attention to instances with $k=1$ but the offset function $f_1(n) =f(n)$ is arbitrary non-negative. This generalizes \MeanThinning, and for $f(n) > 0$ the process is not \MeanBiased as there are load configurations where the processes may not bias allocations away from all overloaded bins.
\begin{framed}
	\vspace{-.45em} \noindent
	\underline{$\RelativeThreshold(f(n))$ Process:}\\
	\textsf{Parameter:} An \textit{offset function} $f(n) \geq 0$.\\
	\textsf{Iteration:} For each round $t \geq 0$, sample two bins $i_1$ and $i_2$ independently and uniformly at random, and update:  
	\begin{equation*}
		\begin{cases}
			x_{i_1}^{t+1} = x_{i_1}^{t} + 1 & \mbox{if $x_{i_1}^{t} <  \frac{t}{n}+f(n)$}, \\
			x_{i_2}^{t+1} = x_{i_2}^{t} + 1 & \mbox{if $x_{i_1}^{t} \geq  \frac{t}{n}+f(n)$}.
		\end{cases}
	\end{equation*}\vspace{-1em}
\end{framed}

Similarly to \MeanThinning, the \RelativeThreshold processes are appealing, as they require just knowledge (or estimate) of the mean load, which changes only every $n$ rounds. For this process we prove the following upper bound on the gap.

\begin{thm}	\label{thm:relupper} There exists a constant $\kappa > 0$ such that for any round $m \geq 0$ and function $f(n)\geq 0$ the $\RelativeThreshold(f(n))$ process satisfies \[
	\Pro{ \Gap(m) \leq \kappa \left(\log n + f(n)\right) } \geq 1-n^{-4}.
	\]
\end{thm}

The bound above is tight for any $\RelativeThreshold(f(n))$ process with $f(n) \geq \log n$.

\def\rellower{Consider any $\RelativeThreshold(f(n))$ process with $f(n) \geq \log n$. Then, for $m := \frac{1}{24} \cdot \frac{n \cdot (f(n))^2}{\log n}$,
	\[
	\Pro{ \Gap(m) \geq \frac{f(n)}{50}  } \geq 1 - n^{-1/4}.
	\]}

\begin{lem}\label{lem:rellower}
	\rellower
\end{lem}

For the remaining offset functions $f(n) \in (0, \log n)$, we can apply the general \TwoThinning lower bound of $\Omega\big( \frac{\log n}{\log \log n} \big)$ derived in~\cite{LS22Queries} (and in~\cite{FGL21})  to the $\RelativeThreshold(f(n))$ process. This means that the upper bound from \cref{thm:relupper} is tight within an $\log \log n$ factor. This also highlights the question of determining the behavior of the gap in this regime; we conjecture that the lower bound is tight for some values of $f(n)$. 

We prove \cref{lem:rellower} later in \cref{sec:lower_bounds}, however we now proceed with the proof of \cref{thm:relupper}. We use a coupling argument involving adding $f(n)$ ``extra balls'' to each bin, which allows us to reduce the  $\RelativeThreshold(f(n))$ process with a non-negative offset function $f(n)$, to \MeanThinning.  Thanks to this reduction, proved in \cref{lem:coupling,lem:f(n)extension} we deduce \cref{thm:relupper}.

\begin{lem} \label{lem:f(n)extension}
Consider the \MeanThinning and $\RelativeThreshold(f(n))$ processes for any non-negative $f(n) \geq 0$. Let $\Gap_{0}$ and $\Gap_{f(n)}$ be their gaps respectively. Then, $\Gap_{f(n)}$ is stochastically smaller than $\Gap_{0}+f(n)$.
\end{lem}

Before proving the lemma, we need the following domination result:

\begin{lem}\label{lem:coupling}
Let $\mathcal{R}$ be the $\TwoThinning(\frac{t}{n} + f(n))$ process where $f(n)$ is non-negative, starting with an empty load vector $x_{\mathcal{R}}^0 = (0, \ldots, 0)$. Further, let $\mathcal{Q}$ be the $\TwoThinning(\frac{t}{n} + f(n))$ process with initial load vector $x_{\mathcal{Q}} = (f(n), \ldots , f(n))$. Then, there is a coupling so that at any round $t \geq 0$, it holds that $(x_{\mathcal{R}}^{t})_i \leq (x_{\mathcal{Q}}^{t})_i$, for any bin $i \in [n]$.
\end{lem}
\begin{proof}[Proof of \cref{lem:coupling}]
Let $j_1=j_1^t$ and $j_2=j_2^t$ be the two bins sampled at round $t \geq 0$, which are uniform and independent over $[n]$. We consider a coupling between $\mathcal{R}$ and $\mathcal{Q}$, where these random bin samples are identical, and prove inductively that for any $t \geq 0$ and any $i \in [n]$,
\[
(x_{\mathcal{R}}^{t})_i \leq (x_{\mathcal{Q}}^{t})_i.
\]
The base case $t=0$ holds by definition. For the inductive step, we make a case distinction:

\noindent\textbf{Case 1} [$(x_{\mathcal{R}}^t)_{j_1} < t/n+f(n)$]. In this case, $\mathcal{R}$ allocates a ball to $j_1$. If $(x_{\mathcal{Q}}^t)_{j_1} < t/n+f(n)$, then $\mathcal{Q}$ also allocates a ball to $j_1$; otherwise, we have $(x_{\mathcal{Q}}^t)_{j_1} \geq t/n+f(n)$, and hence $(x_{\mathcal{Q}}^t)_{j_1} > (x_{\mathcal{R}}^t)_{j_1}$, i.e., $(x_{\mathcal{Q}}^t)_{j_1} \geq (x_{\mathcal{R}}^t)_{j_1}+1$. This implies
    \[
    (x_{\mathcal{Q}}^{t+1})_{j_1} = (x_{\mathcal{Q}}^{t})_{j_1} \geq  (x_{\mathcal{R}}^t)_{j_1} + 1 = (x_{\mathcal{R}}^{t+1})_{j_1},
    \]
    and the inductive step follows from this and the induction hypothesis.
    \medskip

\noindent\textbf{Case 2} [$(x_{\mathcal{R}}^t)_{j_1} \geq t/n+f(n)$]. In this case, $\mathcal{R}$ allocates a ball to $j_2$. By induction hypothesis, $(x_{\mathcal{R}}^t)_{j_2} \leq (x_{\mathcal{Q}}^t)_{j_2}$, which implies $\mathcal{Q}$ also allocates a ball to $j_2$. Thus we have   \[
    (x_{\mathcal{Q}}^{t+1})_{j_2} =
    (x_{\mathcal{Q}}^t)_{j_2} + 1 \geq 
    (x_{\mathcal{R}}^{t})_{j_2}+1 = 
    (x_{\mathcal{R}}^{t+1})_{j_2},
    \] and the inductive step is complete. 
    
Since in both cases all other bins remain unchanged the proof is complete. 
\end{proof}

\begin{lem} \label{lem:coupling_two}
Let $\mathcal{R}$ be the $\TwoThinning(\frac{t}{n} + f(n))$ process starting with the initial load vector $(x_{\mathcal{R}}^0)_1 = \ldots = (x_{\mathcal{R}}^0)_n = f(n)$ and $\mathcal{Q}$ be the \MeanThinning process with initial load vector $x_{\mathcal{Q}}^0 = (0, \ldots, 0)$. Then, there is a coupling so that $x_{\mathcal{R}}^t = x_{\mathcal{Q}}^t + f(n)$ for any round $t \geq 0$.
\end{lem}
\begin{proof}
In the execution of the process $\mathcal{R}$, we start the process at round $t = 0$ from an initial load of $f(n)$ balls in each bin. Since the threshold is $t/n + f(n)$, that is the process $\mathcal{R}$ does not have a threshold relative to the actual mean load, the effect of adding these balls is to reduce the threshold of $\mathcal{R}$ by exactly $f(n)$. Thus, $\mathcal{R}$ is operating  with a threshold of $t/n + f(n) - f(n) = t/n$, which is equivalent to the $\MeanThinning$ process, i.e., $\mathcal{Q}$. So, we obtain a coupling such that $x_{\mathcal{R}}^t = x_{\mathcal{Q}}^t + f(n)$ for any round $t \geq 0$.
\end{proof}

We can now complete the proof of \cref{lem:f(n)extension}.

\begin{proof}[Proof of \cref{lem:f(n)extension}]
We define the following processes:
\begin{itemize}[itemsep=0pt]
  \item $\mathcal{R}_1$: The $\TwoThinning(\frac{t}{n} + f(n))$ process (starting from the empty load vector).
  \item $\mathcal{R}_2$: The $\TwoThinning(\frac{t}{n} + f(n))$ starting with $(x_{\mathcal{R}_2}^0)_1 = \ldots = (x_{\mathcal{R}_2}^0)_n = f(n)$.
  \item $\mathcal{R}_3$: The \MeanThinning process (starting from the empty load vector).
\end{itemize}
By \cref{lem:coupling}, there exists a coupling such that $x_{\mathcal{R}_2}^t$ pointwise majorises $x_{\mathcal{R}_1}^t$ for any round $t \geq 0$. By \cref{lem:coupling_two}, there exists a coupling such that $x_{\mathcal{R}_2}^t = x_{\mathcal{R}_3}^t + f(n)$. Hence, we deduce that there is a coupling between the three processes such that\begin{align*}
  \Gap_{f(n)}(t) 
 & = \Gap_{\mathcal{R}_1}(t) 
  = \max_{i \in [n]} (x_{\mathcal{R}_1}^t)_i - \frac{t}{n}
  \leq \max_{i \in [n]} (x_{\mathcal{R}_2}^t)_i - \frac{t}{n}
  = \max_{i \in [n]} (x_{\mathcal{R}_3}^t)_i + f(n) - \frac{t}{n} \\
 & = \Gap_{0}(t) + f(n). \qedhere
\end{align*}
\end{proof}

\section{Overview of the Analysis}\label{sec:analysis_overview}

In this section, we outline the proof for the $\Oh(\log n)$ upper bound on the gap for \MeanBiased processes, i.e., processes satisfying conditions \POne and \WOne and in addition \PTwo or \WTwo. The proof details are given in \cref{sec:mean_quantile_stabilization,sec:potential_function_inequalities,sec:mean_biased_gap_completion}. Some of the properties we prove are useful for the lower bounds (\cref{sec:lower_bounds}) and in analyzing the sample-efficiency of the processes (\cref{sec:sample_efficiency}).

Our analysis is based on the interplay between the following potential functions. %
\begin{itemize}\itemsep0pt
    \item The \emph{absolute value potential} (also known as the \textit{number of holes}~\cite{BCSV06}): 
    \[
    \Delta^{t}:=\sum_{i=1}^{n} |y_i^t|.
    \]
    In \cref{lem:good_quantile}, we prove that when $\Delta^s = \Oh(n)$, then \Whp, among the next $\Theta(n)$ rounds a constant fraction of them are rounds $s$ where the quantile of the mean $\delta^s \in [\eps, 1 - \eps]$, for some constant $\eps \in \big(0,\frac{1}{2}\big)$.
    
    \item The \emph{quadratic potential}: 
    \[
    \Upsilon^t:=\sum_{i=1}^{n} \left( y_i^t \right)^2.
    \]
    In \cref{lem:quadratic_absolute_relation}, we prove that in any round $t$ with $\Delta^t = \Omega(n)$ holds, the potential $\Upsilon^t$ decreases in expectation over the next round. 

    \item The \textit{exponential potential} for a constant smoothing parameter $\alpha > 0$ (specified in \eqref{eq:alpha_def}): 
    \begin{align} \label{eq:lambda_def}
    \Lambda^t := \Lambda^t(\alpha) := \sum_{i=1}^{n} \exp\left( \alpha \cdot |y_i^t | \right) =
    \sum_{i \in B_+^t} \exp\left( \alpha y_i^{t} \right) + \sum_{i \in B_{-}^t} \exp\left( -\alpha (-y_i^{t}) \right).
    \end{align}
    In \cref{cor:change_for_large_lambda}, we prove that when $\Lambda^t = \Omega(n)$, if $\delta^t \in [\eps, 1 - \eps]$, then $\Lambda^t$ decreases in expectation, otherwise it increases by a smaller factor. This potential is similar to the \textit{hyperbolic cosine potential} used in~\cite{PTW15}, but for each bin there is only one term. Our potential is slightly easier to analyze since the bias we study here is based on whether the bin is overloaded or underloaded. Further, unlike the analysis of the $(1+\beta)$-process in~\cite{PTW15}, the potential can increase in expectation over a single round, even when it is large (\cref{clm:bad_configuration_lambda}).
    
    \item A ``weaker'' instance of the $\Lambda$ potential function with smoothing parameter $\tilde{\alpha} = \Theta(1/n)$: 
    \[
    V^t := V^t(\tilde{\alpha}) := \sum_{i=1}^{n} \exp\left( \tilde{\alpha} \cdot |y_i^t | \right) = \sum_{i \in B_+^t} \exp\left( \tilde{\alpha} y_i^t  \right) + \sum_{i \in B_{-}^t} \exp\left( -\tilde{\alpha} (-y_i^t) \right).
    \]
    In \cref{lem:initial_gap_nlogn}, we prove that $V$ drops in expectation in every round (regardless of the value of $\delta^t$), and so we can establish that $\ex{V^t} = \poly(n)$ at an arbitrary round $t$. Then, using Markov's inequality we establish \Whp~that $\Gap(t) = \Oh(n \log n)$. Note that this is similar to the case $\beta = \Theta(1/n)$ in the $(1+\beta)$-process~\cite{PTW15}. We use this bound on the gap as a starting point for the tighter bound.
\end{itemize}

Having defined the potential functions, let us now describe in more detail why the potential $\Lambda$ may increase in expectation in some rounds, for e.g., the \MeanThinning process.
For this process, there exist bad configurations where $(i)$ $\Gap(t)$ is large and $(ii)$ for all $s \in [t,t+\omega(n)]$ rounds the quantile of the mean load satisfies $\delta^{s}=1-o(1)$ (or $\delta^{s}=o(1)$). When $\delta^s$ is too close to $1$ (or $0$), then the bias away from any \emph{fixed} overloaded (or towards any \emph{fixed} underloaded) bin is too small, and the process allocates balls \textit{almost} uniformly, similarly to \OneChoice. As a result, the exponential potential may increase for several rounds, until $\delta^s$ is bounded away from $0$ and $1$ (see \cref{fig:experiments} for an illustration showing experimental results and \cref{clm:bad_configuration_lambda} for a concrete example of a bad configuration).

So, instead we start with the weaker exponential potential function $V := V(\tilde{\alpha})$ where $\tilde{\alpha} = \Theta(1/n)$. Because \MeanBiased processes always have a small bias to allocate away from (or allocate fewer balls to) the maximum load, we are able to prove in \cref{lem:initial_gap_nlogn} that when $V$ is sufficiently large it decreases in expectation. This allows us to prove that $\ex{V^t} = \poly(n)$ at an arbitrary round $t$ and infer, using Markov's inequality, that the gap is \Whp~$\Gap(t) = \Oh(n \log n)$ (\cref{lem:initial_gap_nlogn}). Then, our next goal is to show that starting with $\Gap(t_0) = \Oh(n \log n)$ we reach a round $s \in [t_0, t_0 + n^3 \log^4 n]$ where $\Lambda^s = \Oh(n)$ (the \textit{recovery phase} -- \cref{sec:recovery}) and finally show that the gap remains $\Oh(\log n)$ for the remaining rounds $[s, m]$ (the \textit{stabilization phase} -- \cref{sec:stabilization}).

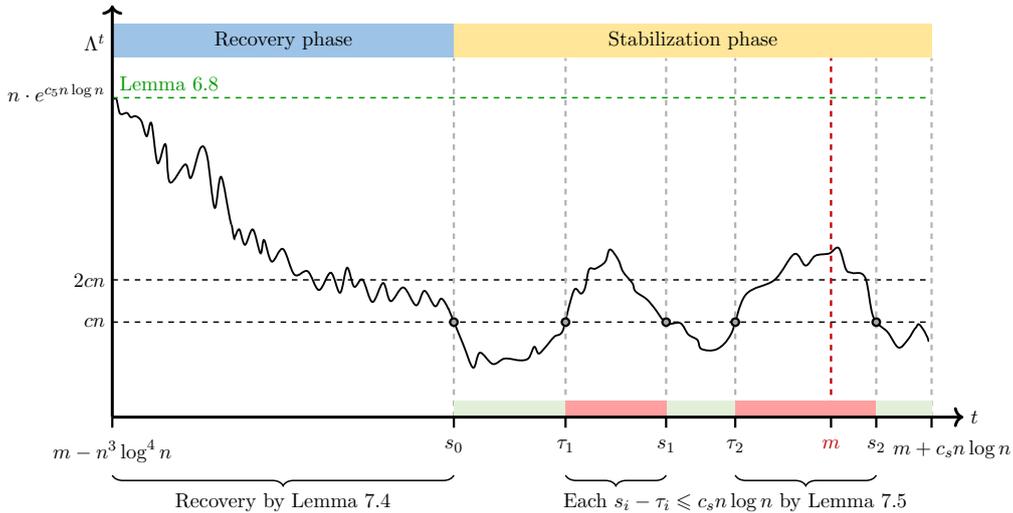
\begin{figure}[h]
    \centering

\scalebox{0.7}{
    \begin{tikzpicture}[
  IntersectionPoint/.style={circle, draw=black, very thick, fill=black!35!white, inner sep=0.05cm}
]

\definecolor{MyBlue}{HTML}{9DC3E6}
\definecolor{MyYellow}{HTML}{FFE699}
\definecolor{MyGreen}{HTML}{E2F0D9}
\definecolor{MyRed}{HTML}{FF9F9F}
\definecolor{MyDarkRed}{HTML}{C00000}

\node[anchor=south west] (plt) at (-0.1, 0) {\includegraphics{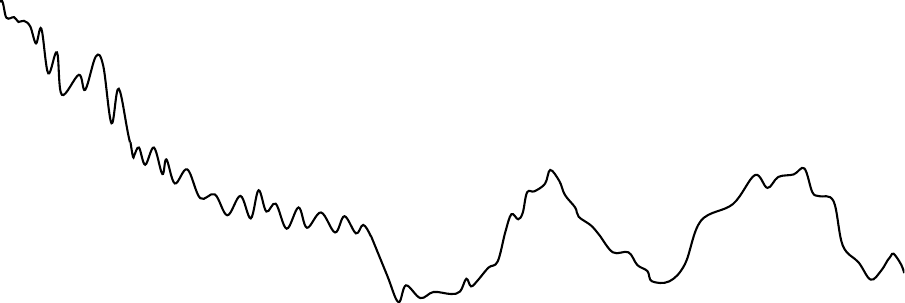}};

\def\xEnd{16}
\def\xLast{15.40}
\def\yLast{6}
\def\cn{1}
\def\yBottom{-0.8}

\node[anchor=south west, inner sep=0.15cm, fill=MyBlue, rectangle,minimum width=6.42cm] at (0, \yLast) {Recovery phase};
\node[anchor=south west, inner sep=0.15cm, fill=MyYellow, rectangle,minimum width=\xLast cm - 6.42cm] at (6.42, \yLast) {Stabilization phase};

\draw[dashed, thick] (0,1) -- (\xLast, 1);
\draw[dashed, thick] (0,1.8) -- (\xLast, 1.8);
\draw[dashed, thick, green!60!black] (0,5.25) -- (\xLast, 5.25);

\node[anchor=south west, green!60!black] at (0,5.25) {\cref{lem:initial_gap_nlogn}};

\node[anchor=east] at (0, \cn) {$cn$};
\node[anchor=east] at (0, 1.8) {$2cn$};
\node[anchor=east] at (0, 5.35) {$n \cdot e^{\C{poly_n_gap} n \log n}$};
\node[anchor=east] at (0, 6.3) {$\Lambda^t$};

\node[anchor=west] at (\xEnd, \yBottom) {$t$};

\def\tA{6.42}
\def\tB{8.52}
\def\tC{10.41}
\def\tD{11.71}
\def\tM{13.51}
\def\tE{14.36}

\newcommand{\drawLine}[3]{
\draw[dashed, very thick, #3] (#1, \yBottom) -- (#1, \yLast);
\draw[very thick] (#1, \yBottom) -- (#1, \yBottom -0.2);
\node[anchor=north] at (#1, \yBottom -0.3) {#2};}

\newcommand{\drawPoint}[3]{
\drawLine{#1}{#2}{#3}
\node[IntersectionPoint] at (#1, \cn) {};}

\draw[very thick] (0, \yBottom) -- (0, \yBottom -0.2);
\node[anchor=north] at (0, \yBottom -0.3) {$m - n^3 \log^4 n$};

\drawPoint{\tA}{$s_0$}{black!30!white};
\drawPoint{\tB}{$\tau_1$}{black!30!white};
\drawPoint{\tC}{$s_1$}{black!30!white};
\drawPoint{\tD}{$\tau_2$}{black!30!white};
\drawPoint{\tE}{$s_2$}{black!30!white};
\drawLine{\tM}{\textcolor{MyDarkRed}{$m$}}{MyDarkRed};
\drawLine{\xLast}{$\qquad m + \C{stab_time} n \log n $}{black!30!white};

\newcommand{\drawRegionRect}[3]{
\node[anchor=south west, rectangle, fill=#3, minimum width=#2 cm- #1 cm, minimum height=0.3cm] at (#1, \yBottom) {};}

\drawRegionRect{\tA}{\tB}{MyGreen}
\drawRegionRect{\tB}{\tC}{MyRed}
\drawRegionRect{\tC}{\tD}{MyGreen}
\drawRegionRect{\tD}{\tE}{MyRed}
\drawRegionRect{\tE}{\xLast}{MyGreen}

\newcommand{\drawBrace}[4]{
\draw [
    thick,
    decoration={
        brace,
        raise=0.5cm,
        amplitude=0.2cm
    },
    decorate
] (#2, \yBottom - 0.6) -- (#1, \yBottom - 0.6) 
node [anchor=north,yshift=-0.7cm,#4] {#3}; }

\drawBrace{0}{\tA}{Recovery by \cref{lem:recovery}}{pos=0.5};
\drawBrace{\tB}{\tC}{}{};
\drawBrace{\tD}{\tE}{Each $s_i - \tau_i \leq \C{stab_time} n \log n$ by \cref{lem:stabilization}}{};

\draw[->, ultra thick] (0,\yBottom) -- (0, 7);
\draw[->, ultra thick] (0,\yBottom) -- (\xEnd, \yBottom);

\end{tikzpicture}}

\caption{Overview of the \textit{recovery phase} and \textit{stabilization phase}. In the recovery phase starting with $V^t = \poly(n)$ (and so $\Gap(t) = \Oh(n \log n)$ and $\Lambda^t \leq e^{\Oh(\log n)}$), after $n^3 \log^4 n$ rounds \Whp~we find a round $s_0$ with $\Lambda^{s_0} \leq cn$. We then switch to the stabilization phase, where in \cref{lem:stabilization} we prove that starting with $\Lambda^{\tau} \leq 2cn$ in the next $\C{stab_time} n \log n$ rounds there is a round $t$ with $\Lambda^{t} \leq cn$, which allows us to infer that $\Gap(m)=\Oh(\log n)$.}
    \label{fig:recovery_stabilisation}
\end{figure}

In both the recovery and the stabilization phase, we will study the evolution of $\delta^t$ and prove that from any load vector, the process eventually reaches a value $\delta^t$ in $[\eps, 1 - \eps]$ for some constant $\eps \in \big(0,\frac{1}{2}\big)$, sufficiently often. The next lemma provides a sufficient condition for this to occur.

\def\goodquantile{
Consider any $\POne \cap \WOne$-process and let $\C{small_delta} \geq 1$ be any integer constant. Then, there exists a constant $\eps:=\eps( \C{small_delta}) \in \big(0,\frac{1}{2}\big)$ where $(4\eps)^{-1}$ is an integer, such that for any round  $t_0\geq 0$,   
\[
 \Pro{ \left| \left\{ t \in \left[t_0, t_0 +   \frac{n}{4\eps}  \right] \colon 
 \delta^{t} \in [\eps, 1 - \eps] \right\} \right| \geq \eps n ~ \Big| ~ \mathfrak{F}^{t_0}, \Delta^{t_0} \leq Cn } \geq 1-e^{-\eps n}.
\]  } 

{\renewcommand{\thelem}{\ref{lem:good_quantile}}
	\begin{lem}[Mean Quantile Stabilization, restated, page~\pageref{lem:good_quantile}]
\goodquantile
	\end{lem} }
	\addtocounter{lem}{-1}

Thus, whenever the absolute value potential $\Delta^{t_1}$ is at most linear, the mean quantile $\delta^t$ is \textit{good} (bounded away from $0$ or $1$) in a constant fraction of the next $\Theta(n)$ rounds.
The next step in the proof is to establish that the sufficient condition on $\Delta^{t_1}$ will be satisfied. To this end, we use a relation between the absolute value potential $\Delta^{t}$ and the quadratic potential $\Upsilon^t$, showing that $\Upsilon^{t}$ drops in expectation over the next round as long as $\Delta^t=\Omega(n)$. Thus, $\Delta^t$ must eventually become linear.

\def\quadraticabsoluterelation{Consider any $\POne \cap \WOne$-process. Then, for any round $t \geq 0$,  \[
\Ex{\left. \Upsilon^{t+1} \,\right|\, \mathfrak{F}^t} \leq \Upsilon^t + \sum_{i \in B_+^t} 2y_i^t p_i^t \cdot w_+ + \sum_{i \in B_-^t} 2y_i^t p_i^t \cdot w_- + 4 \cdot (w_-)^2.
\]
Hence for any $\PTwo \cap \WOne$-process or $\POne \cap \WTwo$-process, this implies by \cref{lem:additive_drift} that there exist  constants $\C{quad_delta_drop}, \C{quad_const_add} > 0$ such that for any round $t \geq 0$,
\[
\Ex{\left. \Upsilon^{t+1} \,\right|\, \mathfrak{F}^t} \leq \Upsilon^t  - \frac{\C{quad_delta_drop}}{n} \cdot \Delta^t + \C{quad_const_add}.
\]}

{\renewcommand{\thelem}{\ref{lem:quadratic_absolute_relation}}
	\begin{lem}[Restated, page~\pageref{lem:quadratic_absolute_relation}]
\quadraticabsoluterelation
	\end{lem} }
	\addtocounter{lem}{-1}

Combining the above two lemmas, we prove that in any sufficiently long interval for a constant fraction rounds the mean quantile $\delta^t$ is good, i.e.,~$\delta^t \in [\eps, 1 - \eps]$. In particular, for the recovery phase, we show this guarantee holds \Whp~for an interval of length $n^3 \log^4 n$ and for the stabilization phase, we prove that it holds \Whp~for an interval of length $\Theta(n \log n)$ given that we start with $\Lambda^s = \Oh(n)$  (\cref{lem:many_good_quantiles_whp}). %
In these good rounds, we prove that the exponential potential $\Lambda^t$ decreases by a multiplicative factor of at least $(1-\C{good_quantile_mult}\alpha / n)$ (see \cref{lem:good_quantile_good_decrease}). In other rounds, the potential $\Lambda^t$ increases by at most $(1+\C{bad_quantile_mult} \alpha^2 /n)$ (see \cref{lem:bad_quantile_increase_bound}). Combining these for sufficiently small $\alpha$, we obtain that the exponential potential function eventually becomes $\Oh(n)$, which implies a logarithmic gap. We refer to \cref{fig:recovery_stabilisation} for a high-level overview of recovery and stabilization, as well as \cref{fig:proof_outline_tikz} for a diagram summarizing most of the crucial lemmas used in the analysis and outlining their relationship.

 \newcommand{\wcp}{w.c.p.\xspace}

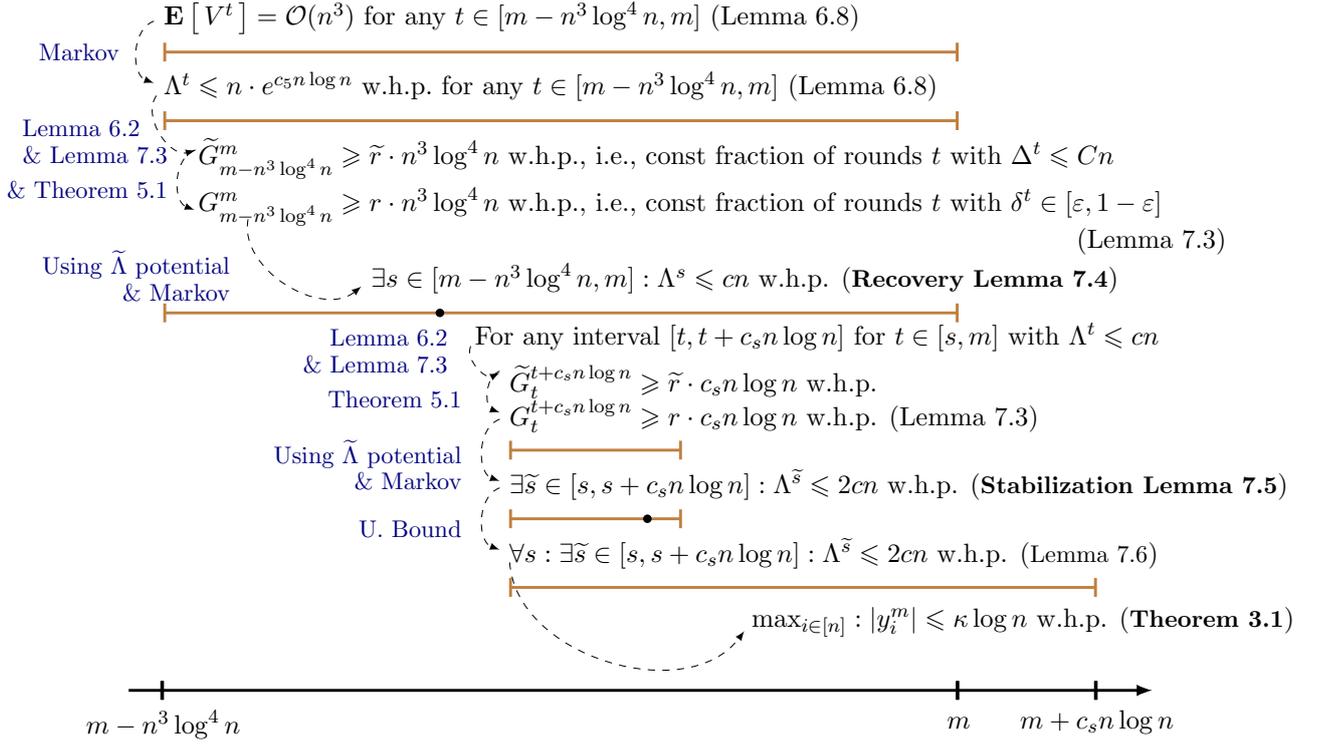
\begin{figure}[H]
\begin{center}
\makebox[\textwidth][c]{
\scalebox{0.91}{
\begin{tikzpicture}[
  txt/.style={anchor=west,inner sep=0pt},
  Dot/.style={circle,fill,inner sep=1.25pt},
  implies/.style={-latex,dashed}]

\def\betaO{0}
\def\betaA{1}
\def\betaB{5}
\def\betaC{12.5}
\def\End{14.5}
\def\yO{-9.5}

\node[anchor=north] at (\betaA, \yO-0.15) {$m - n^3 \log^4 n$};
\draw[|-|, very thick] (\betaA, \yO) -- (\betaA+0.01, \yO);
\node[anchor=north] at (\betaB, \yO) {};
\node[anchor=north] at (\betaC, \yO-0.15) {$\phantom{\log n}m\phantom{\log n}$};
\draw[|-|, very thick] (\betaC, \yO) -- (\betaC+0.01, \yO);
\node[anchor=north] at (\End, \yO-0.15) {$m + c_s n \log n$};
\draw[|-|, very thick] (\End, \yO) -- (\End+0.01, \yO);

\begin{scope}[yshift=-1.2cm]
\node (step1) at (\betaA,1.5) {};
\node[txt] at (step1) {$\Ex{V^t} = \Oh(n^3)$ for any $t \in [m - n^3 \log^4 n, m]$ (\cref{lem:initial_gap_nlogn})};
\draw[|-|, very thick, brown] (\betaA, 1.0) -- (\betaC, 1.0) ;

\node (step2) at (\betaA,0.5) {};
\node[txt] at (step2) {$\Lambda^t \leq n \cdot e^{ \C{poly_n_gap} n \log n}$ \Whp~for any $t \in [m - n^3 \log^4 n, m]$ (\cref{lem:initial_gap_nlogn})};
\draw[|-|, very thick, brown] (\betaA, 0) -- (\betaC, 0) ;

\node (step3) at (\betaA + 0.5,-0.5) {};
\node[txt] at (\betaA + 0.5,-0.55) {$\tilde{G}_{m - n^3 \log^4 n}^{m} \geq \tilde{r} \cdot n^3 \log^4 n$ \Whp, i.e., const fraction of rounds $t$ with $\Delta^t \leq Cn$};
\node (step3H1) at (\betaA + 0.6,-0.4) {};
\node[txt] at (\betaA + 0.5,-1.25) {$G_{m - n^3 \log^4 n}^{m} \geq r \cdot n^3 \log^4 n$ \Whp, i.e., const fraction of rounds $t$ with $\delta^t \in [\eps, 1 - \eps]$};
\node[txt] at (\betaA + 13.2, -1.75) {(\cref{lem:many_good_quantiles_whp})};
\node[anchor=east, black!50!blue] at (\betaA - 0.5, 1) {\small Markov};
\end{scope}

\node (step4) at (\betaA + 1.25,-2.5) {};

\node (step5) at (\betaB - 1,-3.5) {};
\node[txt] at (\betaB - 1, -3.5) {$\exists s \in [m - n^3 \log^4 n, m] : \Lambda^s \leq cn$ \Whp ({\small \textbf{Recovery \cref{lem:recovery}}})};
\draw[|-|, very thick, brown] (\betaA, -4.0) -- (\betaC, -4.0) ;
\node[Dot] at (\betaB,-4.0){};

\node (step6) at (\betaB + 0.5,-4.35) {};
\node[txt] at (step6) {For any interval $[t, t + c_s n \log n]$ for $t \in [s, m]$ with $\Lambda^t \leq cn$};
\node (step6H1) at (\betaB + 1,-4.8) {};
\node[txt] at (\betaB + 1,-5.0) {$\tilde{G}_{t}^{t + c_s n \log n} \geq \tilde{r} \cdot c_s n \log n$ \Whp};
\node (step6H2) at (\betaB + 1,-5.5) {};
\node[txt] at (\betaB + 1,-5.5) {$G_{t}^{t + c_s n \log n} \geq r \cdot c_s n \log n$ \Whp (\cref{lem:many_good_quantiles_whp})};
\draw[|-|, very thick, brown] (\betaB + 1, -6.0) -- (\betaB + 3.5, -6.0) ;

\node (step7) at (\betaB + 1,-6.5) {};
\node[txt] at (\betaB + 1, -6.5) {$\exists \tilde{s} \in [s, s + c_s n \log n]: \Lambda^{\tilde{s}} \leq 2\C{lambda_bound} n$ \Whp ({\small \textbf{Stabilization \cref{lem:stabilization}}})};
\draw[|-|, very thick, brown] (\betaB + 1, -7.0) -- (\betaB + 3.5, -7.0) ;
\node[Dot] at (\betaB + 3,-7.0){};

\node (step8) at (\betaB + 1,-7.5) {};
\node[txt] at (\betaB + 1, -7.5) {$\forall s: \exists \tilde{s} \in [s, s + c_s n \log n]: \Lambda^{\tilde{s}} \leq 2 \C{lambda_bound} n$ \Whp ({\small\cref{lem:good_gap_after_good_lambda}})};
\draw[|-|, very thick, brown] (\betaB + 1, -8.0) -- (\End, -8.0) ;

\node (step9) at (\betaC - 3,-8.5) {};
\node[txt] at (\betaC - 3, -8.5) {$\max_{i \in [n]} : |y_i^m| \leq \kappa \log n$ \Whp ({\small \textbf{\cref{thm:main_technical}}})};

\draw[implies] (step1) edge[bend right=70] (step2);

\draw[implies] (step2) edge[bend right=70] (step3H1);
\node[anchor=east, black!50!blue] at (\betaA - 0.2, -1.3) {\small \cref{lem:quadratic_absolute_relation}};
\node[anchor=east, black!50!blue] at (\betaA + 0.2, -1.7) {\small \& \cref{lem:many_good_quantiles_whp}};

\node[anchor=east, black!50!blue] at (\betaA + 0.2, -2.2) {\small \& \cref{lem:good_quantile}};

\draw[implies] (step3H1) edge[bend right=70] (\betaA + 0.45, -2.5);

\draw[implies] (step4) edge[bend right=70] (step5);
\node[anchor=east, black!50!blue] at (\betaA + 1.1, -3.3) {\small Using $\tilde{\Lambda}$ potential};
\node[anchor=east, black!50!blue] at (\betaA + 1.1, -3.7) {\small \& Markov};

\draw[implies] (step6) edge[bend right=70] (step6H1);

\draw[implies] (step6H1) edge[bend right=70] (step6H2);
\node[anchor=east, black!50!blue] at (\betaA + 4.25, -4.35) {\small \cref{lem:quadratic_absolute_relation}};
\node[anchor=east, black!50!blue] at (\betaA + 4.25, -4.75) {\small \& \cref{lem:many_good_quantiles_whp}};

\draw[implies] (step6H2) edge[bend right=70] (step7);
\node[anchor=east, black!50!blue] at (\betaA + 4.45, -5.25) {\small \cref{lem:good_quantile}};

\node[anchor=east, black!50!blue] at (\betaA + 4.45, -6.05) {\small Using $\tilde{\Lambda}$ potential};
\node[anchor=east, black!50!blue] at (\betaA + 4.45, -6.45) {\small \& Markov};

\draw[implies] (step7) edge[bend right=70] (step8);
\node[anchor=east, black!50!blue] at (\betaA + 4.45, -7.15) {\small U.~Bound};

\draw[implies] (step8) edge[bend right=70] (step9);
\draw[-latex, very thick] (\betaA - 0.5, \yO) -- (\End + .8, \yO);
\end{tikzpicture}}}
\end{center}
\caption{Summary of the key steps in the proof for \cref{thm:main_technical}. By $\tilde{G}_{t_0}^{t_1}$ we denote the number of rounds $t \in [t_0, t_1]$ with $\Delta^t \leq Cn$ and by $G_{t_0}^{t_1}$ the number of rounds $t \in [t_0, t_1]$ with $\delta^t \in [\eps, 1 - \eps]$, also $r=\eps^2$ and $\tilde{r}=4\eps^2$.}
\label{fig:proof_outline_tikz}
\end{figure}
 
\begin{figure}

 \begin{center}
\includegraphics[scale=0.73]{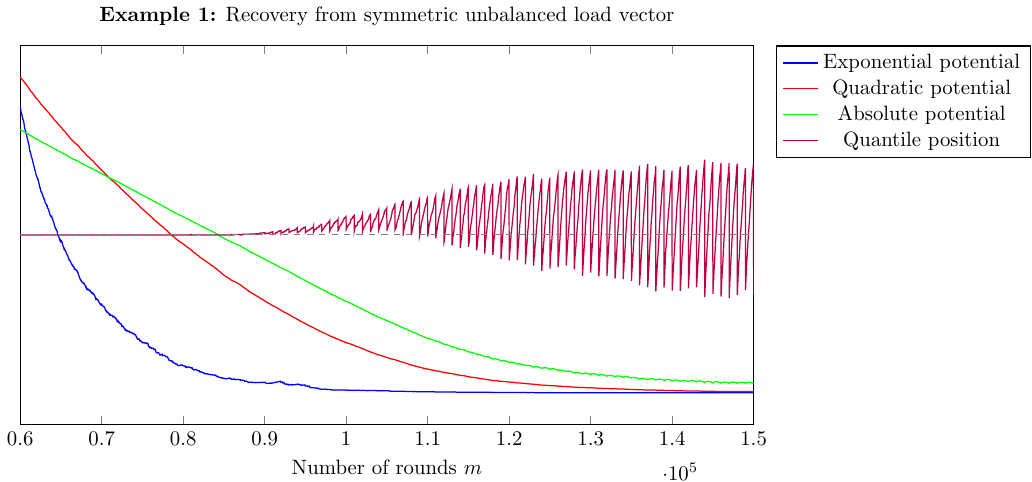}%

~\vspace{.7em}~

\includegraphics[scale=0.73]{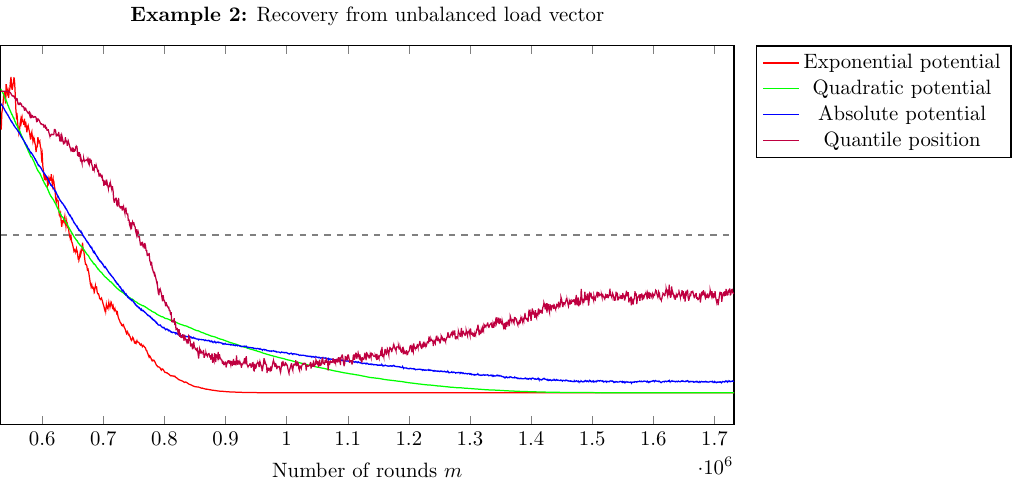}

~\vspace{.em}~

\includegraphics[scale=0.73]{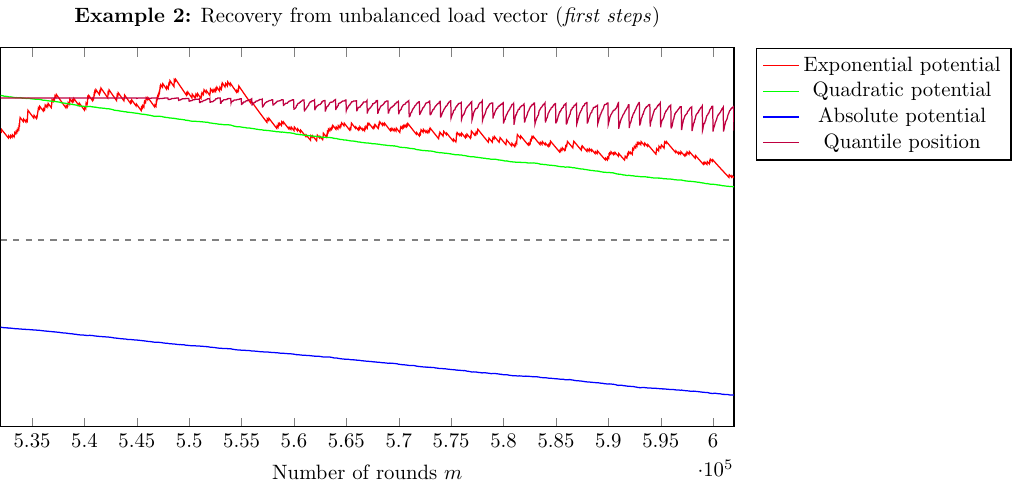}

\caption{Scaled versions of the potential functions for \MeanThinning with different initial load vectors for $n = 1000$ bins. (\textbf{Example 1}) In the top figure, the process starts from a load vector where half of the bins have (normalized) load $+\log n$ and the other half have $-\log n$. 
(\textbf{Example 2}) Here, the process starts from an unbalanced load vector where the quantile of the mean load remains very close to $1$ for $\omega(n \log n)$ many rounds. As can be seen, the exponential potential increases a bit at the beginning, but both the absolute and quadratic potential improve immediately. Once the mean quantile is sufficiently far from $0$ and $1$, the exponential potential also decreases, eventually stabilizing at $\Oh(n)$, as shown in the third figure. Note that the dashed gray line corresponds a perfectly balanced mean quantile, i.e., $\delta^t=1/2$.
} \label{fig:experiments}
\end{center}
\end{figure}

There are a considerable number of constants in this paper and their interdependence can be quite complex. We hope the following remark will shed some light on their respective roles. The concrete values are given in the beginning of \cref{sec:mean_biased_gap_completion}.

\begin{rem}[Relationship Between the Constants] All the constants used in the analysis depend on the constants $w_-$, $w_+$, $\C{p1k1}$, $\C{p1k2}$ (and \C{p2k1}, \C{p2k2} if \PTwo holds) of the process being analyzed. The relation between the absolute value and quadratic potential functions (and specifically \C{quad_delta_drop} and \C{quad_const_add} in \cref{lem:quadratic_absolute_relation}) defines what it means for $\Delta^t$ to be small, i.e., $\Delta^t \leq \C{small_delta} n$ (defined in \eqref{eq:c_def}). This $\C{small_delta}$ in turn specifies the constant $\eps \in \big( 0, \frac{1}{2} \big)$, given by \cref{lem:good_quantile}, which defines a good quantile to be $\delta^t \in [\eps, 1 - \eps]$. Then, $\eps$ specifies the fraction $r=\eps^2$ of rounds with $\delta^t \in [\eps, 1 - \eps]$ in a sufficiently long interval (in \cref{lem:many_good_quantiles_whp}) and the constants $\C{good_quantile_mult}, \C{bad_quantile_mult}$ in the inequalities for the expected change of $\Lambda^t$ for sufficiently small constant $\alpha$ (in Lemmas~\ref{lem:good_quantile_good_decrease} and~\ref{lem:bad_quantile_increase_bound}). Then, using the constants $\eps$, $\C{good_quantile_mult}, \C{bad_quantile_mult}$ we finally set the value of $\alpha$ in $\Lambda$. Finally, $\C{good_quantile_mult}$ and $\C{bad_quantile_mult}$ specify $\C{lambda_bound}$ in \cref{cor:change_for_large_lambda}, which defines what it means for $\Lambda$ to be \textit{small}, i.e., $\Lambda^t \leq cn$.
\end{rem}

\bigskip

\noindent \textbf{Organization of the Remaining Part of the Proof.}

\begin{itemize}[itemsep=-1pt, topsep=-1pt]
    \item In \cref{sec:mean_quantile_stabilization}, we prove the mean quantile stabilization Theorem (\cref{lem:good_quantile}), starting from a round with $\Delta^t \leq Cn$.
    \item In \cref{sec:potential_function_inequalities}, we start by proving the relation between the absolute value and the quadratic potentials (\cref{lem:quadratic_absolute_relation}), then analyze the expected change of the exponential potential $\Lambda$ (\cref{sec:lambda_potential}) and prove that $\Ex{V^t} = \poly(n)$, deducing the $\Oh(n \log n)$ gap (\cref{lem:initial_gap_nlogn}). Finally, we prove some deterministic bounds on the quadratic potential.
    \item In \cref{sec:mean_biased_gap_completion}, we complete the proof for the $\Oh(\log n)$ gap for \MeanBiased processes by analyzing the recovery (\cref{sec:recovery}) and  stabilization (\cref{sec:stabilization}) phases.
\end{itemize}

\section{Mean Quantile Stabilization} \label{sec:mean_quantile_stabilization}

In this section we prove that once the absolute value potential is $\Oh(n)$, then the allocation process will satisfy $\delta^{t} \in [\eps, 1 - \eps]$ for ``many'' of the following rounds. The condition on $\delta^{t}$ means that a constant fraction of bins are overloaded, and a constant fraction are underloaded. Interestingly, this lemma does not require either of the stronger properties \PTwo or \WTwo to hold.

\begin{thm}[Mean Quantile Stabilization]\label{lem:good_quantile}
\goodquantile
\end{thm}

		\begin{proof}
			
			Define $ B_{\#}:=\left\{ i \in [n] \colon |y_i^{t_0}| < 4\C{small_delta} \right\}$ to be the bins whose load deviates from the mean $W^{t}/n$ by less than $4\C{small_delta}$. Then, conditional on the event $\mathcal{C}:=\{\Delta^{t_0} \leq \C{small_delta} n \} $, we have
			\[
			\C{small_delta} n \geq \Delta^{t_0} = \sum_{i \in [n]} |y_i^{t_0}| >  \sum_{i \colon |y_i^{t_0}| \geq 4\C{small_delta}} 4\C{small_delta} = (n-|B_{\#}|) \cdot 4\C{small_delta},
			\]
			 rearranging this gives that conditional on $\mathcal{C}$,
			\begin{equation}\label{eq:sizeofB_*}
				|B_{\#}| > n\cdot \frac{\C{small_delta}(4-1)   }{4\C{small_delta}} =  \frac{3n}{4} .
			\end{equation}Note that the bins in $B_{\#}$  may be underloaded or overloaded. 
			
			We now proceed with two claims. Using the fact that $|B_{\#}|$ is large, the first claim (\cref{clm:underloaded}) proves that there exists an ``early'' round $t \in [t_0,t_1]$ such that a constant fraction of bins are underloaded. Similarly, the second claim (\cref{clm:overloaded})  proves that there exists a ``late'' round $t \in [t_2,t_3]$ with $t_2 =  t_1 + \Theta(n)$ such that a constant fraction of bins are overloaded. Since the set of overloaded bins can only increase by $1$ per round, we then finally conclude that for $\Omega(n)$ rounds $t\in[t_1,t_2]$, both conditions hold, i.e., $\delta^t \in [\eps, 1 - \eps]$. This is made formal by \cref{lem:smoothness}.   
				\begin{clm}\label{clm:underloaded}
				For any integer constant $\C{small_delta} \geq 1$, there exists a constant $\kappa_1:=\kappa_1(C)>0$ such that for $t_1:=t_0+ 4\C{small_delta}  \cdot n + 1 $ we have 
				\[
				\Pro{ \left. \bigcup_{ t \in [t_0,t_1] } \left\{\delta^t   \leq 1- \kappa_1  \right\}  \,\right|\, \mathfrak{F}^{t_0}, \Delta^{t_0} \leq C n } \geq 1-e^{-\kappa_1 n }.
				\]
				
			\end{clm}
			\begin{poc}
				If $|B_{-}^{t_0} \cap B_{\#}| >  |B_{\#}|/3$, then, since $|B_{\#}|> 3n/4$ by \eqref{eq:sizeofB_*}, the statement of the claim follows immediately for $t=t_0$ and $\kappa_1=1/4$.
				Otherwise, we may assume $|B_{+}^{t_0} \cap B_{\#}| \geq 2|B_{\#}|/3$, i.e., at least two thirds of the bins in $B_{\#}$ are overloaded at round $t_0$. Note that the bins in $B_{+}^{t_0} \cap B_{\#}$ all have loads in the range $[W^{t_0}/n,\;W^{t_0}/n + 4\C{small_delta})$. Thus, since loads are integers (as $w_{+}$ and $w_{-}$ are integers), there can be at most $ 4\C{small_delta}$ different load levels within $B_{+}^{t_0} \cap B_{\#}$. Hence, by the pigeonhole principle and \eqref{eq:sizeofB_*}, conditional on $\mathcal{C}$, there exists a subset $B_*\subseteq B_{+}^{t_0} \cap B_{\#}$  with \begin{equation}\label{eq:sizeBtileps}|B_*| \geq \frac{1}{4\C{small_delta}}\cdot |B_{+}^{t_0} \cap B_{\#}|  \geq \frac{1}{4\C{small_delta}}\cdot\frac{2}{3} \cdot |B_{\#}| \geq \frac{1}{4\C{small_delta}}\cdot\frac{2}{3} \cdot\frac{3n}{4}  = \frac{n}{8\C{small_delta}} ,\end{equation} such that all bins $i \in B_*$  have the same (non-negative) load at round $t_0$. Let this common load be denoted by  $0\leq \ell < 4\C{small_delta} $. 
				In the following, let us define the stopping time
				\[
				\tau:=\min\left\{ t \geq t_0 \colon \frac{W^{t}}{n} - \frac{W^{t_0}}{n} > \ell \right\}, 
				\]
				that is $\tau$ is the minimum number of rounds until the average load has dropped by more than $\ell$. By definition, if a bin in $B_*$ was never chosen for an allocation during any round $[t_0,\tau)$, it becomes underloaded. Also note that, while $\tau$ is a random variable, there is a deterministic upper bound of $T:=t_0+n \cdot \ell + 1$, since in each round the total load increases by at least $w_+ \geq 1$.       
				For any round $t \in [t_0+1,\tau]$, let $B_{*}^t$ be the subset of bins in $B_{*}$ that are never allocated to during $[t_0,t)$ and let $B_{*}^{t_0}:=B_{*}$. Note that since any bin in $B_{*}^t$ is still overloaded in such a round $t$, the bin is chosen for an allocation with probability at most $k_1/n$ by condition \POne.
				
				For any round $t \in [t_0,T]$, define
				\[
				Z^t := \frac{|B_{*}^t|}{(1-k_1/n)^{t-t_0}},
				\]
				so in particular, $Z^{t_0} = |B_{*}^{t_0}|=|B_{*}|$.
				Then, for any round $t\in [t_0,\tau)$, 
				\begin{align*}
					\Ex{Z^{t+1} \, \mid \, \mathfrak{F}^{t} } &= \frac{1}{(1-k_1/n)^{t+1 -t_0}} \cdot \Ex{ |B_{*}^{t+1}| \, \mid \, \mathfrak{F}^t } \\
					&\geq \frac{1}{(1-k_1/n)^{t+1 -t_0}} \cdot (1-k_1/n)|B_{*}^{t}| \\
					&= Z^t,
				\end{align*}
				thus $(Z^{t \wedge \tau})_{t=t_0}^{T}$ forms a sub-martingale.

				We now state a bound that will be used often in the upcoming proof: Since $1-x\geq e^{-2x} $ for any $0\leq x\leq 1/2$, $T-t_0 \leq   4\C{small_delta} \cdot n + 1\leq 5\C{small_delta}n$, and $n$ is large, for any $t \in [t_0,T]$ we have  
				\begin{equation}\label{eq:expt} (1-k_1/n)^{t-t_0} \geq e^{-(t-t_0)\cdot 2\C{p1k1}/n} \geq e^{-(T-t_0) \cdot 2\C{p1k1}/n}\geq e^{-10 \C{p1k1}\C{small_delta}}. \end{equation}  
				We proceed to bound $Z^{t+1}-Z^{t}$. First,
				observe that $|B_{*}^t|$ is non-increasing in $t$. Thus, by \eqref{eq:expt} and $|B_{*}^t| \leq n $, 
				\[Z^{t+1} -Z^t %
				\leq \frac{|B_{*}^t|}{(1-k_1/n)^{t-t_0}}\cdot \left( \frac{1}{ 1-k_1/n } -1 \right) \leq \frac{n}{e^{-10 \C{p1k1}\C{small_delta}}}\cdot  \frac{2k_1 }{ n}   = 2k_1e^{10 \C{p1k1}\C{small_delta}}.  \]  
				Similarly, for a lower bound, note that $|B_{*}^t|$ can decrease by at most one in each round, since only one bin is allocated to in each round, and therefore
				\begin{align*}Z^{t+1} -Z^t &\geq  \frac{|B_{*}^{t}|-1}{(1-k_1/n)^{t+1-t_0}} - \frac{|B_{*}^t|}{(1-k_1/n)^{t-t_0}}\\
					&\geq  \frac{|B_{*}^{t}|-1}{(1-k_1/n)^{t-t_0}} - \frac{|B_{*}^t|}{(1-k_1/n)^{t-t_0}}\\
					&\geq -\frac{1}{(1-k_1/n)^{t-t_0}}\\
					&\geq -e^{10 \C{p1k1}\C{small_delta}}.  \end{align*} 
				Thus $|Z^{t+1} -Z^t |\leq 2k_1e^{10 \C{p1k1}\C{small_delta}}$. Applying  Azuma's inequality for sub-martingales (\Cref{lem:azuma}) on $Z^{T \wedge \tau}=Z^{\tau}$ yields
    \begin{equation}\label{eqprobbdd}
					\Pro{ Z^{\tau}    \leq \frac{1}{2}\cdot Z^{t_0} \, \Big| \, \mathfrak{F}^{t_0} } \leq  \exp\left( - \frac{   \left(\frac{1}{2}\cdot Z^{t_0}\right)^2}{ 2 (T-t_0)\cdot (2k_1e^{10 \C{p1k1}\C{small_delta}})^2 } \right). 
				\end{equation} Observe that  $\{Z^{\tau}    \leq \frac{1}{2}\cdot Z^{t_0}\}  = \{|B_{*}^{\tau}| \leq (1-k_1/n)^{\tau-t_0} \cdot   \frac{1}{2}\cdot |B_{*}| \} $. Recall that, conditional on $\mathcal{C}:=\{\Delta^{t_0} \leq Cn\}$,  we have $|B_{*}| \geq \frac{n}{ 8\C{small_delta}} $ by \eqref{eq:sizeBtileps}, and $\tau-t_0 \leq   5\C{small_delta}n$. Thus, by \eqref{eq:expt} and \eqref{eqprobbdd},
				\[\Pro{ \left.|B_{*}^{\tau}| \leq e^{10 \C{p1k1} \C{small_delta}}\cdot \frac{1}{2}\cdot\frac{n}{ 8\C{small_delta}} \; \right| \; \mathfrak{F}^{t_0}, \; \mathcal{C}}\leq  \exp\left( - \frac{n}{5\cdot 2^{11}\cdot  k_1^2 C^3 e^{20 \C{p1k1}\C{small_delta}}} \right). 
				\]Since all load levels in $B_*$ before round $t_0$ are identical, we conclude that they all become underloaded the first time in the same round $\tau\in [t_0,t_1]$, so, by merging the constants
				\[
				\Pro{ \left. \bigcup_{t\in [t_0,t_1]} \{| B_{-}^{t} \cap B_* | \geq \kappa_1 \cdot n\}  \,\right|\, \mathfrak{F}^{t_0}, \; \mathcal{C} } \geq 1-e^{-\kappa_1 n },
				\]
				where $\kappa_1:=(5\cdot 2^{11}\cdot  k_1^2 C^3 e^{20 \C{p1k1}\C{small_delta}})^{-1}   >0$. 
			\end{poc}

			\begin{clm} \label{clm:overloaded}
				For any integer constant $\C{small_delta} \geq 1$ there exists a constant $\kappa_2:=\kappa_2(C)>0$ such that for $t_2 := t_1 + \big\lceil  \frac{n}{w_+}\big\rceil $, where $t_1:=t_0+ 4\C{small_delta}  n + 1 $, and
				$t_3 := t_2 + \big\lceil \frac{n}{10 w_-}\big\rceil$ we have
				\[
				\Pro{ \left. \bigcup_{ t \in [t_2,t_3] } \left\{\delta^t  \geq\kappa_2  \right\}  ~ \right| ~ \mathfrak{F}^{t_0}, \Delta^{t_0} \leq Cn } \geq 1-e^{-\kappa_2 n }.
				\]
			\end{clm}

			\begin{poc}%
				If  $|B_{+}^{t_2} \cap B_{\#}| > |B_{\#}   |/6$, the claim follows  by \eqref{eq:sizeofB_*}  for $t = t_2$ and $\kappa_2 = 1/8$. So, if we let $B_*:=  B_{-}^{t_2} \cap B_{\#}$, then we can assume, by \eqref{eq:sizeofB_*}, that  \begin{equation}\label{eq:sizebstar-}|B_*| \geq \frac{5}{6} \cdot |B_{\#}  | >\frac{5}{6} \cdot \frac{3n}{4} = \frac{5n}{8}      .\end{equation} 
				Observe that for any  bin $i \in  B_*$, at round $t_2$ the normalized load satisfies 
				\[
				y_i^{t_2} \geq y_{i}^{t_0} - (t_2-t_0) \cdot \frac{w_{-}}{n} >
				-4\C{small_delta}- \left( 4\C{small_delta} n +1 +\left\lceil  \frac{n}{w_{+}}\right\rceil\right)\cdot \frac{w_{-}}{n} > -10\C{small_delta}\cdot w_{-},
				\]
				using that $y_i^{t_0} > -4\C{small_delta}$ as $i \in B_* \subseteq B_{\#} $, $C\geq 1$, and $w_-\geq w_+\geq 1$ by \WOne. It follows that if $i$ is chosen for an allocation $20\C{small_delta}$ times during $[t_2,t_3]$ while underloaded, it must become overloaded before round $t_3$ as
				\begin{equation}\label{eq:undertoover}
					y_i^{t_3} \geq y_i^{t_2}  - (t_3-t_2)\frac{w_-}{n} + 20\C{small_delta}\cdot w_{-}  \geq 	-10\C{small_delta}\cdot w_{-} - 2\cdot w_{-}  + 20\C{small_delta}\cdot w_{-}
					> 0.
				\end{equation} 
				 Observe that, while $i \in B_{*}$ is underloaded, it will be chosen for allocation with probability at least $\C{p1k2}/n $ in each round by Condition \POne. Our aim now will be to couple the allocation process over the interval $[t_2,t_3]$ to some sequences of independent random variables. Firstly,  denote \begin{equation}\label{eq:sdef}B_*^t:=B_* \cap B_-^t \qquad \text{ and }\qquad s := \frac{4}{5}\cdot  |B_* |.\end{equation} By \eqref{eq:sizebstar-} we have 
				 \begin{equation}\label{eq:sizeBts}
				 	s := \frac{4}{5}\cdot  |B_* |\geq \frac{4}{5}\cdot \frac{5n}{8} = \frac{n}{2}.
				 \end{equation} Observe that, since $ t_3=t_2 + \big\lceil \frac{n}{10 w_-}\big\rceil$ and by \eqref{eq:sizeofB_*}, for any time $t\in [t_2,t_3] $ we have 
				 \begin{equation}\label{eq:bstarts}
		  |B_*^t| \geq |B_* |  - \left\lceil \frac{n}{10 w_-}\right\rceil  = s + \frac{1}{5}\cdot  |B_*| - \left\lceil \frac{n}{10 w_-}\right\rceil>  s + \frac{n}{8}- \left\lceil \frac{n}{10 w_-}\right\rceil \geq s   ,   \end{equation} thus at least $s$ many bins from $B_*$ remain underloaded during the interval $[t_2,t_3]$. For $t>t_2$ let $(\nu_i^t)_{i\in [n]}$ be given by \[\nu_i^t :=\begin{cases}
					 	\left(1-\frac{\C{p1k2}s}{n}\right)^{-1}\cdot p_i^t  &\text{ if }i\in [n]\setminus B_{*}^t,  \\
					\left(1-\frac{\C{p1k2}s}{n}\right)^{-1}\cdot\left(p_i^t-\frac{\C{p1k2}s}{|B_*^t|n}\right)  &\text{ if } i\in B_{*}^t.
				\end{cases} \] We first show that $\nu_i^t\in[0,1]$.  As $\C{p1k2}\leq 1$ we see that $\C{p1k2}s/n\leq 4/5$  by \eqref{eq:sdef}. Thus, for $i\in [n]\setminus B_*^t$, we have $0\leq  p_i^t \leq \nu_i^t \leq 5\C{p1k1}/n <  1 $. For $i\in  B_*^t$ we have $\nu_i^t \geq  p_i^t - \frac{\C{p1k2}s}{|B_*^t|n} > p_i^t - \frac{\C{p1k2}}{n} \geq 0 $, as $|B_{*}^t|>s$ by \eqref{eq:bstarts} and $B_*^t\subseteq B_-^t$.  Finally, since $p_i^t\geq \C{p1k2}/n $ for all $i\in B_-^t$ and $ s< |B_*^t|\leq |B_-^t| $ we have 
    \[\nu_i<   \frac{p_i^t}{ 1-\C{p1k2}s/n}\leq \frac{1-  \C{p1k2}|B_-^t|/n}{ 1-\C{p1k2}s/n}  <1 .  \] To see that $\nu^t$, where $\nu_i^t\in[0,1]$,  is a probability vector, we observe that    \[\left(1-\frac{\C{p1k2}s}{n}\right)\cdot \sum_{i\in [n]}\nu_i^t = \sum_{i\in [n]\setminus B_{*}^t} p_i^t  + \sum_{i\in B_{*}^t}\left(p_i^t-\frac{\C{p1k2}s}{|B_*^t|n}\right) = 1-\frac{\C{p1k2}s}{n}. \] 
				For a set $S\neq \emptyset$, the uniform distribution $\mathsf{Uni}(S)$ samples each $s\in S$ with probability $1/|S|$.
    
    The first step of the coupling is to sample two independent infinite sequences $(X^t)_{t>t_2}$ and $(U^j)_{j\geq 1}$, where $X^t \sim \mathsf{Ber}( \C{p1k2}s/n)$ i.i.d.\ for each $t> t_2$, and $U^j \sim \mathsf{Uni}(B_{*})$ i.i.d.,~for each $j\geq 1$.

				\noindent We can then generate the process as follows: initialize $j^{t_2+1}:=1$, and for each round $t> t_2$, 
				\begin{itemize}
					\item If $ X^t = 0$, then  allocate to a single bin $i\in [n]$ sampled according to $\nu^t $, and set $j^{t+1} :=  j^t$.
					
					\item If $X^t = 1$ and $U^{j^{t}}\in B_*^t $, then allocate  to the bin $U^{j^{t}}
					$, and set $j^{t+1}:=  j^{t}+1$.
					
					\item If $X^t = 1$ and $U^{j^t}\notin B_*^t$, then allocate to a bin $R^{t} \sim \mathsf{Uni}(B_*^t )$ sampled independently, and set $j^{t+1} :=  j^t+1$.
				\end{itemize} 
				Let $\mathcal{A}_i^t$ be the event that bin $i\in[n]$ receives the allocation in round $t\in[t_2+1,t_3]$ and observe that, for $i\in B_*^t$,
				\begin{align*}\Pro{\mathcal{A}_i^t\mid   X^t =1} &=  \Pro{U^{j^{t}} =i \mid U^{j^{t}}\in B_*^t}\cdot \Pro{U^{j^{t}}\in B_*^t} + \Pro{R^t= i}\cdot \Pro{U^{j^{t}}\notin B_*^t}\\
				&=  \frac{1}{|B_*^t|}\cdot \Pro{U^{j^{t}}\in B_*^t} + \frac{1}{|B_*^t|}\cdot   \Pro{U^{j^{t}}\notin B_*^t}\\
				&= \frac{1}{|B_*^t|}.
			\end{align*}  
				\noindent Thus, to see this gives a valid coupling, note that a ball is allocated to $i \in  B_{*}^t$ with probability   
				\begin{align*}  \Pro{\mathcal{A}_i^t} &=  \Pro{X^t=0}\cdot \nu_i^t +  \Pro{X^t=1}\cdot\Pro{\mathcal{A}_i^t\mid   X^t =1}\\ &= \left(1-\frac{\C{p1k2}s}{n}\right)\cdot  \frac{p_i^t-\frac{\C{p1k2}s}{|B_{*}^t|n}}{1-\C{p1k2}s/n}   + \frac{\C{p1k2}s}{n}\cdot \frac{1}{|B_{*}^t|}  \\
					& = p_{i}^t. \end{align*}The corresponding calculation for $i \in  [n]\setminus B_{*}^t$ is immediate.

				Observe that by \eqref{eq:sizeBts} we have $\C{p1k2}s/n\geq    \C{p1k2}/2  $. Thus if we define the event \[\mathcal{E}_1:=\left\{\sum_{t=t_2+1}^{t_3 } X^t \geq  \tau_1 \right\}, \qquad \text{where } \tau_1:=\left\lceil\frac{\C{p1k2}(t_3-t_2)}{4} \right\rceil.\] Then, by a Chernoff bound,   we have 
	\begin{equation}\label{eq:E1prob}   \Pro{\neg\mathcal{E}_1 ~\big|~\mathfrak{F}^{t_0} ,  \mathcal{C}} \leq  \Pro{\mathsf{Bin}\left(t_3-t_2,\frac{\C{p1k2}}{2} \right)  < \tau_1}\leq e^{-\frac{\C{p1k2} (t_3-t_2) }{2^2\cdot 2\cdot 2}} \leq e^{- \frac{\C{p1k2} n }{160 w_-}}.  \end{equation}
				For $i\in B_{*}$ let $S_i := \sum_{j=1}^{\tau_1 } \mathbf{1}_{\{U^j=i\}}$. Then, since $\Pro{U^j=i} =1/|B_{*}|\geq 1/n$,   we have 
				\begin{align}\label{eq:kappa4}
					\Pro{ S_i  \geq 20\C{small_delta} \;\big|\;\mathfrak{F}^{t_0} ,\; \mathcal{C}} &\geq \Pro{ \mathsf{Bin}\left(\tau_1 ,\; \frac{1}{n}\right) \geq  20\C{small_delta} }\notag\\ 
					&\geq  \binom{\tau_1 }{ 20\C{small_delta}} \left(\frac{1}{n} \right)^{ 20\C{small_delta}}\cdot\left(1-\frac{\C{p1k2}}{n}  \right)^{\tau_1 - 20\C{small_delta}}\notag  \\
					&\geq \left(\frac{\C{p1k2}}{800Cw_-} \right)^{20\C{small_delta}} \cdot e^{-1}=:\kappa,
				\end{align}
   since $\tau_1 = \left\lceil\frac{\C{p1k2}(t_3-t_2)}{4} \right\rceil = (1+o(1))\cdot \frac{\C{p1k2} n}{40 w_-}$, where in the last inequality in \eqref{eq:kappa4} we used $\binom{n}{k}\geq (\frac{n}{k})^k$, the inequality $1-x\geq e^{-2x} $ for any $0\leq x\leq 1/2$, and $n$ being large. Define the random variable \[Z := \left|\left\{ i \in B_{*} \;:\; S_i \geq 20\C{small_delta} \right\} \right|\quad\text{and the event } \quad \mathcal{E}_2 :=\left\{ Z\geq \kappa n/4\right\}.\]  
				Observe that by \eqref{eq:kappa4} and \eqref{eq:sizeBts},
				\[\Ex{Z\;\big|\;\mathfrak{F}^{t_0} ,\; \mathcal{C}} =  \sum_{i\in B_{*}} \Pro{ S_i  \geq 20\C{small_delta}\;\big|\;\mathfrak{F}^{t_0} ,\; \mathcal{C} } \geq |B_{*}|\cdot  \kappa  \geq \frac{\kappa  n}{2}. \]  
				Notice also that $Z$ is generated by the sequence $(U^j)_{j=1}^{\tau_1}$ of mutually independent $B_{*}$-valued random variables. Furthermore, changing the value of any single $U^j$ can cause $Z$ to change by at most one. Thus, applying the Method of Bounded Independent Differences~(\cref{mobd}), assuming \eqref{eq:sizebstar-}, yields,
				\begin{equation}\label{eq:E2prob}
					\Pro{\neg \mathcal{E}_2\;\big|\;\mathfrak{F}^{t_0} ,\; \mathcal{C}} =\Pro{  Z < \frac{1}{2} \cdot  \frac{\kappa  n}{2} \;\Bigg| \; \mathfrak{F}^{t_0} ,\; \mathcal{C}} \leq    \exp\left(- \frac{ \frac{1}{4} \left( \kappa  n/2 \right)^2 }{2\cdot \tau_1 \cdot 1^2} \right) \leq  \exp\left(-  \kappa ^2 n  \right).	\end{equation}

				Our first observation is that, conditional on $\mathcal{E}_1\cap \mathcal{E}_2$, at least $Z\geq \kappa n/4 $ bins become overloaded at least once at some round in $[t_2,t_3]$. To see this note conditional on $\mathcal{E}_1$ an allocation is made based on the value of $U^j$ for each $j\leq \tau_1$, and conditional on $\mathcal{E}_2$ there are $Z$ many $i\in B_{*}$ such that $\sum_{j=1}^{\tau_1 } \mathbf{1}_{\{U^{j}=i\}} \geq 20\C{small_delta}$. Finally, from the coupling, observe that whenever $U^j=i$, either an allocation is made to $i$, or $i$ is already overloaded. Thus, for $Z$ many bins at some time during the window $[t_2,t_3]$, either the bin has received at least $20C $ allocations (which would make them overloaded by \eqref{eq:undertoover}) or they became overloaded. However, if each bin is overloaded for only one round, and all at different rounds, then it is possible that at no single round in $[t_2,t_3]$ do we have $\Omega(n)$ overloaded bins. We now explain why this cannot happen.  
				
				Note that during the interval $[t_2,t_3]$, as $t_3-t_2 = \big\lceil \frac{n}{10 w_-}\big\rceil$, there is at most one round $t$ where the integer parts of the mean load changes, that is, there is only at most one round $t\in [t_2,t_3)$ such that $\lfloor W^{t}/  n \rfloor < \lfloor W^{t+1}/n \rfloor$. Hence only at the transition from $t$ to $t+1$ could an overloaded bin become underloaded during the interval $[t_2,t_3]$. Therefore, if $Z\geq \kappa n/4$  we must have
				\[
				| B_{+}^{t} \cap B_* | \geq  \kappa n/8 \qquad \mbox{ or } \qquad
				| B_{+}^{t_3} \cap B_* | \geq \kappa n/8.
				\]
				Hence, conditional on $\mathcal{E}_1\cap \mathcal{E}_2$, $Z \geq \kappa n/4$ implies that there exists $t \in [t_2,t_3]$ such that $|B_{+}^{t}| \geq \kappa n/8$. The claim then follows by choosing $\kappa_2 := \frac{1}{2}\cdot \min\{\kappa^2, \kappa/8, \frac{\C{p1k2} n }{160 w_-}\}  $ as 
				\[\Pro{\neg \mathcal{E}_1\cup \neg\mathcal{E}_2 \;\big|\;\mathfrak{F}^{t_0} ,\; \mathcal{C} }  \leq  \exp\left( - \frac{\C{p1k2} n }{160 w_-}\right) + \exp\left(-  \kappa ^2   \cdot n  \right) \leq \exp\left(-  \kappa_2   \cdot n  \right),   \] by the union bound over  \eqref{eq:E1prob} and \eqref{eq:E2prob}. 
			\end{poc}

			The first claim implies that, w.p.\ $ 1- e^{-\kappa_1 n}$, that the process reaches a round $s_1 \in [t_0,t_1]$ with at least $\kappa_1 n$ underloaded bins. The second claim shows, w.p.\ $ 1-e^{-\kappa_2 n}$, that the process reaches a round $s_2 \in [t_1+\lceil \frac{n}{w_{+}}\rceil ,t_3]$ with at least $\kappa_2 n$ overloaded bins. By the union bound, both events occur with probability $ 1- e^{-\kappa_1 n}-e^{-\kappa_2 n}$.
			Since $\delta^{t}=|B_{+}^{t}|/n$, we have $\delta^{t+1} \leq \delta^{t} + \frac{1}{n}$,
			and in this case applying \cref{lem:smoothness} below with \[r_0=s_1,\qquad  r_1=s_2,\qquad  f(t)=\delta^{t},\quad \text{and}\quad \vartheta:= \min \{ \kappa_2, \kappa_1,2/3 \},\] yields 
			\begin{align*}
				\left| \left\{ t \in [t_0,t_3] \colon \delta^t \in [\vartheta/2,1-\vartheta/2] \right\} \right| &\geq
				\left| \left\{ t \in [s_1,s_2] \colon \delta^t \in [\vartheta/2,1-\vartheta/2] \right\} \right| \\ &\geq \min\{ \vartheta/2 \cdot n, s_2 - s_1  \} \\
				&\geq \min\left\{ \vartheta/2 \cdot n, \frac{n}{w_+} \right\} .
			\end{align*}
			Thus, finally, we require an $\eps>0$ in the statement of \cref{lem:good_quantile} which satisfies the following three conditions: $(i):$ $ \eps \leq \min\{\vartheta/2, 1/w_+\} $,  $(ii):$  $ \frac{n}{4\eps}   \geq  t_3-t_0  = 4\C{small_delta} n+1   + \big\lceil \frac{n}{w_{+}}  \big\rceil + \big\lceil \frac{n}{10w_{-}}\big\rceil $,  and $(iii):$  $(4\eps)^{-1}$ is an integer. Observe that such an $\eps>0$ can always be found by making $\eps$ sufficiently small, giving the result.   
		\end{proof}

It remains to state and prove the technical lemma used in the proof of \Cref{lem:good_quantile}. 
  \begin{lem}\label{lem:smoothness}
For integers $0\leq r_0 \leq r_1$, and reals $\vartheta \in (0,2/3)$ and $\xi> 0 $, we let $f: [r_0,r_1] \cap \mathbb{N} \rightarrow [0,1]$ be a function satisfying 
\begin{enumerate}[label = (\roman*)]
 \item $f(r_0) \leq 1-\vartheta$,
 \item $f(r_1) \geq \vartheta$,
 \item  $f(t+1) \leq f(t) + \xi$ for all $t \in [r_0,r_1-1]$.
\end{enumerate}
Then,
\[
 \left| \left\{ t \in [r_0,r_1] \colon f(t) \in \left[ \frac{\vartheta}{2},1- \frac{\vartheta}{2} \right] \right\} \right| \geq \min \left\{ \frac{\vartheta}{2 \xi} , r_1-r_0  \right \}.
\]
\end{lem}
\begin{proof}
First consider the case where there exists a $t \in [r_0,r_1]$ with $f(t) \geq 1-\vartheta/2$. Further, let $t$ be the first round with that property, hence for any $s \in [r_0,t]$, $f(s) < 1 - \vartheta/2$. Further, thanks to the third property, $f(t-x) \geq f(t) - x \xi \geq  1 - \vartheta/2 - x \xi$ for any $x \geq 0$ (which also implies $t-r_0 \geq \vartheta / (2 \xi) $ due to precondition $(i)$), and thus as long as $0  \leq x  \leq \vartheta/(2 \xi)$,
\[
 f(t - x) \geq f(t) - x \xi \geq 1 - \vartheta/2 - \vartheta /2  \geq \vartheta /2,
\]
since $\vartheta \leq 2/3$.
Hence for any $s \in [t-\vartheta/(2 \xi),t]$,
\[
 f(s) \in [ \vartheta/2, 1-\vartheta/2].
\]
Now consider the case where for all rounds $t \in [r_0,r_1]$ we have $f(t) \leq 1 - \vartheta/2$. Since $f(r_1) \geq \vartheta$, we conclude for any $x \leq \vartheta/(2 \xi)$,
\[
 f(r_1 - x) \geq f(r_1) - x \xi \geq \vartheta -\vartheta / 2 \geq \vartheta / 2.
\]
Hence for any $s \in [\max\{r_0,r_1-\vartheta/(2 \xi)\},r_1]$ we have $ f(s) \in [ \vartheta/2, 1-\vartheta/2]$.
\end{proof}

\section{Potential Function Inequalities} \label{sec:potential_function_inequalities}

In this section, we derive several inequalities involving potential functions. 

In \cref{sec:quadratic_absolute_potential}, we establish the interplay between the \textit{absolute value} and \textit{quadratic} potential which, later in \cref{sec:taming}, will be used in combination with \cref{lem:good_quantile}  to establish that in a sufficiently long interval there is a constant fraction of rounds $t$ with $\delta^t \in [\eps, 1 - \eps]$.

In \cref{sec:lambda_potential}, most of the effort goes into establishing a drop in the expectation of the \textit{exponential potential} $\Lambda := \Lambda(\alpha)$ for sufficiently small smoothing parameter $\alpha > 0$, which depends on the constants defined by the process. One of the main insights for $\Lambda := \Lambda(\alpha)$ with $\alpha = \Theta(1)$ is \cref{cor:change_for_large_lambda}, which establishes: $(i)$ a \textit{drop} in the expectation of $\Lambda$ over one round if the mean quantile satisfies $\delta^t \in [\eps, 1 - \eps]$, and $(ii)$ a \textit{not too large} increase for any quantile. For the potential $V := V(\tilde{\alpha})$ with $\tilde{\alpha} = \Theta(1/n)$ we establish that it drops in every round. In turn this implies that $\Ex{V^m} = \poly(n)$, which implies the weak bound of $\Oh(n \log n)$ on the gap.

On a high level, much of the analysis in \cref{sec:lambda_potential} follows relatively standard estimates and bears resemblance to the one in \cite{PTW15}. However, the conditions we enforce on processes are more relaxed than those in \cite{PTW15} and this makes the analysis a lot more challenging. In particular, unlike in \cite{PTW15}, the conditions imposed on the probability allocation vector $p^t$ are time/load dependent. That being said, a reader may wish to skip this part (or the proofs). 

Finally, in \cref{sec:quadratic_and_exp_potentials}, we prove some deterministic bounds on the quadratic potential.

\subsection{Interplay between Absolute Value and Quadratic Potentials} \label{sec:quadratic_absolute_potential}

In this section, we prove the relation between the absolute value and quadratic potentials. We will use this in \cref{lem:many_good_quantiles_whp} to deduce that in a sufficiently long interval, there is a constant fraction of rounds with $\delta^t \in [\eps, 1 - \eps]$. 

\begin{lem}\label{lem:additive_drift}
Consider any $\PTwo \cap \WOne$-process or $\POne \cap \WTwo$-process. Then, there exists a constant $\C{quad_delta_drop}  > 0$, such that for any round $t \geq 0$,
	\[
 \sum_{i \in B_+^t} 2y_i^t p_i^t \cdot w_+ + \sum_{i \in B_-^t} 2y_i^t p_i^t \cdot w_-  \leq - \Delta^t \cdot \frac{\C{quad_delta_drop}}{n}.
	\]
	\end{lem}
\begin{proof}  Before we begin, note that $\sum_{i \in B_+^t} y_i^t = - \sum_{i \in B_-^t} y_i^t = \frac{1}{2} \cdot \Delta^t$. We consider two cases based on the conditions satisfied: 

\noindent \textbf{Case 1 [\PTwo and \WOne hold]:} In this case \WOne implies $w_{-} \geq w_{+}\geq 1$, and so we have
\begin{align}\label{eq:killweight}	\sum_{i \in B_+^t} 2y_i^t p_i^t \cdot w_+ + \sum_{i \in B_-^t} 2y_i^t p_i^t \cdot w_-
		 &= \sum_{i \in B_+^t} 2y_i^t p_i^t \cdot w_+ - \sum_{i \in B_-^t} 2|y_i^t| p_i^t\cdot w_- \notag \\ 
		 &\leq w_+\cdot \left(\sum_{i \in B_+^t} 2y_i^t p_i^t  - \sum_{i \in B_-^t} 2|y_i^t| p_i^t\right).  \end{align}
By \PTwo we have $\sum_{i=1}^k p_i^t \leq \sum_{i=1}^k \left(\frac{1}{n} - \frac{\C{p2k1} \cdot (1-\delta^t)}{n}\right)$ for all $1\leq k\leq \delta^t n$, and also that $y_i^t\geq 0$ and non-increasing for all $1\leq i\leq \delta^t n$ (equivalently for all $i\in B_i^t$). Hence, by \cref{lem:quasilem},
\[\sum_{i \in B_+^t} 2y_i^t \cdot p_i^t \leq \sum_{i \in B_+^t} 2y_i^t \cdot \left(\frac{1}{n} - \frac{\C{p2k1} \cdot (1-\delta^t)}{n}\right) =  \Delta^t \cdot \left(\frac{1}{n} - \frac{\C{p2k1} \cdot (1-\delta^t)}{n}\right).\] Similarly, $\sum_{i \in B_-^t} 2|y_i^t| p_i^t\geq \Delta^t \cdot \left(\frac{1}{n} + \frac{\C{p2k2} \cdot \delta^t}{n}\right) $, also by \cref{lem:quasilem}. Thus, it follows from \eqref{eq:killweight} that 
\[	\sum_{i \in B_+^t} 2y_i^t p_i^t \cdot w_+ + \sum_{i \in B_-^t} 2y_i^t p_i^t \cdot w_- \leq  w_+ \cdot \Delta^t \cdot \left( \frac{1}{n} - \frac{\C{p2k1} \cdot (1-\delta^t)}{n} - \frac{1}{n}-  \frac{\C{p2k2} \cdot \delta^t}{n} \right) \leq -\Delta^t\cdot \frac{\C{quad_delta_drop}}{n},\]where $\C{quad_delta_drop} :=  w_+ \cdot \min \{ \C{p2k1}, \C{p2k2} \}$.

\noindent \textbf{Case 2 [\POne and \WTwo hold]:} By applying \cref{lem:quasilem} using the majorization conditions in \POne,
	\[
\sum_{i \in B_+^t} 2y_i^t p_i^t \cdot w_+ + \sum_{i \in B_-^t} 2y_i^t p_i^t \cdot w_- \leq \sum_{i \in B_+^t} 2y_i^t \cdot \frac{w_+}{n} +    \sum_{i \in B_-^t} 2y_i^t  \cdot \frac{w_-}{n} = \frac{\Delta^t}{n}\cdot \left( w_+ -  w_- \right) \leq - \Delta^t\cdot \frac{\C{quad_delta_drop}}{n},
	\]
	for $\C{quad_delta_drop} := w_- - w_+ > 0$, since the weights are constants satisfying $w_- > w_+$.
  \end{proof}
We can now prove the main result in this section. 
\begin{lem}\label{lem:quadratic_absolute_relation}
\quadraticabsoluterelation
\end{lem}

\begin{proof}
We begin by decomposing the quadratic potential $\Upsilon$ over the $n$ bins, using $\Upsilon_i^{t+1}:=(y_i^{t+1})^2$:
\[
 \Upsilon^{t+1} = \sum_{i=1}^n \Upsilon_i^{t+1} = \sum_{i=1}^n  (y_i^{t+1})^2.
\]
We shall analyze the aggregate terms for overloaded and underloaded bins. By the \WOne condition, if the chosen bin is overloaded, we allocate a ball of weight $w_+$, otherwise we allocate a ball of weight $w_-$. Also, recall that $P_+^t := \sum_{i \in B_+^t} p_i^t$ and $P_-^t := \sum_{i \in B_-^t} p_i^t$.

\medskip 

\noindent\textbf{Case 1} [Overloaded bins]. Let $i \in [n]$ be any overloaded bin, Then,
\begin{align*}
\Ex{\left. \Upsilon_{i}^{t+1} \,\right|\, \mathfrak{F}^t } 
 & = \underbrace{p_i^t \cdot \Big(y_i^t + w_+ - \frac{w_+}{n}\Big)^2}_{\text{allocate to $i$}} + \underbrace{(P_+^t - p_i^t) \cdot \Big(y_i^t - \frac{w_+}{n}\Big)^2}_{\text{allocate to $B_+^t \setminus \lbrace i \rbrace$}} + \underbrace{P_-^t \cdot \Big(y_i^t - \frac{w_-}{n}\Big)^2}_{\text{allocate to $B_-^t$}}.
\end{align*}
Aggregating over all overloaded bins and expanding out the squares,
\begin{align}\label{eq:l1l2overload}
\sum_{i \in B_+^t} \Ex{\left. \Upsilon_{i}^{t+1} \,\right|\, \mathfrak{F}^t }
 & = \sum_{i \in B_+^t} \bigg[ (y_i^t)^2 + 2y_i^t \cdot \left(p_i^t \cdot \Big(w_+ - \frac{w_+}{n}\Big) - (P_+^t - p_i^t) \cdot \frac{w_+}{n} - P_-^t \cdot \frac{w_-}{n} \right)\notag \\ 
 & \qquad + p_i^t \cdot \Big(w_+ - \frac{w_+}{n} \Big)^2 + (P_+^t - p_i^t) \cdot \frac{(w_+)^2}{n^2} +  P_-^t \cdot \frac{(w_-)^2}{n^2} \bigg] \notag \\
 & = \sum_{i \in B_+^t} \bigg[ (y_i^t)^2 + 2y_i^t \cdot \Big(p_i^t \cdot w_+ - P_+^t \cdot \frac{w_+}{n} - P_-^t \cdot \frac{w_-}{n}\Big) \notag\\ 
 & \qquad + p_i^{t} \cdot (w_{+})^2 -2 \cdot p_i^t \cdot \frac{(w_+)^2}{n} + P_{+}^t \cdot \frac{(w_+)^2}{n^2} + P_{-}^t \cdot \frac{(w_-)^2}{n^2} \bigg] \notag \\
 & \leq \sum_{i \in B_+^t} \bigg[ (y_i^t)^2 + 2y_i^t \cdot \Big(p_i^t \cdot w_+ - P_+^t \cdot \frac{w_+}{n} - P_-^t \cdot \frac{w_-}{n}\Big) \bigg] + 2 \cdot (w_{-})^2,
\end{align}
using in the last step that \[
 \sum_{i \in B_+^t} \bigg[p_i^{t} \cdot (w_{+})^2 + P_{+}^t \cdot \frac{(w_+)^2}{n^2} + P_{-}^t \cdot \frac{(w_-)^2}{n^2} \bigg]
 \leq (w_+)^2 + P_{+}^t \cdot \frac{(w_+)^2}{n} + P_{-}^t \cdot \frac{(w_-)^2}{n} \leq 2 \cdot (w_{-})^2.
\]

\medskip

\noindent\textbf{Case 2} [Underloaded bins]. Let $i \in [n]$ be any underloaded bin. Then, 
\begin{align*}
\Ex{\left. \Upsilon_{i}^{t+1} \,\right|\, \mathfrak{F}^t } 
 & = \underbrace{p_i^t \cdot \Big(y_i^t + w_- - \frac{w_-}{n}\Big)^2}_{\text{allocate to $i$}} + \underbrace{(P_-^t - p_i^t) \cdot \Big(y_i^t - \frac{w_-}{n}\Big)^2}_{\text{allocate to $i \in B_-^t \setminus \lbrace i \rbrace$}} + \underbrace{P_+^t \cdot \Big(y_i^t - \frac{w_+}{n}\Big)^2}_{\text{allocate to $i \in B_+^t$}}.
\end{align*}
Aggregating over all underloaded bins and expanding out the squares,
\begin{align} \label{eq:l1l2underload}
\sum_{i \in B_-^t} \Ex{\left. \Upsilon_{i}^{t+1} \,\right|\, \mathfrak{F}^t } 
 & = \sum_{i \in B_-^t} \bigg[ (y_i^t)^2 + 2y_i^t \cdot \Big(p_i^t \cdot \Big(w_- - \frac{w_-}{n}\Big) - (P_-^t - p_i^t) \cdot \frac{w_-}{n} - P_+^t \cdot \frac{w_+}{n}\Big) \notag \\
 & \qquad +p_i^t \cdot \Big(w_- - \frac{w_-}{n} \Big)^2 + (P_-^t - p_i^t) \cdot \frac{(w_-)^2}{n^2} + P_+^t \cdot \frac{(w_+)^2}{n^2} \bigg] \notag \\
 & = \sum_{i \in B_-^t} \bigg[ (y_i^t)^2 + 2y_i^t \cdot \Big(p_i^t \cdot w_- - P_-^t \cdot \frac{w_-}{n} - P_+^t \cdot \frac{w_+}{n} \Big)\notag  \\
 & \qquad +p_i^t \cdot (w_-)^2 - 2 \cdot p_i^t \cdot \frac{(w_-)^2}{n} + P_-^t \cdot \frac{(w_-)^2}{n^2} + P_+^t \cdot \frac{(w_+)^2}{n^2} \bigg]\notag  \\
 & \leq \sum_{i \in B_-^t} \bigg[(y_i^t)^2 + 2y_i^t \cdot \Big(p_i^t \cdot w_- - P_-^t \cdot \frac{w_-}{n} - P_+^t \cdot \frac{w_+}{n} \Big) \bigg] + 2\cdot (w_-)^2,
\end{align}
using in the last step that
\[
  \sum_{i \in B_-^t} \bigg[ p_i^{t} \cdot (w_{-})^2 + P_{-}^t \cdot \frac{(w_-)^2}{n^2} + P_{+}^t \cdot \frac{(w_+)^2}{n^2} \bigg] 
  \leq (w_+)^2 + P_{+}^t \cdot \frac{(w_+)^2}{n} + P_{-}^t \cdot \frac{(w_-)^2}{n}
  \leq 2 \cdot (w_{-})^2.
\]

\medskip 

Recall that $\sum_{i \in B_+^t} y_i^t = - \sum_{i \in B_-^t} y_i^t = \frac{1}{2} \cdot \Delta^t$. This leads to cancellation from some terms in the brackets when we combine \eqref{eq:l1l2overload} and \eqref{eq:l1l2underload}, as follows
\begin{align*}
\Ex{\left. \Upsilon^{t+1} \,\right|\, \mathfrak{F}^t} 
&= \sum_{i \in B_+^t} \Ex{\left. \Upsilon_{i}^{t+1} \,\right|\, \mathfrak{F}^t} + \sum_{i \in B_-^t} \Ex{\left. \Upsilon_{i}^{t+1} \,\right|\, \mathfrak{F}^t} \\
 & \leq 
  \Upsilon^t  + \sum_{i \in B_+^t} 2y_i^t p_i^t \cdot w_+ + \sum_{i \in B_-^t} 2y_i^t p_i^t \cdot w_- + 4 \cdot (w_-)^2. \qedhere\end{align*}
\end{proof}

\subsection{Exponential Potential} \label{sec:lambda_potential}

In this section, we consider the exponential potential $\Lambda := \Lambda(\alpha)$ defined in \cref{eq:lambda_def}. Let $\mathcal{G}^t := \mathcal{G}^t(\eps)$ be the event that $\delta^t \in [\eps, 1 - \eps]$ holds. In \cref{lem:good_quantile_good_decrease}, we will show that the potential $\Lambda$ drops in expectation over one round $t$ when $\mathcal{G}^t$ holds, i.e.,
\[
\Ex{\left. \Lambda^{t+1} \,\right|\, \mathfrak{F}^t, \mathcal{G}^t} \leq \Lambda^t \cdot \left(1 - \frac{2 \C{good_quantile_mult} \alpha}{n} \right) + 8.
\]
Note that for constant smoothing parameter $\alpha$, the exponential potential cannot decrease in expectation in every round (\cref{clm:bad_configuration_lambda}). Therefore, in \cref{lem:bad_quantile_increase_bound} we will show that it does not increase too much in a round when the mean quantile is not good, i.e.,
\[
\Ex{\left. \Lambda^{t+1} \,\right|\, \mathfrak{F}^t, \neg \mathcal{G}^t} \leq  \Lambda^t \cdot \left( 1 + \frac{\C{bad_quantile_mult} \alpha^2}{2n}\right) + \C{bad_quantile_mult}.
\]
In the above expression the decrease factor can be made arbitrarily larger than the increase factor, by choosing $\alpha > 0 $ small enough. Once we prove that there is a constant fraction of rounds with a good quantile (\cref{lem:many_good_quantiles_whp}), we will use an adjusted version of $\Lambda$ (defined in \eqref{eq:tilde_lambda_def}) to show that overall it decreases in expectation (\cref{lem:recovery} and \cref{lem:stabilization}), when the potential is sufficiently large.

For the analysis of $\Lambda$ as well as $V$, we consider the labeling of the bins $i \in [n]$ used by the allocation process in round $t$ so that $x_i^{t}$ is non-decreasing in $i \in [n]$. We write 
\[
\Lambda^t(\alpha) =: \sum_{i=1}^n \Lambda_i^t = \sum_{i=1}^n e^{\alpha |y_i^t|} \quad  \Big(\text{and} \quad V^t(\tilde{\alpha}) =: \sum_{i=1}^n V_i^t = \sum_{i=1}^n e^{\tilde{\alpha} |y_i^t|} \Big),
\]
and handle separately the following three cases of bins based on their load: %
\begin{itemize}
  \item \textbf{Case 1} [Robustly Overloaded Bins]. The set of bins $B_{++}^t$ defined as the set of bins with load $y_i^t \geq \frac{w_-}{n}$. These bins are \textit{robustly overloaded} in the sense that they are guaranteed to be overloaded in round $t+1$ (i.e., to be in $B_+^{t+1}$), since the mean load can increase by at most $\frac{w_-}{n}$.
  
  For the exponential potential $\Lambda^t$ (and $V^t$ respectively), the expected contribution of a single bin $i \in B_{++}^t$ in $\Lambda^{t+1}$ is given by,
\begin{equation} \label{eq:change_case_1}
\Ex{ \left. \Lambda_i^{t+1} \,\right|\, \mathfrak{F}^t}
 = \Lambda_i^t \cdot \Big(\underbrace{p_i^t \cdot e^{\alpha w_+ -\alpha w_+/n}}_{\text{allocate to $i$}} + \underbrace{(P_+^t - p_i^t) \cdot e^{-\alpha w_+/n}}_{\text{allocate to $B_{+}^t\setminus \{ i \}$}} + \underbrace{P_-^t \cdot e^{-\alpha w_-/n} }_{\text{allocate to $B_{-}^t$}}\Big).
\end{equation}
  
  \item \textbf{Case 2} [Robustly Underloaded Bins]. The set of bins $B_{--}^t$ with load $y_i^t \leq -w_-$. These are bins in $B_-^t$ that are guaranteed to be underloaded in round $t+1$ (i.e., to be in $B_-^{t+1}$), since any bin can be allocated a ball of weight at most $w_-$ in one round.
  
  For the exponential potential $\Lambda^t$ (and $V^t$ respectively), the expected contribution of a single bin $i \in B_{--}^t$ in $\Lambda^{t+1}$ is given by,
\begin{align}
\Ex{ \left. \Lambda_i^{t+1} \,\right|\, \mathfrak{F}^t}
  = \Lambda_i^t \cdot \Big( \underbrace{p_i^t \cdot e^{- \alpha w_- + \alpha w_-/n}}_{\text{allocate to $i$}} + \underbrace{(P_-^t - p_i^t) \cdot e^{\alpha w_-/n}}_{\text{allocate to $B_-^t \setminus \{i \}$}} + \underbrace{P_+^t \cdot e^{\alpha w_+/n}}_{\text{allocate to $B_+^t$}} \Big). \label{eq:change_case_2}
\end{align}
  \item \textbf{Case 3} [Swinging Bins]. The set of bins $B_{+/-}^t$ with load $y_i^t \in (-w_-, \frac{w_-}{n})$.
\end{itemize}

We begin by showing that the aggregated contribution of the swinging bins to the expected change of the potential $\Lambda$ is \textit{small}. This will be used in the proofs of Lemmas \ref{lem:good_quantile_good_decrease} and \ref{lem:bad_quantile_increase_bound}.

\begin{lem}\label{lem:bins_close_to_mean}
Consider the potential $\Lambda := \Lambda(\alpha)$ for any $\alpha \in (0, \frac{1}{4w_-}]$ and any non-negative vector $(\kappa_i)_{i=1}^n$. Then, for any set of swinging bins $S \subseteq B_{+/-}^t$ and for any round $t \geq 0$,
\[
\sum_{i \in S} \Ex{ \left. \Lambda_i^{t+1} \,\right|\, \mathfrak{F}^t} \leq \sum_{i \in S} \Lambda_i^t \cdot ( 1 - \kappa_i ) + 2 \cdot \sum_{i \in S} \kappa_i + 3.
\]
\end{lem}
Note that this statement says that we can make the \textit{multiplicative factor} arbitrarily small at the expense of increasing the \textit{additive factor}.
\begin{proof}
For each bin $i \in B_{+/-}^t$, the two events affecting the contribution of the bin to $\Lambda^{t+1}$ are $(i)$ ``internal'' due to a ball being allocated to $i$ and $(ii)$ ``external'' due to the change in the mean. 

For $(i)$, the chosen bin $i \in B_{+/-}^t$ can increase by at most $w_-$, so the potential value satisfies $\Lambda_i^{t+1} \leq e^{2\alpha w_-}$, since $\alpha \leq \frac{1}{4w_-}$. For $(ii)$, the maximum change in the mean load is at most $w_-/n$, and so $\Lambda_i^{t+1} \leq \Lambda_i^t \cdot e^{\alpha w_- /n}$.  Combining the two contributions and aggregating over all bins in $S$,
\begin{align*}
\sum_{i \in S} \Ex{\left. \Lambda_i^{t+1} \,\right|\, \mathfrak{F}^t} 
 & \leq 
\sum_{i \in S} \Lambda_i^t \cdot e^{\alpha w_-/n} \cdot (1 - p_i^t) + \sum_{i \in S} 2 \cdot p_i^t \\
 & \leq \sum_{i \in S} \Lambda_i^t \cdot e^{\alpha w_-/n} + e^{2\alpha w_-} \\
 & \stackrel{(a)}{\leq} \sum_{i \in S} \Lambda_i^t \cdot \Big(1 + \frac{2 \alpha w_-}{n}\Big) + e^{2\alpha w_-} 
 \\ 
 &\stackrel{(b)}{\leq} \sum_{i \in S} \Lambda_i^t + 2\alpha w_- \cdot e^{\alpha w_-} + e^{2\alpha w_-} \\
 & \stackrel{(c)}{\leq} \sum_{i \in S} \Lambda_i^t + 3,
\end{align*}
using in $(a)$ the Taylor estimate $e^z \leq 1 + 2z$ for any $z < 1.2$, for sufficiently large $n$, in $(b)$~that $\Lambda_i^{t} \leq e^{\alpha w_{-}}$ for any $i \in B_{+/-}^t$ and in $(c)$ that $2\alpha w_- \cdot e^{\alpha w_-} + e^{2\alpha w_-} \leq 3$ since $\alpha w_- \leq \frac{1}{4}$.  By adding and subtracting $\sum_{i \in S} \Lambda_i^t \cdot \kappa_i$,
\begin{align*}
\sum_{i \in S} \ex{\Lambda_i^{t+1} \mid \mathfrak{F}^t} 
 & \leq \sum_{i \in S} \Lambda_i^t\cdot (1 - \kappa_i ) + \sum_{i \in S} \Lambda_i^t \cdot \kappa_i + 3 \leq \sum_{i \in S} \Lambda_i^t\cdot (1 - \kappa_i) + 2 \cdot \sum_{i \in S} \kappa_i + 3,
\end{align*}
using that $\Lambda_i^t \leq e^{\alpha w_-} \leq 2$ for $i \in S$.
\end{proof}

Next, we show that if the mean quantile $\delta^t$ of the mean is in $[\eps, 1 - \eps]$, then the potential $\Lambda$ drops in expectation over one round. The following lemma applies to both $\Lambda$ and $V$, as $V$ is an instance of $\Lambda$. For $\Lambda := \Lambda(\alpha)$ with $\alpha = \Theta(1)$ this lemma applies when $\eps$ is constant (\cref{cor:change_for_large_lambda}). For $V := V(\tilde{\alpha})$ with $\tilde{\alpha} = \Theta(1/n)$, this lemma applies at any round as it always holds that $\delta^t \in [1/n, 1 - 1/n]$ or $\delta^t = 1$ in which case all loads are equal to the mean load (\cref{lem:initial_gap_nlogn}).

\begin{lem}\label{lem:good_quantile_good_decrease} Consider any $\WOne \cap \PTwo$-process, any $\eps \in \big(0,\frac{1}{2}\big)$, $\C{good_quantile_mult}:= \frac{\eps}{16} \cdot \min\{ w_+ \C{p2k1}, w_- \C{p2k2} \}$ and any $\alpha$ such that
\begin{equation}\label{eq:c_3alphacond1}
0 < \alpha \leq \min\left\lbrace\frac{1}{4w_-}, \frac{\C{p2k2} \eps}{2 w_- (1 + \C{p2k2}\eps)}, \frac{\C{p2k1} \eps}{2 w_+(1-\C{p2k1} \eps)} \right\rbrace.
\end{equation}
Likewise, for any $\POne \cap \WTwo$-process, any $\eps \in \big(0,\frac{1}{2}\big)$, $\C{good_quantile_mult}:= \frac{\eps}{16} \cdot (w_- - w_+)$ and any $\alpha > 0$ such that
\begin{equation}\label{eq:c_3alphacond2}
0 <  \alpha \leq \min\left\lbrace\frac{1}{4w_-}, \frac{\eps( w_- - w_+)}{4 w_-^2}, \frac{\eps}{2 w_- (2+\eps)} \right\rbrace.
\end{equation}
If \eqref{eq:c_3alphacond1} or \eqref{eq:c_3alphacond2} holds,
 then for the potential $\Lambda := \Lambda(\alpha)$ and for any round $t\geq 0$ we have
\[
\Ex{ \left. \Lambda^{t+1} \,\right|\, \mathfrak{F}^t, \mathcal{G}^t(\eps) } \leq \Lambda^t \cdot \left(1 - \frac{2 \C{good_quantile_mult} \alpha}{n} \right) + 8. %
\]
\end{lem}
\begin{proof}Consider an arbitrary round $t$ such that $\mathcal{G}^t := \mathcal{G}^t(\eps) := \{ \delta^t \in [\eps, 1 - \eps] \}$ holds and label the bins such that $x_i^{t}$ is non-decreasing in $i \in [n]$.

\medskip

\noindent\textbf{Case 1} [Overloaded Bins]. Among the overloaded bins, let us first consider robustly overloaded bins. By \cref{eq:change_case_1}, for any robustly overloaded bin $i \in B_{++}^t$,
\begin{align}
\Ex{ \left. \Lambda_i^{t+1} \, \right| \, \mathfrak{F}^t}
 & = \Lambda_i^t \cdot \Big(p_i^t \cdot e^{-\alpha w_+/n + \alpha w_+} + (P_+^t - p_i^t) \cdot e^{-\alpha w_+/n} + P_-^t \cdot e^{-\alpha w_-/n} \Big) . \notag \\
 \intertext{Applying the Taylor estimate $e^{z} \leq 1+z+z^2$, which holds for any $z \leq 1.75$, since $\alpha w_+ \leq 1$ (and $\alpha w_- \leq 1$), and using that $\big( -\frac{\alpha w_+}{n} + \alpha w_+\big)^2 \leq (\alpha w_+)^2$,}
 \Ex{\left. \Lambda_i^{t+1} \,\right|\, \mathfrak{F}^t} & \leq \Lambda_i^t \cdot \Big(1 + p_i^t \cdot \Big(-\frac{\alpha w_+}{n} + \alpha w_+ + (\alpha w_+)^2 \Big) \notag  \\ 
 & \qquad + (P_+^t - p_i^t) \cdot \Big(-\frac{\alpha w_+}{n} + \Big(\frac{\alpha w_+}{n}\Big)^2 \Big) + P_-^t \cdot \Big(- \frac{\alpha w_-}{n} + \Big(\frac{\alpha w_-}{n} \Big)^2 \Big) \Big). \notag \\
 \intertext{Using that $P_-^t + P_+^t = 1$, $w_- \geq w_+$ and rearranging terms we obtain,}
 \Ex{ \left. \Lambda_i^{t+1} \,\right|\, \mathfrak{F}^t} & \leq \Lambda_i^t \cdot \Big(1 + p_i^t \cdot \Big(-\frac{\alpha w_+}{n} + \alpha w_+ + (\alpha w_+)^2 \Big) \notag \\
 & \qquad - (P_+^t - p_i^t) \cdot \frac{\alpha w_+}{n} - P_-^t \cdot \frac{\alpha w_-}{n} + \left( \frac{\alpha w_-}{n} \right)^2 \Big) \notag \\
 & = \Lambda_i^t \cdot \Big(1 + p_i^t \cdot \Bigl(\alpha w_+ + (\alpha w_+)^2 \Bigr) - P_+^t \cdot \frac{\alpha w_+}{n} - P_-^t \cdot \frac{\alpha w_-}{n} + \left( \frac{\alpha w_-}{n} \right)^2 \Big) \notag \\
 & = \Lambda_i^t \cdot \Big(1 + p_i^t \cdot (\alpha w_+ + (\alpha w_+)^2) + P_+^t \cdot \Big(\frac{\alpha w_-}{n} - \frac{\alpha w_+}{n}\Big) - \frac{\alpha w_-}{n} + \left( \frac{\alpha w_-}{n} \right)^2 \Big).\label{eq:positive_robustly}
\end{align}
Now, we turn our attention to swinging overloaded bins. By \cref{lem:bins_close_to_mean}, for $S := B_{+/-}^t \cap B_{+}^t$ and $\kappa_i := \frac{\alpha w_-}{n}$,\begin{align}
\sum_{i \in S}  \Ex{ \left. \Lambda_i^{t+1} \,\right|\, \mathfrak{F}^t}
 & \leq \sum_{i \in S} \Lambda_i^t \cdot \Big( 1 - \frac{\alpha w_-}{n} \Big) + 2 \cdot \sum_{i \in S} \frac{\alpha w_-}{n} + 3 \notag \\
 & \leq \sum_{i \in S} \Lambda_i^t \cdot \Big( 1 - \frac{\alpha w_-}{n} \Big) + 4  \notag \\
 & \leq \sum_{i \in S} \Lambda_i^t \cdot \Big(1 + p_i^t \cdot (\alpha w_+ + (\alpha w_+)^2) + P_+^t \cdot \Big(\frac{\alpha w_-}{n} - \frac{\alpha w_+}{n}\Big) \notag \\
 & \qquad \qquad - \frac{\alpha w_-}{n} + \left( \frac{\alpha w_-}{n} \right)^2 \Big) + 4, \label{eq:positive_swinging}
\end{align}
where in the last inequality we have added several non-negative terms.

Combining \eqref{eq:positive_robustly} and \eqref{eq:positive_swinging}, we have that\begin{align} \label{eq:lambda_case_1}
& \sum_{i \in B_+^t}  \Ex{ \left. \Lambda_i^{t+1} \,\right|\, \mathfrak{F}^t}  \notag \\
 & \leq \sum_{i \in B_+^t} \Lambda_i^t \cdot \Big(1 + p_i^t \cdot (\alpha w_+ + (\alpha w_+)^2) + P_+^t \cdot \Big(\frac{\alpha w_-}{n} - \frac{\alpha w_+}{n}\Big) - \frac{\alpha w_-}{n} + \left( \frac{\alpha w_-}{n} \right)^2 \Big) + 4.
\end{align}

By condition \POne, we have $P_+^t = \sum_{i = 1}^{\delta^t n} p_i^t \leq (\delta^t n) \cdot \frac{1}{n} = \delta^t$. By the assumption $\mathcal{G}^t$, we have that $\delta^t \leq 1-\eps$, so it also follows that $P_+^t \leq 1-\eps$. Hence,
\begin{align*}
\sum_{i \in B_+^t} & \Ex{ \left. \Lambda_i^{t+1} \,\right|\, \mathfrak{F}^t, \mathcal{G}^t} \\ 
 & \leq \sum_{i \in B_+^t} \Lambda_i^t \cdot \Big(1 - \alpha w_+ \cdot \Bigl(\frac{1}{n} - p_i^t \cdot (1 + \alpha w_+ )\Bigr) - \frac{\alpha \eps}{n} \cdot (w_- - w_+) + \left( \frac{\alpha w_-}{n} \right)^2 \Big) + 4.
\end{align*}
\noindent \textbf{Case 1.A} [\PTwo holds]. Recall that by condition \PTwo, the probabilities of the overloaded bins satisfy,
\[
\left( p_1^t, \ldots, p_{n \delta^t}^t \right) \preceq \left( \frac{1}{n} - \frac{\C{p2k1} \cdot (1-\delta^t)}{n}, \ldots, \frac{1}{n} - \frac{\C{p2k1} \cdot (1-\delta^t)}{n} \right)
\]
and hence applying \cref{lem:quasilem} using that $\Lambda_i^t$ are non-increasing in $i$ over $B_+^t$,
\begin{align*}
\sum_{i \in B_+^t} & \Ex{ \left. \Lambda_i^{t+1} \right| \mathfrak{F}^t, \mathcal{G}^t} \\
 & \leq \sum_{i \in B_+^t} \Lambda_i^t \cdot \Big(1 - \frac{\alpha w_+}{n} \cdot \Bigl(1 - (1 - \C{p2k1} \eps) \cdot (1 + \alpha w_+ )\Bigr) - \frac{\alpha \eps}{n} \cdot (w_- - w_+) + \left( \frac{\alpha w_-}{n} \right)^2 \Big) + 4 \\
 & \leq \sum_{i \in B_+^t} \Lambda_i^t \cdot \Big(1 - \frac{\alpha w_+}{n} \cdot \Bigl(1 - (1 - \C{p2k1} \eps) \cdot (1 + \alpha w_+ )\Bigr) + \left( \frac{\alpha w_-}{n} \right)^2 \Big) + 4 \\
 & = \sum_{i \in B_+^t} \Lambda_i^t \cdot \Big(1 - \frac{\alpha w_+}{n} \cdot \Bigl(\C{p2k1} \eps - \alpha w_+ \cdot (1 - \C{p2k1} \eps )\Bigr) + \left( \frac{\alpha w_-}{n} \right)^2 \Big) + 4 \\
 & \stackrel{(a)}{\leq} \sum_{i \in B_+^t} \Lambda_i^t \cdot \Big(1 - \frac{\alpha w_+ }{n} \cdot \frac{\C{p2k1} \eps}{2} + \left( \frac{\alpha w_-}{n} \right)^2 \Big) + 4 \\ 
 & \stackrel{(b)}{\leq} \sum_{i \in B_+^t} \Lambda_i^t \cdot \Big(1 - \frac{\alpha\eps }{4n} \cdot w_+ \C{p2k1} \Big) + 4,
\end{align*}
using in $(a)$ that $\frac{\C{p2k1} \eps}{2}\geq \alpha w_+(1 - \C{p2k1} \eps )$ (which is equivalent to $\alpha \leq \frac{\C{p2k1} \eps}{2 w_+(1-\C{p2k1} \eps)}$) and in $(b)$ that $\frac{\alpha\eps }{4n} \cdot w_+ \C{p2k1} \geq \left( \frac{\alpha w_-}{n} \right)^2$
(which follows since $\alpha=O(\eps)$).

\medskip

\noindent \textbf{Case 1.B} [\WTwo holds]. By condition \PTwo, the probabilities of the overloaded bins satisfy
\[
\left( p_1^t, \ldots, p_{\delta^t n}^t \right) \preceq \left( \frac{1}{n}, \ldots, \frac{1}{n} \right),
\]
and hence applying \cref{lem:quasilem} using that $\Lambda_i^t$ are non-increasing in $i$ over $B_+^t$,
\begin{align*}
\sum_{i \in B_+^t} & \Ex{ \left. \Lambda_i^{t+1} \,\right|\, \mathfrak{F}^t, \mathcal{G}^t} \\
 & \leq \sum_{i \in B_+^t} \Lambda_i^t \cdot \Big(1 - \alpha w_+ \cdot \Bigl(\frac{1}{n} - \frac{1}{n} \cdot (1 + \alpha w_+ )\Bigr) - \frac{\alpha \eps}{n} \cdot (w_- - w_+) + \left( \frac{\alpha w_-}{n} \right)^2 \Big) + 4 \\
 & = \sum_{i \in B_+^t} \Lambda_i^t \cdot \Big(1 + \frac{(\alpha w_+)^2}{n} - \frac{\alpha \eps}{n} \cdot (w_- - w_+) + \left( \frac{\alpha w_-}{n} \right)^2 \Big) + 4 \\
 & = \sum_{i \in B_+^t} \Lambda_i^t \cdot \Big(1 - \frac{\alpha}{n} \cdot \Big(\eps \cdot (w_- - w_+) - \alpha w_+^2 \Big) + \left( \frac{\alpha w_-}{n} \right)^2 \Big) + 4 \\
 & \stackrel{(a)}{\leq} \sum_{i \in B_+^t} \Lambda_i^t \cdot \Big(1 - \frac{\alpha }{n} \cdot \frac{\eps}{2} \cdot (w_- - w_+) + \left( \frac{\alpha w_-}{n} \right)^2 \Big) + 4 \\
 & \stackrel{(b)}{\leq} \sum_{i \in B_+^t} \Lambda_i^t \cdot \Big(1 - \frac{\alpha \eps}{4n} \cdot(w_- - w_+) \Big) + 4,
\end{align*}
using in $(a)$ that $\frac{\eps}{2} \cdot (w_- - w_+) \geq \alpha w_+^2$ (as implied by $\alpha \leq \frac{\eps (w_- - w_+)}{4w_-^2} \leq \frac{\eps (w_- - w_+)}{2w_+^2}$) and in $(b)$ that $\frac{\alpha \eps}{4n} \cdot(w_- - w_+) \geq \left( \frac{\alpha w_-}{n} \right)^2$ (since $\alpha = \Oh(\eps)$).

So, in both \textbf{Case 1.A} and \textbf{Case 1.B}, for the constant $\C{good_quantile_mult} >0$ defined in the statement,
\begin{align*}
\sum_{i \in B_{+}^t} \Ex{ \left. \Lambda_i^{t+1} \,\right|\, \mathfrak{F}^t, \mathcal{G}^t} 
\leq \sum_{i \in B_{+}^t} \Lambda_i^t \cdot \left(1 - \frac{2 \C{good_quantile_mult} \alpha}{n} \right) + 4. 
\end{align*}

\noindent \textbf{Case 2} [Underloaded Bins]. Among the underloaded bins, let us first consider robustly underloaded bins. By \cref{eq:change_case_2}, for any robustly underloaded bin $i \in B_{--}^t$,
\begin{align}
\Ex{ \left. \Lambda_i^{t+1} \,\right|\, \mathfrak{F}^t}
 & = \Lambda_i^t \cdot \Big(p_i^t \cdot e^{\alpha w_-/n - \alpha w_-} + (P_-^t - p_i^t) \cdot e^{\alpha w_-/n} + P_+^t \cdot e^{\alpha w_+/n} \Big). \notag
 \intertext{Applying the Taylor estimate $e^{z} \leq 1+z+z^2$, which holds for any $z \leq 1.75$, since $\frac{\alpha w_+}{n} \leq 1$ (and $\frac{\alpha w_-}{n} \leq 1$) and using that $\big( -\frac{\alpha w_-}{n} + \alpha w_-\big)^2 \leq (\alpha w_-)^2$,}
 \Ex{ \left. \Lambda_i^{t+1} \,\right|\, \mathfrak{F}^t} & \leq \Lambda_i^t \cdot \Big(1 + p_i^t \cdot \Big(\frac{\alpha w_-}{n} - \alpha w_- + (\alpha w_-)^2 \Big) \notag \\ 
 & \qquad + (P_-^t - p_i^t) \cdot \Big(\frac{\alpha w_-}{n} + \Big(\frac{\alpha w_-}{n} \Big)^2 \Big) + P_+^t \cdot \Big(\frac{\alpha w_+}{n} + \Bigl(\frac{\alpha w_+}{n} \Big)^2 \Big) \Big). \notag
 \intertext{Using that $P_-^t + P_+^t = 1$, $w_- \geq w_+$ and rearranging terms we obtain,}
\Ex{ \left. \Lambda_i^{t+1} \,\right|\, \mathfrak{F}^t} 
 & \leq \Lambda_i^t \cdot \Big(1 + p_i^t \cdot \Big(\frac{\alpha w_-}{n} - \alpha w_- + (\alpha w_-)^2 \Big) \notag \\
 & \qquad + (P_-^t - p_i^t) \cdot \frac{\alpha w_-}{n} + P_+^t \cdot \frac{\alpha w_+}{n} + \left( \frac{\alpha w_-}{n} \right)^2 \Big) \notag \\
 & = \Lambda_i^t \cdot \Big(1 - p_i^t \cdot (\alpha w_- - (\alpha w_-)^2) + P_-^t \cdot \frac{\alpha w_-}{n}  + P_+^t \cdot \frac{\alpha w_+}{n} + \left( \frac{\alpha w_-}{n} \right)^2 \Big) \notag \\
 & = \Lambda_i^t \cdot \Big(1 - p_i^t \cdot (\alpha w_- - (\alpha w_-)^2) + P_-^t \cdot \left(\frac{\alpha w_-}{n} - \frac{\alpha w_+}{n}\right) + \frac{\alpha w_+}{n} + \left( \frac{\alpha w_-}{n} \right)^2 \Big). \label{eq:robustly_negative} %
\end{align}
Now, we turn our attention to swinging underloaded bins. By \cref{lem:bins_close_to_mean} for $S := B_{+/-}^t \cap B_{-}^t$ and $\kappa_i := p_i^t \cdot (aw_- - (aw_-)^2)$,\begin{align}
\sum_{i \in S} \Ex{ \left. \Lambda_i^{t+1} \,\right|\, \mathfrak{F}^t}
 & \leq \sum_{i \in S} \Lambda_i^t \cdot ( 1 - \kappa_i ) - 2 \cdot \sum_{i = 1}^n p_i^t \cdot (aw_- - (aw_-)^2) +  3 \notag \\
 & \leq \sum_{i \in S} \Lambda_i^t \cdot ( 1 - \kappa_i ) + 4 \notag \\
 & \leq \sum_{i \in S} \Lambda_i^t \cdot \Big(1 - p_i^t \cdot (\alpha w_- - (\alpha w_-)^2) + P_-^t \cdot \left(\frac{\alpha w_-}{n} - \frac{\alpha w_+}{n}\right) \notag \\
 & \qquad \qquad + \frac{\alpha w_+}{n} + \left( \frac{\alpha w_-}{n} \right)^2 \Big) + 4, \label{eq:negative_swinging}
\end{align}
where in the last inequality we have added several non-negative terms.

Combining \eqref{eq:robustly_negative} and \eqref{eq:negative_swinging}, we get\begin{align} \label{eq:lambda_case_2}
& \sum_{i \in B_-^t} \Ex{ \left. \Lambda_i^{t+1} \,\right|\, \mathfrak{F}^t} \notag \\
& \leq \sum_{i \in B_-^t} \Lambda_i^t \cdot \Big(1 - p_i^t \cdot (\alpha w_- - (\alpha w_-)^2) + P_-^t \cdot \left(\frac{\alpha w_-}{n} - \frac{\alpha w_+}{n}\right) + \frac{\alpha w_+}{n} + \left( \frac{\alpha w_-}{n} \right)^2 \Big) + 4.
\end{align}

\noindent\textbf{Case 2.A} [\PTwo holds]. Recall that by condition \PTwo, the probabilities of the underloaded bins satisfy,
\[
\left( p_{n}^t, \ldots, p_{\delta^t n + 1}^t \right) \succeq \left( \frac{1}{n} + \frac{\C{p2k2} \cdot \delta^t}{n}, \ldots , \frac{1}{n} + \frac{\C{p2k2} \cdot \delta^t}{n} \right).
\]
Hence applying \cref{lem:quasilem} to \eqref{eq:lambda_case_2} using that $\Lambda_i^t$ are non-increasing in $i$, $(\alpha w_-)^2 \leq \alpha w_-$ (since $\alpha w_- \leq 1$) and $\delta^t \geq \eps$, gives
\begin{align*}
\sum_{i \in B_-^t} & \Ex{ \left. \Lambda_i^{t+1} \,\right|\, \mathfrak{F}^t, \mathcal{G}^t}\\
 & \stackrel{(a)}{\leq} \sum_{i \in B_-^t} \Lambda_i^t \cdot \Big(1 - \frac{1 + \C{p2k2}\eps}{n} \cdot (\alpha w_- - (\alpha w_-)^2) + 1 \cdot \left(\frac{\alpha w_-}{n} - \frac{\alpha w_+}{n}\right) + \frac{\alpha w_+}{n} + \left( \frac{\alpha w_-}{n} \right)^2 \Big) + 4 \\
 & = \sum_{i \in B_-^t} \Lambda_i^t \cdot \Big(1 - \alpha w_- \cdot \Big(\frac{1 + \C{p2k2}\eps}{n} \cdot (1 - \alpha w_-) - \frac{1}{n} \Big)  + \left( \frac{\alpha w_-}{n} \right)^2 \Big) + 4 \\
 & = \sum_{i \in B_-^t} \Lambda_i^t \cdot \Big(1 - \frac{\alpha w_-}{n} \cdot \Big((1 + \C{p2k2}\eps ) \cdot (1 - \alpha w_-) - 1 \Big) + \left( \frac{\alpha w_-}{n} \right)^2 \Big) + 4 \\
 & = \sum_{i \in B_-^t} \Lambda_i^t \cdot \Big(1 - \frac{\alpha w_-}{n} \cdot \Big(\C{p2k2}\eps - \alpha w_- \cdot (1 + \C{p2k2} \eps)\Big) + \left( \frac{\alpha w_-}{n} \right)^2 \Big) + 4 \\
 & \stackrel{(b)}{\leq} \sum_{i \in B_-^t} \Lambda_i^t \cdot \Big(1 - \frac{\alpha w_- }{n} \cdot \frac{\C{p2k2} \eps}{2} + \left( \frac{\alpha w_-}{n} \right)^2 \Big) + 4 \\ 
 & \stackrel{(c)}{\leq} \sum_{i \in B_-^t} \Lambda_i^t \cdot \Big(1 - \frac{\alpha \eps}{4n} \cdot \C{p2k2} w_- \Big) + 4,
\end{align*}
using in $(a)$ that $P_-^t \leq 1$, in $(b)$ that $\frac{\C{p2k2} \eps}{2} \geq \alpha w_- \cdot (1 + \C{p2k2} \eps)$ (as implied by $\alpha \leq \frac{\C{p2k2} \eps}{2 w_- (1 + \C{p2k2}\eps)}$) and in $(c)$ that $\frac{\alpha \eps}{4n} \cdot \C{p2k2} w_- \geq \left( \frac{\alpha w_-}{n} \right)^2$ (as implied by $\alpha = \Oh(\eps)$).

\medskip

\noindent \textbf{Case 2.B} [\WTwo and $P_-^t \leq 1 - \frac{\eps}{2}$ hold]. By condition \POne, we have that
\[
\left( p_{n}^t, \ldots, p_{\delta^t n + 1}^t \right) \succeq \left( \frac{1}{n}, \ldots , \frac{1}{n} \right),
\]
and hence applying \cref{lem:quasilem} to \eqref{eq:lambda_case_2} using that $\Lambda_i^t$ are non-increasing in $i$ and that $(\alpha w_-)^2 \leq \alpha w_-$,
\begin{align*}
& \sum_{i \in B_-^t} \Ex{ \left. \Lambda_i^{t+1} \,\right|\, \mathfrak{F}^t, \mathcal{G}^t} \\
 & \ \ \leq \sum_{i \in B_-^t} \Lambda_i^t \cdot \Big(1 - \frac{1}{n} \cdot (\alpha w_- - (\alpha w_-)^2) + \left(1 - \frac{\eps}{2} \right) \cdot \Big(\frac{\alpha w_-}{n} - \frac{\alpha w_+}{n}\Big) + \frac{\alpha w_+}{n} + \left( \frac{\alpha w_-}{n} \right)^2 \Big) + 4 \\
 & \ \  = \sum_{i \in B_-^t} \Lambda_i^t \cdot \Big(1 + \frac{(\alpha w_-)^2}{n} - \frac{\eps}{2} \cdot \Big(\frac{\alpha w_-}{n} - \frac{\alpha w_+}{n}\Big) + \left( \frac{\alpha w_-}{n} \right)^2 \Big) + 4 \\
 & \ \ = \sum_{i \in B_-^t} \Lambda_i^t \cdot \Big(1 - \frac{\alpha}{n} \cdot \Big( \frac{\eps}{2} \cdot (w_- - w_+) - \alpha w_-^2 \Big) + \left( \frac{\alpha w_-}{n} \right)^2 \Big) + 4\\
 & \ \ \stackrel{(a)}{\leq} \sum_{i \in B_-^t} \Lambda_i^t \cdot \Big(1 - \frac{\alpha}{n} \cdot \frac{\eps}{4} \cdot (w_- - w_+) + \left( \frac{\alpha w_-}{n} \right)^2 \Big) + 4 \\
 & \ \ \stackrel{(b)}{\leq} \sum_{i \in B_-^t} \Lambda_i^t \cdot \Big(1 - \frac{\alpha \eps}{8n} \cdot (w_- - w_+) \Big) + 4.
\end{align*}
using in $(a)$ that $\frac{\eps}{4} \cdot (w_- - w_+) \geq \alpha w_-^2$ (as implied by $\alpha \leq \frac{\eps( w_- - w_+)}{4 w_-^2}$) and in $(b)$ that $\frac{\alpha \eps}{8n} \cdot (w_- - w_+) \geq \left( \frac{\alpha w_-}{n} \right)^2$ (as implied by $\alpha = \Oh(\eps)$).

\medskip 

\noindent\textbf{Case 2.C} [\WTwo and $P_-^t > 1 - \frac{\eps}{2}$ hold]. By condition \WTwo, the probabilities of the underloaded bins satisfy,
\[
  \left( p_{n}^t, \ldots, p_{\delta^t n + 1}^t \right) \succeq \left( \frac{P_-^t}{|B_-^t|}, \ldots , \frac{P_-^t}{|B_-^t|} \right).
\]
Using \cref{clm:eps_ineq} and that $P_-^t > 1 - \frac{\eps}{2}$ and $|B_-^t| \leq (1 - \eps) \cdot n$, we have that\[
 \overline{p}_i^t := \frac{P_-^t}{|B_-^t|} > \frac{1 - \frac{\eps}{2}}{(1 - \eps)n} > \frac{1 + \frac{\eps}{2}}{n}.
\]
Hence, applying \cref{lem:quasilem} to \eqref{eq:lambda_case_2} using that $\Lambda_i$ are non-increasing in $i$ and $(\alpha w_-)^2 \leq \alpha w_-$,
\begin{align*}
\sum_{i \in B_{-}^t} & \Ex{\left. \Lambda_i^{t+1} \,\right|\, \mathfrak{F}^t, \mathcal{G}^t} \\
 & \leq \sum_{i \in B_{-}^t} \Lambda_i^t \cdot \Big(1 - \overline{p}_i^t \cdot (\alpha w_- - (\alpha w_-)^2) + 1 \cdot \Big(\frac{\alpha w_-}{n} - \frac{\alpha w_+}{n} \Big) + \frac{\alpha w_+}{n} + \left( \frac{\alpha w_-}{n} \right)^2 \Big) + 4 \\
 & = \sum_{i \in B_{-}^t} \Lambda_i^t \cdot \Big(1 - \overline{p}_i^t \cdot (\alpha w_- - (\alpha w_-)^2) + \frac{\alpha w_-}{n} + \left( \frac{\alpha w_-}{n} \right)^2 \Big) + 4 \\
 & = \sum_{i \in B_{-}^t} \Lambda_i^t \cdot \Big(1 - \alpha w_- \cdot \Big(\overline{p}_i^t \cdot (1 - \alpha w_-) - \frac{1}{n}\Big) + \left( \frac{\alpha w_-}{n} \right)^2 \Big) + 4 \\
 & \leq \sum_{i \in B_{-}^t} \Lambda_i^t \cdot \Big(1 - \frac{\alpha w_-}{n} \cdot \Big( \Big(1 + \frac{\eps}{2}\Big) \cdot (1 - \alpha w_-) - 1 \Big) + \left( \frac{\alpha w_-}{n} \right)^2\Big) + 4 \\ 
 & = \sum_{i \in B_{-}^t} \Lambda_i^t \cdot \Big(1 - \frac{\alpha w_-}{n} \cdot \Big(\frac{\eps}{2} - \alpha w_- \cdot \Big(1 + \frac{\eps}{2} \Big)\Big) + \left( \frac{\alpha w_-}{n} \right)^2\Big) + 4 \\ 
 & \stackrel{(a)}{\leq} \sum_{i \in B_{-}^t} \Lambda_i^t \cdot \left(1 - \frac{\alpha w_-}{n} \cdot \frac{\eps}{4} + \left( \frac{\alpha w_-}{n} \right)^2\right) + 4 \\
 & \stackrel{(b)}{\leq} \sum_{i \in B_{-}^t} \Lambda_i^t \cdot \left(1 - \frac{\alpha \eps}{8n} \cdot w_- \right) + 4,
\end{align*}
using in $(a)$ that $\frac{\eps}{4} \geq \alpha w_- \cdot (1 + \frac{\eps}{2})$ (which is equivalent to $\alpha \leq \frac{\eps}{2 w_- (2+\eps)}$) and in $(b)$ that $\frac{\alpha \eps}{8n} \cdot w_- \geq \left( \frac{\alpha w_-}{n} \right)^2$ (as implied by $\alpha = \Oh(\eps)$).

Hence, in all subcases of \textbf{Case 2}, for $\C{good_quantile_mult} > 0$ as defined in the statement,
\[
\sum_{i \in B_{-}^t} \Ex{ \left. \Lambda_i^{t+1} \,\right|\, \mathfrak{F}^t } \leq \sum_{i \in B_{-}^t} \Lambda^t \cdot \left(1 - \frac{2 \C{good_quantile_mult} \alpha}{n} \right) + 4.
\]
Finally, combining the \textbf{Case 1} and \textbf{Case 2}, we get the conclusion. 
\end{proof}

Now, we will prove a loose upper bound on the expected increase of $\Lambda$ over one round, which holds at an arbitrary round $t$, i.e., even when $\mathcal{G}^t$ does not hold.

\begin{lem} \label{lem:bad_quantile_increase_bound}
Consider any $\PTwo \cap \WOne$-process or $\POne \cap 
\WTwo$-process and the potential $\Lambda := \Lambda(\alpha)$ with any $\alpha > 0$ satisfying the preconditions of \cref{lem:good_quantile_good_decrease}. Then, for the constant $\C{bad_quantile_mult} := \max\{ 8, 4w_-^2 \} > 0$ and for any round $t \geq 0$,
\[
\Ex{\left. \Lambda^{t+1} \,\right|\, \mathfrak{F}^t} \leq \Lambda^t \cdot \left(1 + \frac{\C{bad_quantile_mult} \alpha^2}{2n} \right) + \C{bad_quantile_mult}.
\]
\end{lem}
\begin{proof}

\noindent \textbf{Case 1} [Overloaded Bins]. By \eqref{eq:lambda_case_1} in \textbf{Case 1} of \cref{lem:good_quantile_good_decrease}, 
\begin{align*}
\sum_{i \in B_+^t} & \Ex{ \left. \Lambda_i^{t+1} \,\right|\, \mathfrak{F}^t} \\
 & \leq \sum_{i \in B_+^t} \Lambda_i^t \cdot \Big(1 + p_i^t \cdot (\alpha w_+ + (\alpha w_+)^2) + P_+^t \cdot \Big(\frac{\alpha w_-}{n} - \frac{\alpha w_+}{n}\Big) - \frac{\alpha w_-}{n} + \left( \frac{\alpha w_-}{n} \right)^2 \Big) + 4. \\
 \intertext{Using the majorization in condition \POne and \cref{lem:quasilem}, and that $P_+^t \leq 1$,}
 & \leq \sum_{i \in B_+^t} \Lambda_i^t \cdot \Big(1 + \frac{1}{n} \cdot \Big(\alpha w_+ + (\alpha w_+)^2\Big) + \Big(\frac{\alpha w_-}{n} - \frac{\alpha w_+}{n} \Big) - \frac{\alpha w_-}{n} + \left( \frac{\alpha w_-}{n} \right)^2 \Big) + 4 \\
 & = \sum_{i \in B_+^t} \Lambda_i^t \cdot \Big(1 + \frac{1}{n} \cdot (\alpha w_+)^2 + \left( \frac{\alpha w_-}{n} \right)^2 \Big) + 4 \leq \Lambda_i^t \cdot \left(1 + \frac{2\alpha^2 w_+^2}{n}\right) + 4,
\end{align*}
using in the last step that $w_-$ is a constant.

\medskip

\noindent \textbf{Case 2} [Underloaded Bins]. By \eqref{eq:lambda_case_2} in \textbf{Case 2} of \cref{lem:good_quantile_good_decrease},
\begin{align*}
\sum_{i \in B_-^t} & \Ex{ \left. \Lambda_i^{t+1} \,\right|\, \mathfrak{F}^t} \\
  & \leq \sum_{i \in B_-^t} \Lambda_i^t \cdot \Big(1 - p_i^t \cdot (\alpha w_- - (\alpha w_-)^2) + P_-^t \cdot \left(\frac{\alpha w_-}{n} - \frac{\alpha w_+}{n}\right) + \frac{\alpha w_+}{n} + \left( \frac{\alpha w_-}{n} \right)^2 \Big) + 4 \\
 \intertext{Using the majorization in condition \POne and \cref{lem:quasilem}, and that $P_-^t \leq 1$,}
 & \leq \sum_{i \in B_-^t} \Lambda_i^t \cdot \Big(1 + \frac{1}{n} \cdot (-\alpha w_- + (\alpha w_-)^2) + \Big(\frac{\alpha w_-}{n} - \frac{\alpha w_+}{n}\Big) + \frac{\alpha w_+}{n} + \left( \frac{\alpha w_-}{n} \right)^2 \Big) + 4 \\
 & = \sum_{i \in B_-^t} \Lambda_i^t \cdot \Big(1 + \frac{1}{n} \cdot (\alpha w_-)^2 + \left( \frac{\alpha w_-}{n} \right)^2 \Big) + 4 \leq \Lambda_i^t \cdot \left(1 + \frac{2\alpha^2 w_-^2}{n}\right) + 4,
\end{align*}
using in the last step that $w_+$ and $w_-$ are constants.

Thus, combining the two cases and choosing $\C{bad_quantile_mult} := \max\{ 8, 4w_-^2 \}$, gives the result.
\end{proof}

We now combine the statements (and constants) from Lemmas \ref{lem:good_quantile_good_decrease} and \ref{lem:bad_quantile_increase_bound} into a single corollary for $\Lambda := \Lambda(\alpha)$ with constant smoothing parameter $\alpha$. 

\begin{cor} \label{cor:change_for_large_lambda}
Consider any $\PTwo \cap \WOne$-process or $\POne \cap 
\WTwo$-process, any constant $\eps \in \big(0, \frac{1}{2}\big)$ and the potential $\Lambda := \Lambda(\alpha)$ with $\alpha := \alpha(\eps) > 0$ satisfying the preconditions of \cref{lem:good_quantile_good_decrease}. Further, let $\C{good_quantile_mult} := \C{good_quantile_mult}(\eps) > 0$ be as defined in \cref{lem:good_quantile_good_decrease} and $\C{bad_quantile_mult} := \C{bad_quantile_mult}(\eps) > 0$ as defined in \cref{lem:bad_quantile_increase_bound}. Then, for $\C{lambda_bound} := \max\big\{\frac{8}{\C{good_quantile_mult}\alpha},\frac{2}{\alpha^2}\big\} \geq 1$, $(i)$~for any round $t \geq 0$,
\[
\Ex{\left. \Lambda^{t+1} \,\right|\, \mathfrak{F}^t, \mathcal{G}^t, \Lambda^t > \C{lambda_bound}n} \leq \Lambda^t \cdot \left(1 - \frac{\C{good_quantile_mult} \alpha}{n}\right).
\]
Further, $(ii)$~for any round $t \geq 0$,
\[
\Ex{\left. \Lambda^{t+1} \,\right|\,  \mathfrak{F}^t, \Lambda^t > cn} \leq \Lambda^t \cdot \left(1 + \frac{\C{bad_quantile_mult}\alpha^2}{n}\right).
\]
\end{cor}

\begin{proof}
\textit{First statement}. Using \cref{lem:good_quantile_good_decrease},  we have that \begin{align*}\Ex{\left. \Lambda^{t+1} \,\right|\, \mathfrak{F}^t, \mathcal{G}^t, \Lambda^t > \C{lambda_bound} n} 
& \leq \Lambda^t \cdot \left(1 - \frac{2\C{good_quantile_mult} \alpha}{n}\right) + 8 = \Lambda^t \cdot \left(1 - \frac{\C{good_quantile_mult} \alpha}{n}\right) + \left(8 - \Lambda^t \cdot \frac{\C{good_quantile_mult} \alpha}{n}\right)   \\
& \leq \Lambda^t \cdot \left(1 - \frac{\C{good_quantile_mult} \alpha}{n}\right).
\end{align*}
\textit{Second statement}. Using \cref{lem:bad_quantile_increase_bound}, we have that 
\[
\Ex{\left. \Lambda^{t+1} \,\right|\, \mathfrak{F}^t, \Lambda^t > \C{lambda_bound} n} 
\leq \Lambda^t \cdot \left(1 + \frac{\C{bad_quantile_mult}\alpha^2}{2n}\right) + \C{bad_quantile_mult}
\leq \Lambda^t \cdot \left(1 + \frac{\C{bad_quantile_mult}\alpha^2}{2n}\right) + \Lambda^t \cdot \frac{\C{bad_quantile_mult} \alpha^2}{2n}
= \Lambda^t \cdot \left(1 + \frac{\C{bad_quantile_mult}\alpha^2}{n}\right),
\] 
having used that $\Lambda^t > cn \geq \frac{2}{\alpha^2} \cdot n$.
\end{proof}

The next upper bound we use in the application of the bounded difference inequality in the proof of \cref{lem:many_good_quantiles_whp}.
\begin{lem} \label{lem:small_change_for_linear_lambda}
Consider any $\PTwo \cap \WOne$-process or $\POne \cap \WTwo$-process. Then, for any constant $\kappa >0$, and for any rounds $t_0, t_1 \geq 0$ such that $t_0 \leq t_1 \leq t_0 + \kappa \cdot n \log n$,
\[
\Pro{\max_{t \in [t_0,t_1]}\max_{i \in [n]} | y_i^{t} | \leq \log^2 n \, \Big| \, \mathfrak{F}^{t_0},\; \Lambda^{t_0} \leq n^2} \geq 1 - n^{-12}.
\]
\end{lem}
\begin{proof}
Consider the sequence $(\ex{\Lambda^t \mid \mathfrak{F}^{t_0}, \Lambda^{t_0} \leq n^2 })_{t = t_0}^{t_1}$. By  \cref{lem:bad_quantile_increase_bound} for every $t \in [t_0, t_1]$,
\[
\Ex{\left. \Lambda^{t+1} \,\right|\, \mathfrak{F}^{t}} \leq  \Lambda^t \cdot \left( 1 + \frac{\C{bad_quantile_mult} \alpha^2}{2n}\right) + \C{bad_quantile_mult}.
\]
Using the tower law of expectation, we have that
\[
\ex{\left. \Lambda^{t+1} \,\right|\, \mathfrak{F}^{t_0}, \Lambda^{t_0} \leq n^2} 
 \leq  \Ex{\left. \Lambda^t \,\right|\, \mathfrak{F}^{t_0}, \Lambda^{t_0} \leq n^2} \cdot \left( 1 + \frac{\C{bad_quantile_mult} \alpha^2}{2n}\right) + \C{bad_quantile_mult}.
\]
Hence, applying \cref{lem:geometric_arithmetic}~$(i)$ with $a := 1 + \frac{\C{bad_quantile_mult} \alpha^2}{2n} > 1$ and $b := \C{bad_quantile_mult} > 0$, we get that for any $t \in [t_0, t_1]$
\begin{align*}
\Ex{\left. \Lambda^{t} \,\right|\, \mathfrak{F}^{t_0}, \Lambda^{t_0} \leq n^2} 
 & \leq \Lambda^{t_0} \cdot a^{t_1 - t_0} + b \cdot \sum_{s = t_0}^{t-1} a^{s-t_0} \\
 & \leq n^2 \cdot \left(1 + \frac{\C{bad_quantile_mult} \alpha^2}{2n}\right)^{t-t_0} + \C{bad_quantile_mult} \cdot (t - t_0) \cdot \left(1 + \frac{\C{bad_quantile_mult} \alpha^2}{2n}\right)^{t_1-t_0} \\
 & \stackrel{(a)}{\leq} n^2 \cdot n^{\kappa \cdot \alpha^2 \cdot \C{bad_quantile_mult}} + \C{bad_quantile_mult} \cdot (\kappa n \log n) \cdot n^{\kappa \cdot \alpha^2 \cdot \C{bad_quantile_mult}}
 \leq n^{3 + \kappa \cdot \alpha^2 \cdot \C{bad_quantile_mult}},
\end{align*}
where in $(a)$ we used that $1 + z \leq e^z$ for any $z$ and that $t \leq t_1 \leq t_0 + \kappa n \log n$.

Using Markov's inequality, $\Pro{\Lambda^{t} \leq n^{3 + \kappa \cdot \alpha^2 \cdot \C{bad_quantile_mult} + 14} ~\big|~ \mathfrak{F}^{t_0}, \Lambda^{t_0} \leq n^{2}} \geq 1 - n^{-14}$ for any $t\in [t_0, t_1]$, which implies
\[
 \Pro{ \left. \max_{i \in [n]} \left|y_i^{t}\right| \leq \frac{1}{\alpha} \cdot (\kappa \cdot \alpha^2 \cdot \C{bad_quantile_mult} + 17) \cdot \log n  ~\right|~ \mathfrak{F}^{t_0}, \;\Lambda^{t_0} \leq n^2 } \geq 1 - n^{-14}.
\]
Since $\frac{1}{\alpha} \cdot (\kappa \cdot \alpha^2 \cdot \C{bad_quantile_mult} + 17) \cdot \log n < \log^2 n$ for sufficiently large $n$, by taking a union bound over all rounds $t \in [t_0, t_1]$ we get the claim.
\end{proof}

We now turn our attention to the potential $V := V(\tilde{\alpha})$ with $\tilde{\alpha} = \Theta(1/n)$ and show that $\Ex{V^m} = \poly(n)$ which in turn implies the $\Oh(n \log n)$ bound on the absolute value of the normalized loads. This will be used in the proof of \cref{thm:main_technical} as the starting point of the ``recovery phase'' (\cref{lem:recovery}).

\begin{lem} \label{lem:initial_gap_nlogn}
Consider any $\PTwo \cap \WOne$-process or $\POne \cap \WTwo$-process. Then, there exists a constant $\C{poly_n_gap} > 0$ such that for any round $t \geq 0$,
\[
\Pro{\max_{i \in [n]} \left| y_i^t \right| \leq \C{poly_n_gap} n \log n} \geq 1 - n^{-12}.
\]
Further, $(ii)$~for any two rounds $t_0 \geq 0$ and $t_1 \geq t_0$,
\[
\Pro{\left. \max_{i \in [n]} \left| y_i^{t_1} \right| \leq 2\C{poly_n_gap} n \log n \,\right|\, \max_{i \in [n]} \left| y_i^{t_0} \right| \leq \C{poly_n_gap} n \log n} \geq 1 - n^{-12}.
\]
\end{lem}
\begin{proof}
\textit{First statement.} We set $\tilde{\eps} := \frac{1}{n}$. For any $\PTwo \cap \WOne$-process, we define \[
\tilde{\alpha} := \min\left\lbrace\frac{1}{4w_-}, \frac{\C{p2k2} \tilde{\eps}}{2 w_- (1 + \C{p2k2}\tilde{\eps})}, \frac{\C{p2k1} \tilde{\eps}}{2 w_+(1-\C{p2k1} \tilde{\eps})} \right\rbrace \quad \text{ and } \quad \C{good_quantile_mult}:= \frac{\tilde{\eps}}{16} \cdot \min\{ w_+ \C{p2k1}, w_- \C{p2k2} \},
\]
and for any $\POne \cap \WTwo$-process, we define\[
\tilde{\alpha} := \min\left\lbrace\frac{1}{4w_-}, \frac{\tilde{\eps}( w_- - w_+)}{4 w_-^2}, \frac{\tilde{\eps}}{2 w_- (2+\tilde{\eps})} \right\rbrace 
\quad \text{ and } \quad \C{good_quantile_mult}:= \frac{\tilde{\eps}}{16} \cdot (w_- - w_+).
\]
So, in both cases the smoothing parameter $\tilde{\alpha}$ satisfies the preconditions \eqref{eq:c_3alphacond1} and \eqref{eq:c_3alphacond2} respectively of \cref{lem:good_quantile_good_decrease}, and so the potential $V := V(\tilde{\alpha})$ satisfies for some constant $\C{good_quantile_mult}' := \C{good_quantile_mult}/\tilde{\eps} > 0$ that 
\begin{align*} %
\Ex{ \left. V^{t+1} \,\right|\, \mathfrak{F}^t, \delta^t \in \left[\frac{1}{n}, 1 - \frac{1}{n} \right] } \leq V^t \cdot \left(1 - \frac{2 \C{good_quantile_mult}' \tilde{\alpha} \tilde{\eps}}{n} \right) + 8.
\end{align*}
When $\delta^t = 1$, all bins have load equal to the mean load, so they are all swinging bins. So, by \cref{lem:bins_close_to_mean}, we have that\[
\Ex{ \left. V^{t+1} \,\right|\, \mathfrak{F}^t, \delta^t = 1 } \leq V^t \cdot \left(1 - \frac{2 \C{good_quantile_mult}' \tilde{\alpha} \tilde{\eps}}{n} \right) + 2 \C{good_quantile_mult}' \tilde{\alpha} \tilde{\eps} + 3 \leq V^t \cdot \left(1 - \frac{2 \C{good_quantile_mult}' \tilde{\alpha} \tilde{\eps}}{n} \right) + 8,
\] 
using that $2 \C{good_quantile_mult}' \tilde{\alpha} \tilde{\eps} \leq 1$.

Combining the two cases, we have that 
\begin{align}
\Ex{ \left. V^{t+1} \,\right|\, \mathfrak{F}^t}
  & \leq \max\left\{ \Ex{  V^{t+1} \,\left|\, \mathfrak{F}^t, \delta^t \in \left[\frac{1}{n}, 1 - \frac{1}{n} \right] \right. }, \Ex{ \left. V^{t+1} \,\right|\, \mathfrak{F}^t, \delta^t = 1 } \right\} \notag \\
  & \leq V^t \cdot \left(1 - \frac{2 \C{good_quantile_mult}' \tilde{\alpha} \tilde{\eps}}{n} \right) + 8. \label{eq:drop_rough_potential}
\end{align}

By \cref{lem:geometric_arithmetic}~$(iii)$, and since $V^0 =n$ holds deterministically, it follows that for any round $t \geq 0$,\[
\Ex{V^t} \leq \frac{4}{\C{good_quantile_mult}' \tilde{\alpha} \tilde{\eps}} \cdot n \leq \kappa \cdot n^3,
\]
for some constant $\kappa > 0$, since $\tilde{\alpha} = \Theta\big(\frac{1}{n}\big)$ and $\tilde{\eps} = \frac{1}{n}$.

By Markov's inequality, we have $\Pro{V^t \leq \kappa \cdot n^{15}}  \geq 1 - n^{-12}$, which also implies that
\begin{align*}
 \Pro{\max_{i\in [n]} \left| y_i^t \right| \leq  \frac{1}{\tilde{\alpha}} \log \kappa + \frac{15}{\tilde{\alpha}} \cdot \log n } \geq 1 - n^{-12}. 
\end{align*} 
Thus, as $\tilde{\alpha}=\Theta(1/n)$, we have $\Pro{\max_{i \in [n]} \left| y_i^t \right| \leq \C{poly_n_gap} n \log n } \geq 1 - n^{-12}$ for some constant $\C{poly_n_gap} := \frac{16\tilde{\eps}}{\tilde{\alpha}}$.

\textit{Second statement.} Consider an arbitrary round $t_0 \geq 0$, where $\max_{i \in [n]} \left| y_i^{t_0} \right| \leq \C{poly_n_gap} n \log n$. Then, we also have that \[
V^{t_0} \leq n \cdot e^{\tilde{\alpha} \C{poly_n_gap} n \log n}.
\]
Recalling~\cref{eq:drop_rough_potential}, and applying \cref{lem:geometric_arithmetic}~$(ii)$ with $a := 1 - \frac{2 \C{good_quantile_mult}' \tilde{\alpha} \tilde{\eps}}{n}$ and $b := 8$, we obtain 
\[
V^{t_1} \leq V^{t_0} \cdot \left(1 - \frac{2 \C{good_quantile_mult}' \tilde{\alpha} \tilde{\eps}}{n} \right)^{t_1 - t_0} + \kappa \cdot n^3 \leq 2n \cdot e^{\tilde{\alpha} \C{poly_n_gap} n \log n},
\]
using that $\C{poly_n_gap} > \frac{15}{\tilde{\alpha}}$. By  Markov's inequality,
\begin{align*}
 & \Pro{\left. V^{t_1} \leq n^{12} \cdot 2n \cdot e^{\tilde{\alpha} \C{poly_n_gap} n \log n} \,\right|\, \max_{i \in [n]} \left| y_i^{t_0} \right| \leq \C{poly_n_gap} n \log n} \\
 & \qquad \geq \Pro{\left. V^{t_1} \leq e^{2 \tilde{\alpha} \C{poly_n_gap} n \log n} \,\right|\, \max_{i \in [n]} \left| y_i^{t_0} \right| \leq \C{poly_n_gap} n \log n}
   \geq 1 - n^{12}.
\end{align*}
Finally, when $\{ V^{t_1} \leq e^{2 \tilde{\alpha} \C{poly_n_gap} n \log n} \}$ holds, then \[
\max_{i \in [n]} \left| y_i^{t_1}\right| \leq \frac{1}{\tilde{\alpha}} \cdot \log V^{t_1} \leq 2 \C{poly_n_gap} n \log n. \qedhere
\]
\end{proof}

\subsection{Deterministic Bounds on the Quadratic Potential} \label{sec:quadratic_and_exp_potentials}

The next lemma provides a bound for the quadratic potential in terms of the exponential potential $\Lambda$ and holds for an arbitrary process.
\begin{lem} \label{lem:lambda_bound_implies_upsilon_bound}
Consider an arbitrary load vector $x^t$ and the potentials $\Upsilon^t := \Upsilon^t(x^t)$ and $\Lambda^t := \Lambda(\alpha, x^t)$ for any $\alpha \in (0, 1]$. Then,
\[
\Upsilon^t \leq \left(\frac{4}{\alpha} \cdot \log\left(\frac{4}{\alpha}\right)\right)^2\cdot \Lambda^t.
\]
\end{lem}
\begin{proof}
Let $\kappa := (4/\alpha) \cdot \log(4/\alpha)$. For any $y \geq \kappa$, it follows that
\[
e^{\alpha y/2} = e^{\alpha y/4} \cdot e^{\alpha y/4} \stackrel{(a)}{\geq} \frac{\alpha y}{4} \cdot e^{\alpha y/4} \stackrel{(b)}{\geq} \frac{\alpha y}{4} \cdot \frac{4}{\alpha} \geq y,
\]
using in $(a)$ that $e^y \geq y$ (for any $y \geq 0$) and in $(b)$ that $\frac{\alpha y}{4} \geq \log(4/\alpha)$ (as $y \geq \kappa$). Hence for $y \geq \kappa$,
\[
e^{\alpha y} = e^{\alpha y/2} \cdot e^{\alpha y/2} \geq y \cdot y = y^2.
\]
Thus, we conclude
\begin{align*}
 \Upsilon^t &= \sum_{i=1}^n ( y_i^t )^2 
  \leq \sum_{i=1}^n \max \left\{ \Lambda_i^{t},  \kappa^2 \right\} 
  \stackrel{(a)}{\leq} \sum_{i=1}^n \max \left\{ \Lambda_i^{t},  \kappa^2 \cdot \Lambda_i^{t} \right\} 
  \stackrel{(b)}{=} \kappa^2 \cdot \Lambda^t,
\end{align*} 
using in $(a)$ that $\Lambda_i^{t} \geq 1$ for any $i \in [n]$ and in $(b)$ that $\kappa \geq 1$.
\end{proof}

The next lemma is very basic but is used in \cref{lem:many_good_quantiles_whp}, so we prove it for completeness. 
\begin{lem}\label{lem:basic}
Consider any $\WOne$-process. Then, for any round $t\geq 0$ \[
\left|\Upsilon^{t+1}-\Upsilon^{t}\right| 
 \leq 4 w_-\cdot \max_{i \in [n]} |y_i^t| + 2w_-^2.\]
	\end{lem}
\begin{proof}
Recall that $w^{t+1} := W^{t+1}-W^{t}$ is the number of balls allocated in the $(t+1)$-th allocation. We will upper bound the change in the quadratic potential using\[
  \left| \Upsilon^{t+1} - \Upsilon^t \right| \leq \sum_{i = 1}^n \left| \Upsilon_i^{t+1} - \Upsilon_i^t \right|,
\]
where $\Upsilon_i^t := (y_i^t)^2$ and $\Upsilon_i^{t+1} := (y_i^{t+1})^2$. Now, let $i \in [n]$ be an arbitrary bin. We bound $|\Upsilon_i^{t+1} - \Upsilon_i^t|$ by considering two cases based on whether bin $i$ was allocated to in the $(t+1)$-th allocation or not:

\textbf{Case 1:} If $i$ was allocated to, then $y_{i+1}^t = y_i^t + w^{t+1} \cdot \big( 1 - \frac{1}{n} \big)$ and so
\begin{align*}
\left|\Upsilon_i^{t+1} - \Upsilon_i^t \right| 
 & = \left|\left(y_{i}^t + w^{t+1} \cdot \Bigl( 1 - \frac{1}{n} \Bigr) \right)^2 -  \left( y_i^t \right)^2\right| \\
 & \leq 2 \cdot |y_i^t| \cdot \left|w^{t+1} \cdot \Bigl( 1 - \frac{1}{n}\Bigr) \right| + \left| w^{t+1} \cdot \Bigl( 1 - \frac{1}{n}\Bigr) \right|^2 \\
 & \leq 2 \cdot \max_{j \in [n]} |y_j^t|\cdot w_- + w_-^2.  \end{align*}

\textbf{Case 2:} If $i$ was not allocated to, then $y_{i+1}^t = y_i^t - \frac{w^{t+1}}{n}$ and so
 \begin{align*}
 \left|\Upsilon_i^{t+1} - \Upsilon_i^t \right| 
  = \left|  \left( y_i^{t} -\frac{w^{t+1}}{n}   \right)^2 -  \left( y_i^{t}  \right)^2\right| 
  \leq 2 \cdot \left|   \frac{w^{t+1}}{n}\right|\cdot |y_{i}^t| + \left|\frac{w^{t+1}}{n}\right|^2  
  \leq \frac{2w_-}{n }\cdot  \max_{j \in [n]} |y_j^t| + \frac{w_-^2}{n^2}, 
\end{align*}
using that $w^{t+1} \leq w_-$. 

By aggregating over all bins using the two cases above, we get the conclusion.
\end{proof}

\section{Proof of the Bound on the Gap} \label{sec:mean_biased_gap_completion}
 
In this section, we complete the proof that any $\PTwo \cap \WOne$ or $\POne \cap \WTwo$-process satisfies $\Gap(m) = \Oh(\log n)$ \Whp~at an arbitrary round $m$. As outlined in \cref{sec:analysis_overview}, the proof consists of a recovery phase and a stabilization phase.  \cref{fig:recovery_stabilisation} depicts these two phases. We start by specifying the constants used in the proof.

Let $\eps := \eps(C) \in \big(0, \frac{1}{2} \big)$ be as defined in \cref{lem:good_quantile} for 
\begin{align} \label{eq:c_def}
C := \left\lceil\frac{8\C{quad_const_add}}{\C{quad_delta_drop}}\right\rceil+1,    
\end{align}
where $\C{quad_delta_drop}, \C{quad_const_add} > 0$ are as defined in \cref{lem:quadratic_absolute_relation}. We will be using the potential $\Lambda := \Lambda(\alpha)$ with smoothing parameter $\alpha > 0$ defined as 
\begin{align} \label{eq:alpha_def}
  \alpha := \min \left\{ \frac{\C{good_quantile_mult} \eps^2}{2\C{bad_quantile_mult} (1 - \eps^2)}, \; \alpha' \right\},
\end{align}
where $\C{good_quantile_mult} := \C{good_quantile_mult}(\eps) > 0$ is as defined in \cref{lem:good_quantile_good_decrease}, $\C{bad_quantile_mult} := \C{bad_quantile_mult}(\eps) > 0$ as defined in \cref{lem:bad_quantile_increase_bound}, and  for $\PTwo \cap \WOne$ processes \[
  \alpha' := \min\left\lbrace\frac{1}{4w_-}, \; \frac{\C{p2k2} \eps}{2 w_- (1 + \C{p2k2}\eps)}, \; \frac{\C{p2k1} \eps}{2 w_+(1-\C{p2k1} \eps)} \right\rbrace,
\]
and for $\POne \cap \WTwo$ processes\[
  \alpha' := \min\left\lbrace\frac{1}{4w_-}, \; \frac{\eps( w_- - w_+)}{4 w_-^2}, \; \frac{\eps}{2 w_- (2+\eps)} \right\rbrace.
\]

In the \textit{recovery phase} (\cref{sec:recovery}), we prove that starting at $t_0 = m - n^3 \log^4 n$ with $V^{t_0} = \poly(n)$ (i.e., $\Lambda^{t_0} \leq \exp(2 \C{poly_n_gap} n \log n)$), \Whp~there exists a round $s \in [t_0, m]$ with $\Lambda^s \leq \C{lambda_bound} n$, for constant $c > 0$ defined in \cref{cor:change_for_large_lambda}. We do this by proving that starting with $\Lambda^{t_0} \leq \exp(2 \C{poly_n_gap} n \log n)$, then \Whp~there is a constant fraction of rounds $s \in [t_0, t_0 + n^3 \log^4 n]$ with $\delta^s \in [\eps, 1 - \eps]$. By analyzing an ``adjusted version'' $\tilde{\Lambda}$ (defined below in \cref{eq:tilde_lambda_def}) of the exponential potential $\Lambda$, taking advantage of the fact that $\Lambda^t$ decreases in expectation whenever $\delta^t \in [\eps, 1 - \eps]$ and increases at most by a smaller factor otherwise (\cref{cor:change_for_large_lambda}),  we show that there exists $s \in [t_0, t_0 +  n^3 \log^4 n]$ such that $\Lambda^s \leq cn$.

Next, in the \textit{stabilization phase} (\cref{sec:stabilization}), we first show that starting with $\Lambda^{t_0} \leq 2cn$ at any round $t_0$, we have that $\Lambda^{s} \leq cn$ for some round $s \in (t_0, t_0 + \Theta(n \log n)]$. We do this by proving that \Whp~if $\Lambda^{t_0} \leq cn$, then there is a constant fraction of rounds $s \in [t_0, t_0 + \Theta(n \log n)]$ with $\delta^s \in [\eps, 1 - \eps]$. 
Again, by analyzing an ``adjusted version'' $\tilde{\Lambda}$ of the exponential potential $\Lambda$, we show that \Whp~there exists $s \in (t_0, t_0 + \Theta(n \log n)]$ such that $\Lambda^s \leq cn$. Next, we take the union bound over the remaining $\Oh(n^3 \log^4 n)$ rounds, which gives $\Lambda^{s_1} \leq cn$ \Whp~at some $s_1 \in  [m - \Theta(n \log n), m]$ and $\Lambda^{s_2} \leq cn$ \Whp~at some $s_2 \in [m, m + \Theta(n \log n), m]$. By a smoothness argument, this in turn implies that $\max_{i \in [n]} |y_i^m| = \Oh(\log n)$. 
 
In order to complete the analysis in Sections \ref{sec:recovery} and \ref{sec:stabilization} described above, we must first give some definitions and establish some technical tools. In particular, in \cref{sec:adj_martin} we shall prove that the new adjusted exponential potential function $\tilde{\Lambda}^t$ is a super-martingale. In \cref{sec:taming} we then use the interplay between the absolute value and quadratic potential functions to prove that in a constant fraction of the rounds the mean quantile is ``good''.

In the following, we consider an arbitrary round $t_0 \geq 0$ which will be starting point of our analysis.
For $\eps\in (0,1)$ and $\C{lambda_bound} >1$ as specified above, for any round $s \geq t_0$ we define the event,\begin{equation*}%
\mathcal{E}_{t_0}^{s} := \bigcap_{t \in [t_0, s]} \{\Lambda^t > \C{lambda_bound} n\}.  \end{equation*} 
Showing that this event does not hold for suitable $t_0$ and $s$ is crucial in both recovery and stabilization phases. For rounds $s \geq t_0$, we let $G_{t_0}^{s}:=G_{t_0}^{s}(\eps)$ be the number of rounds $t \in [t_0, s]$ with $\delta^t \in [\eps, 1 - \eps]$, and similarly we define $B_{t_0}^{s} := (s - t_0+1) - G_{t_0}^s$. 

We now introduce the \textit{adjusted exponential potential function} $\tilde{\Lambda}_{t_0}^s$ which involves the random variables and event above. Let $\C{good_quantile_mult} := \C{good_quantile_mult}(\eps) > 0$ be the constant in \cref{lem:good_quantile_good_decrease} and $\C{bad_quantile_mult} := \C{bad_quantile_mult}(\eps) > 0$ the constant in \cref{lem:bad_quantile_increase_bound}. Then we define $\tilde{\Lambda}_{t_0}^{t_0}(\alpha, \eps,  \C{good_quantile_mult}, \C{bad_quantile_mult}) := \Lambda^{t_0}(\alpha)$ and, for any round $s > t_0$,
\begin{equation}\label{eq:tilde_lambda_def}
\tilde{\Lambda}_{t_0}^s(\alpha, \eps, \C{good_quantile_mult}, \C{bad_quantile_mult}) := \Lambda^s \cdot \mathbf{1}_{\mathcal{E}_{t_0}^{s-1}} \cdot \exp\left( - \frac{\C{bad_quantile_mult} \alpha^2}{n} \cdot B_{t_0}^{s-1}(\eps) \right) \cdot \exp\left( + \frac{\C{good_quantile_mult} \alpha}{n} \cdot G_{t_0}^{s-1}(\eps) \right). 
\end{equation}

We also define $\tilde{G}_{t_0}^{s} := \tilde{G}_{t_0}^{s}(C)$ to be the number of rounds $t \in [t_0, s]$ with $\Delta^t \leq Cn$. 

\subsection{The Adjusted Exponential Potential is a Super-Martingale} \label{sec:adj_martin}
We now show that the sequence defined in \eqref{eq:tilde_lambda_def} forms a super-martingale.

\begin{lem}\label{lem:lambda_tilde_is_supermartingale}
Consider any $\PTwo \cap \WOne$-process or $\POne \cap \WTwo$-process and an arbitrary round $t_0 \geq 0$. Further, let the potential $\tilde{\Lambda}_{t_0} := \tilde{\Lambda}_{t_0}(\alpha, \eps, \C{good_quantile_mult}, \C{bad_quantile_mult})$ for $\alpha > 0$ as defined in \eqref{eq:alpha_def}, $\eps>0$ as defined in \eqref{eq:c_def}, and $\C{good_quantile_mult}:=\C{good_quantile_mult}(\eps)>0$ and $\C{bad_quantile_mult}:=\C{bad_quantile_mult}(\eps) > 0$ as defined in Lemmas \ref{lem:good_quantile_good_decrease} and \ref{lem:bad_quantile_increase_bound} respectively. Then, for any round $s \geq t_0$,
\[
 \Ex{ \left. \tilde{\Lambda}_{t_0}^{s+1} \,\right|\, \mathfrak{F}^s} \leq \tilde{\Lambda}_{t_0}^{s}.
\]
\end{lem}

\begin{proof}
We see that $\ex{ \tilde{\Lambda}_{t_0}^{s+1}   \mid \mathfrak{F}^s}$ is given by 
\begin{align*}
   &\Ex{ \left. \Lambda^{s+1} \cdot \mathbf{1}_{\mathcal{E}_{t_0}^{s}} \,\right|\, \mathfrak{F}^s} \cdot \exp\bigg( - \frac{\C{bad_quantile_mult} \alpha^2}{n} \cdot B_{t_0}^{s} \bigg) \cdot \exp\bigg( \frac{\C{good_quantile_mult} \alpha}{n} \cdot G_{t_0}^{s} \bigg) \\ 
 & = \Ex{ \left. \Lambda^{s+1} \cdot \mathbf{1}_{\mathcal{E}_{t_0}^{s}} \,\right|\, \mathfrak{F}^s} \cdot \exp\left( -\frac{\C{bad_quantile_mult} \alpha^2}{n} \cdot \mathbf{1}_{\neg \mathcal{G}^s} + \frac{\C{good_quantile_mult} \alpha}{n} \cdot \mathbf{1}_{\mathcal{G}^s} \right) \cdot \exp\left( - \frac{\C{bad_quantile_mult} \alpha^2}{n} \cdot B_{t_0}^{s - 1} \right) \cdot \exp\left( \frac{\C{good_quantile_mult} \alpha}{n} \cdot G_{t_0}^{s - 1} \right).
\end{align*}
Thus, we see that it suffices to prove that \begin{equation}\label{eq:sufficient2}
\Ex{\left. \Lambda^{s+1} \cdot \mathbf{1}_{\mathcal{E}_{t_0}^{s}} \,\right|\, \mathfrak{F}^s} \cdot \exp\left( -\frac{\C{bad_quantile_mult} \alpha^2}{n} \cdot \mathbf{1}_{\neg \mathcal{G}^s} + \frac{\C{good_quantile_mult} \alpha}{n} \cdot \mathbf{1}_{\mathcal{G}^s} \right) \leq \Lambda^{s} \cdot \mathbf{1}_{\mathcal{E}_{t_0}^{s-1}}.\end{equation}To show \eqref{eq:sufficient2}, we consider two cases based on whether $\mathcal{G}^s$ holds.

\medskip 

\noindent\textbf{Case 1} [$\mathcal{G}^s$ holds]. Recall that the event $\mathcal{G}^s$ means $\delta^s\in[\eps, 1 - \eps]$ holds. Further, we are additionally conditioning on the event $\mathcal{E}_{t_0}^{s}$, via the indicator, and so $\Lambda^t>cn$ holds for any round $t \in [t_0,s]$. Thus we can use the upper bound from \cref{cor:change_for_large_lambda} to get
\[
\Ex{ \left. \Lambda^{s+1} \cdot \mathbf{1}_{\mathcal{E}_{t_0}^{s}} \,\right|\, \mathfrak{F}^s, \mathcal{G}^s} \leq \Lambda^{s} \cdot \mathbf{1}_{\mathcal{E}_{t_0}^{s-1}}  \cdot \left(1 - \frac{\C{good_quantile_mult} \alpha}{n} \right) \leq \Lambda^{s} \cdot \mathbf{1}_{\mathcal{E}_{t_0}^{s-1}} \cdot \exp\left(- \frac{\C{good_quantile_mult} \alpha}{n} \right).
\]
Hence, since in this case $\mathbf{1}_{\mathcal{G}^s}=1$, the left-hand side of \eqref{eq:sufficient2} is equal to
\begin{align*}
& \Ex{\left. \Lambda^{s+1} \cdot \mathbf{1}_{\mathcal{E}_{t_0}^{s}} \,\right|\, \mathfrak{F}^s, \mathcal{G}^s} \cdot \exp\left( -\frac{\C{bad_quantile_mult} \alpha^2}{n} \cdot \mathbf{1}_{\neg \mathcal{G}^s} + \frac{\C{good_quantile_mult} \alpha}{n} \cdot \mathbf{1}_{\mathcal{G}^s} \right) \\
 & \qquad \leq \left(\Lambda^{s} \cdot \mathbf{1}_{\mathcal{E}_{t_0}^{s-1}} \cdot \exp\left(- \frac{\C{good_quantile_mult} \alpha}{n} \right)\right) \cdot \exp\left( \frac{\C{good_quantile_mult} \alpha}{n} \right) \\ & \qquad = \Lambda^{s} \cdot \mathbf{1}_{\mathcal{E}_{t_0}^{s-1}} .
\end{align*}
\noindent\textbf{Case 2} [$\mathcal{G}^s$ does not hold]. Similarly, by \cref{cor:change_for_large_lambda}~$(ii)$. \begin{align*}
\Ex{\left. \Lambda^{s+1} \cdot \mathbf{1}_{\mathcal{E}_{t_0}^{s}} \,\right|\, \mathfrak{F}^s, \neg \mathcal{G}^s} 
 \leq \Lambda^{s} \cdot \mathbf{1}_{\mathcal{E}_{t_0}^{s-1}} \cdot \left(1 + \frac{\C{bad_quantile_mult} \alpha^2 }{n}\right) %
\leq \Lambda^{s} \cdot \mathbf{1}_{\mathcal{E}_{t_0}^{s-1}}\cdot \exp\left( \frac{\C{bad_quantile_mult} \alpha^2}{n} \right).
\end{align*}
Hence,  since in this case $\mathbf{1}_{\mathcal{G}^s}=0$, the left-hand side of \eqref{eq:sufficient2} is equal to
\begin{align*}
& \Ex{\left. \Lambda^{s+1} \cdot \mathbf{1}_{\mathcal{E}_{t_0}^{s}} \,\right|\, \mathfrak{F}^s, \neg \mathcal{G}^s} \cdot \exp\left( -\frac{\C{bad_quantile_mult} \alpha^2}{n} \cdot \mathbf{1}_{\neg \mathcal{G}^s} + \frac{\C{good_quantile_mult} \alpha}{n} \cdot \mathbf{1}_{\mathcal{G}^s} \right) \\
 &\qquad \leq \left( \Lambda^{s} \cdot \mathbf{1}_{\mathcal{E}_{t_0}^{s-1}}\cdot \exp\left( - \frac{\C{bad_quantile_mult} \alpha^2}{n} \right)\right)   \cdot \exp\left( \frac{\C{bad_quantile_mult} \alpha^2}{n} \right) \\  
 &\qquad = \Lambda^{s} \cdot \mathbf{1}_{\mathcal{E}_{t_0}^{s-1}}.
\end{align*} Since  \eqref{eq:sufficient2} holds in both cases, we deduce that $(\tilde{\Lambda}_{t_0}^s)_{s \geq t_0}$ forms a super-martingale. 
\end{proof}

\subsection{Taming the Mean Quantile and Absolute Value Potential} \label{sec:taming}
 
We have seen in the previous section that if we augment the exponential potential $\Lambda^t$ with some terms involving the random variable $G_{t_0}^s$, then the resulting potential $\tilde{\Lambda}^t$ is a super-martingale. Thus, it will be useful to control $G_{t_0}^{s}:=G_{t_0}^{s}(\eps)$, which we recall is the number of rounds $t \in [t_0, s]$ with $\delta^t \in [\eps, 1 - \eps]$. This is done in part by controlling $\tilde{G}_{t_0}^{s} := \tilde{G}_{t_0}^{s}(C)$, which we recall is the number of rounds $t \in [t_0, s]$ with $\Delta^t \leq Cn$. In particular, our first lemma shows that for an interval of length at least $\frac{n}{4\eps} $, if we have a constant fraction of round $t$ with $\Delta^t \leq Cn$, then \Whp we also have a constant fraction of rounds with $\delta^t \in [\eps, 1 - \eps]$.

\begin{lem}\label{lem:newcorrespondence}
Consider any $\PTwo \cap \WOne$-process or $\POne \cap \WTwo$-process and let $\eps>0$ as defined in \eqref{eq:c_def}. Then, for any rounds $t_0 \geq 0$ and $t_1 \geq t_0 + \frac{n}{4\eps} $,
\[\Pro{ G_{t_0}^{t_1}(\eps) \geq 4\eps ^2\cdot \tilde{G}_{t_0}^{t_1-\frac{n}{4\eps} }(C) \;\Big|\; \mathfrak{F}^{t_0}} \geq 1- (t_1-t_0) \cdot e^{-\eps n}. \]
\end{lem}

\begin{proof}  Let $\ell:=\frac{n}{4\eps} $. For any round $t \in [t_0,t_1]$, define
\[
\mathcal{H}^t := \Bigl\{ \Delta^{t} > Cn  \Bigr\} \cup \Bigl\{ \bigl| \left\{ s \in [t,t+\ell] \colon 
 \delta^{s} \in [\eps, 1 - \eps] \right\} \bigr| \geq \eps n \Bigr\}.
\]
Note that event $\mathcal{H}^t$ is logically equivalent to the statement: $\Delta^{t} \leq Cn$ implies 
\[\left| \left\{ s \in [t,t+\ell] \colon 
 \delta^{s} \in [\eps, 1 - \eps] \right\} \right| \geq \eps n.
\]
Then by \cref{lem:good_quantile},
\[
 \Pro{ \mathcal{H}^t \;\Big|\; \mathfrak{F}^{t_0}} \geq 1- e^{-\eps n},
\]
and so the union bound gives
\begin{align} \label{eq:cap_h_t}
  \Pro{ \left. \bigcap_{t=t_0}^{t_1} \mathcal{H}^{t} \;\right|\; \mathfrak{F}^{t_0}} \geq 1 - (t_1-t_0) \cdot e^{- \eps n}.
\end{align}
In the following, we will condition on the event $\bigcap_{t=t_0}^{t_1} \mathcal{H}^{t} $. Next define for any $t \in [t_0,t_1-\ell]$,
\[
  g^t := \left\{ s \in [t,t+\ell] \colon 
 \delta^{s} \in [\eps, 1 - \eps] \right\}.
\]Then, conditional on \cref{eq:cap_h_t},
\begin{align*}
 G_{t_0}^{t_1} &\geq \left|
 \bigcup_{t=t_0}^{t_1-\ell} g^t \right|  \geq \frac{\sum_{t=t_0}^{t_1-\ell} |g^t|}{ \max_{s \in [t,t+\ell]} \Bigl| \bigr\{t \in [t_0,t_1-\ell] \colon s \in g^t  \bigr\}\Bigr| }  \geq \frac{ \tilde{G}_{t_0}^{t_1-\ell} \cdot \eps \cdot n}{ \ell } \geq  \tilde{G}_{t_0}^{t_1-\frac{n}{4\eps}} \cdot 4\eps^2 ,
\end{align*}
using in the last step that $\ell:=\frac{n}{4\eps} $. This completes the proof.\end{proof}

The next lemma establishes the key fact that \Whp~when $\Upsilon^{t_0} \leq T$ there are many rounds close to $t_0$ with a good quantile (or the gap becomes too large). We will apply this lemma twice, once in the recovery phase with $T := n^3 \log^4 n$ (\cref{lem:recovery}) and once in the stabilization phase with $T := \Theta(n \log n)$ (\cref{lem:stabilization})

 \begin{lem} \label{lem:many_good_quantiles_whp}
Consider any $\PTwo \cap \WOne$-process or $\POne \cap \WTwo$-process and the potential $\Lambda := \Lambda(\alpha)$ with $\alpha>0$ as defined in \eqref{eq:alpha_def} and $\eps>0$ as defined in \eqref{eq:c_def}. Further, consider any $T \geq \frac{n}{\eps}$, any round $t_0 \geq 0$ and let \[
\tau:=\inf\Big\{ t \geq t_0 \colon \max_{i \in [n]} |y_i^{t}| > \sqrt{T} \cdot \log^{-2} n \Big\}.
\]
Then, for any round $t_1 \geq 0$ satisfying $t_0 + T \leq t_1 \leq t_0 + T \log^2 n$,
\[
\Pro{ \left\{ G_{t_0}^{t_1}(\eps) \geq \eps^2 \cdot (t_1 - t_0) \right\} \cup \{ \tau \leq t_1 \} \;\Big|\; \mathfrak{F}^{t_0}, \Upsilon^{t_0} \leq T} \geq 1 - (t_1 - t_0)  \cdot e^{-\eps n} - n^{-\omega(1)}.
\] 
\end{lem}  
\begin{proof}
Let $\C{quad_delta_drop}, \C{quad_const_add} > 0$ be the constants given in \cref{lem:quadratic_absolute_relation}. We define $Z^{t_0} := \Upsilon^{t_0}$ and for any round $t > t_0$
\[
 Z^{t} := \Upsilon^{t} - \C{quad_const_add} \cdot (t-t_0) + \frac{\C{quad_delta_drop}}{n}\sum_{s=t_0}^{t-1} \Delta^{s}.
\]
This sequence forms a super-martingale since by \cref{lem:quadratic_absolute_relation},
\begin{align*}
\Ex{\left. Z^{t+1} \,\right|\, \mathfrak{F}^{t}} 
  & = \Ex{\left. \Upsilon^{t+1} - \C{quad_const_add} \cdot (t-t_0+1) + \frac{\C{quad_delta_drop}}{n}\sum_{s=t_0}^{t} \Delta^{s} ~\right|~ \mathfrak{F}^{t}} \\
  & = \Ex{\left. \Upsilon^{t+1}\,\right|\, \mathfrak{F}^t} - \C{quad_const_add} \cdot (t-t_0+1) + \frac{\C{quad_delta_drop}}{n}\sum_{s=t_0}^{t} \Delta^{s} \\
  & \leq \Upsilon^{t} +\C{quad_const_add} - \frac{\C{quad_delta_drop}}{n} \cdot \Delta^t - \C{quad_const_add} \cdot (t-t_0+1) + \frac{\C{quad_delta_drop}}{n}\sum_{s=t_0}^{t} \Delta^{s} \\
  & = \Upsilon^t - \C{quad_const_add} \cdot (t-t_0 ) + \frac{\C{quad_delta_drop}}{n}\sum_{s=t_0}^{t-1} \Delta^{s} = Z^{t}.
\end{align*}
Further, consider the stopped random variable
\[
 \tilde{Z}^{t} := Z^{t \wedge \tau},
\]
which also forms a super-martingale. 
To prove concentration of $\tilde{Z}^{t+1}$, we will now derive an upper bound on $\left| \tilde{Z}^{t+1} - \tilde{Z}^{t} \right|$, by considering the following two cases:

\medskip

\noindent\textbf{Case 1} [$t \geq \tau$].
In this case, $\tilde{Z}^{t+1} = Z^{(t+1) \wedge \tau} = Z^{\tau}$, and similarly, $\tilde{Z}^{t} = Z^{t \wedge \tau} = Z^{\tau}$, so \[
\left| \tilde{Z}^{t+1} - \tilde{Z}^{t} \right| = 0.
\]

\noindent\textbf{Case 2} [$t < \tau$]. In this case, we have that $\max_{i \in [n]} |y_i^t| \leq \sqrt{T} \cdot \log^{-2} n$ and so\begin{align*}
 \left| \tilde{Z}^{t+1} - \tilde{Z}^{t} \right| 
 & \leq \left| \Upsilon^{t+1} - \Upsilon^t \right| + c_2 + \frac{c_1}{n} \cdot \Delta^t \\
 & \stackrel{(a)}{\leq} 5w_-\cdot \sqrt{T} \cdot \log^{-2} n + c_2 + c_1 \sqrt{T} \cdot \log^{-2} n \\
 & \leq (5w_- + 2c_1) \cdot \sqrt{T} \cdot \log^{-2} n,
\end{align*}
using in $(a)$ that $|\Upsilon^{t+1} - \Upsilon^t| \leq 4w_-\cdot \sqrt{T} \cdot \log^{-2} n + 2w_-^2\leq 5w_-\cdot \sqrt{T} \cdot \log^{-2} n$ by \cref{lem:basic}, and $\Delta^t \leq n \cdot \max_{i \in [n]} |y_i^t|$.

\medskip

Combining \textbf{Case 1} and \textbf{Case 2} above, we conclude
\[
  \left| \tilde{Z}^{t+1} - \tilde{Z}^{t} \right| \leq (5w_- + 2c_1) \cdot \sqrt{T} \cdot \log^{-2} n.
\]
Using Azuma's inequality (\cref{lem:azuma}) for super-martingales,
\begin{align*}
\Pro{ \left. \tilde{Z}^{t_1+1} - \tilde{Z}^{t_0} > T \, \right| \, \mathfrak{F}^{t_0}, \Upsilon^{t_0} \leq T } 
 &\leq \exp\left( - \frac{ T^2 }{ 2 \cdot \sum_{t=t_0}^{t_1} ((5w_- + 2c_1) \cdot \sqrt{T} \cdot \log^{-2} n)^2  } \right) \\ 
 &= \exp\left( - \frac{ T^2 }{ 2 \cdot (t_1-t_0)\cdot (5w_- + 2c_1)^2 \cdot T \cdot \log^{-4} n  } \right) \\
 &\stackrel{(a)}{\leq} \exp\left( - \frac{\log^2 n}{ 2 \cdot (5w_- + 2c_1)^2 } \right) \\
 &\stackrel{(b)}{=} n^{-\omega(1)},
\end{align*}
using in $(a)$ that $t_1 \leq t_0 + T \log^2 n$ and in $(b)$ that $w_+$ and $c_1$ are constants.

For the sake of a contradiction, assume now that at least half of the rounds $t \in [t_0,t_1]$ satisfy $\Delta^t \geq Cn$, which implies that
\[
  \sum_{t=t_0}^{t_1} \Delta^t \geq \frac{t_1-t_0+1}{2} \cdot Cn.
\]
When the event\[
  \left\{ \tilde{Z}^{t_1+1} \leq \tilde{Z}^{t_0} + T \right\} \cap \{ \tau > t_1 \}
\]
holds, then we have that $\big\{ Z^{t_1+1} \leq Z^{t_0} + T \big\}$ and so
\[
\Upsilon^{t_1+1}  - \C{quad_const_add} \cdot (t_1-t_0+1) + \frac{\C{quad_delta_drop}}{n}\sum_{t=t_0}^{t_1} \Delta^{t} \leq \Upsilon^{t_0} + T.\]
Rearranging the inequality above gives
\begin{align} \Upsilon^{t_1+1} & \leq \Upsilon^{t_0} + T + \C{quad_const_add} \cdot (t_1-t_0+1) - \frac{\C{quad_delta_drop}}{n}\sum_{t=t_0}^{t_1} \Delta^{t} \notag  \\
  & \leq \Upsilon^{t_0} + T + \C{quad_const_add} \cdot (t_1-t_0+1) - \frac{\C{quad_delta_drop}}{n} \cdot \frac{t_1 - t_0+1}{2} \cdot Cn \notag \\
  & \leq \Upsilon^{t_0} + T + (t_1 - t_0 +1) \cdot \left(\C{quad_const_add} - \frac{\C{quad_delta_drop}}{2} \cdot C\right) .\label{eq:upsilon_contradiction}
\end{align} 
Recall that we start from a round $t_0$ where $\Upsilon^{t_0} \leq T$, that $C > \frac{8c_2}{c_1}$ (and $c_2 > 1$) and $t_1 - t_0 \geq T$. So by \eqref{eq:upsilon_contradiction} we have   
\[
\Upsilon^{t_1+1} 
 \leq  T + T +  T \cdot \left(\C{quad_const_add} - \frac{\C{quad_delta_drop}}{2} \cdot C\right) 
 < 2T - 3T \cdot c_2 < 0,
\]
which is a contradiction. We conclude that when $\{ Z^{t_1+1} \leq Z^{t_0} + T \}$, then half of the rounds $t \in [t_0,t_1]$ satisfy $\Delta^{t} \leq C n$, i.e.
\begin{align}
\Pro{\left. \left\{ \tilde{G}_{t_0}^{t_1} \geq \frac{1}{2} \cdot (t_1 - t_0) \right\} \cup \{ \tau \leq t_1 \} ~\right|~ \mathfrak{F}^{t_0}, \Upsilon^{t_0} \leq T } \geq 1 - n^{-\omega(1)}. \label{eq:first_union}
\end{align}
Now, \cref{lem:newcorrespondence} gives
\begin{equation}%
\Pro{ \left. G_{t_0}^{t_1} \geq 4\eps^2  \cdot \tilde{G}_{t_0}^{t_1 - \frac{n}{4\eps}} \;\right|\; \mathfrak{F}^{t_0}, \Upsilon^{t_0} \leq T} \geq 1- (t_1-t_0) \cdot e^{-\eps n}. \label{eq:second_union}
\end{equation}
Then, noting that
$\tilde{G}_{t_0}^{t_1 - \frac{n}{4\eps}} \geq \tilde{G}_{t_0}^{t_1} - \frac{n}{4\eps}$ and
by taking the union bound of \eqref{eq:first_union} and \eqref{eq:second_union} we have,
\begin{equation*}
\Pro{ \left. \left\{ G_{t_0}^{t_1} \geq 4\eps^2 \cdot \Big(\frac{1}{2} \cdot (t_1 - t_0) - \frac{n}{4\eps}\Big) \right\} \cup \{ \tau \leq t_1 \} \;\right|\; \mathfrak{F}^{t_0}, \Upsilon^{t_0} \leq T} \geq 1- (t_1-t_0) \cdot e^{-\eps n} - n^{-\omega(1)}.
\end{equation*}
Since, $\frac{1}{4} \cdot (t_1 - t_0) \geq \frac{1}{4} T \geq \frac{n}{4\eps}$, we can deduce that 
\begin{equation*}
\Pro{ \left. \left\{ G_{t_0}^{t_1} \geq \eps^2 \cdot (t_1 - t_0) \right\} \cup \{ \tau \leq t_1 \} \;\right|\; \mathfrak{F}^{t_0}, \Upsilon^{t_0} \leq T} \geq 1- (t_1-t_0) \cdot e^{-\eps n} - n^{-\omega(1)}.\qedhere
\end{equation*}
\end{proof}

\subsection{Recovery Phase of the Process} \label{sec:recovery}

In the following, we start the analysis at round $t_0 = m - n^3 \log^4 n$ and using \cref{lem:initial_gap_nlogn}, we obtain that \Whp~$\Gap(t_0) = \Oh(n \log n)$ and $\Lambda^{t_0} \leq n \cdot \exp(\Oh(n \log n))$. Using \cref{lem:many_good_quantiles_whp} we conclude that \Whp~for a constant fraction of the rounds $t \in [t_0,t_0 + n^3 \log^4 n]$, it holds that $\delta^t \in [\eps, 1 - \eps]$.
Next, using the adjusted exponential potential $\tilde{\Lambda}$ we exploit the drop of the exponential potential function in those rounds to infer that \Whp~there is a round $s \in [t_0,t_0+n^3 \log^4 n]$ with $\Lambda^{s} \leq cn$ (\cref{lem:recovery}). In \cref{sec:stabilization}, we will prove that \Whp~for the remaining rounds in $[s, m]$, which are at most $\Oh(n^3 \log^4 n)$, the potential becomes $\Oh(n)$ once every $\Oh(n \log n)$ rounds (\cref{lem:stabilization}).

\begin{lem}[Recovery] \label{lem:recovery}
Consider any $\PTwo \cap \WOne$-process or $\POne \cap \WTwo$-process and the potential $\Lambda := \Lambda(\alpha)$ for $\alpha > 0$ as defined in \eqref{eq:alpha_def}. Then, for $\C{lambda_bound}>1$ as defined  in \cref{cor:change_for_large_lambda}, for any round $m \geq n^3 \log^4 n$,
\[
\Pro{\bigcup_{s \in [m - n^3 \log^4 n, m]}\{ \Lambda^s \leq \C{lambda_bound} n\} } \geq 1 - n^{-5}.
\]
\end{lem}
\begin{proof}
Let $T := n^3 \log^4 n$, $t_0 := m - T$ and $t_1 := m$. By \cref{lem:initial_gap_nlogn}, there exists a constant $\C{poly_n_gap} > 0$, such that
\begin{align} \label{eq:starting_bound_rec}
\Pro{\max_{i \in [n]} \left| y_i^{t_0} \right| \leq \C{poly_n_gap} n \log n } \geq 1 - n^{-12}.
\end{align}
Note that when $\big\{ \max_{i \in [n]} \left| y_i^{t_0} \right| \leq \C{poly_n_gap} n \log n \big\}$ holds, then we also have that \[
\Upsilon^{t_0} \leq n \cdot (\C{poly_n_gap} n \log n)^2 
  \leq \C{poly_n_gap}^2 n^3 \log^2 n, \qquad \text{and} \qquad \Lambda^{t_0} \leq n \cdot e^{\alpha \C{poly_n_gap} n \log n} 
  \leq e^{2\C{poly_n_gap} n \log n}.
\]
By \cref{lem:many_good_quantiles_whp} we have for $\tau:=\inf\big\{ t \geq t_0 \colon \max_{i \in [n]} |y_i^{t}| > \sqrt{T} \cdot \log^{-2} n \big\}$, and using that the event $\big\{ \max_{i \in [n]} \left| y_i^{t_0} \right| \leq \C{poly_n_gap} n \log n \big\}$ implies that $\{ \Upsilon^{t_0} \leq T \}$,
\begin{align*}
& \Pro{ \left\{ G_{t_0}^{t_1} \geq \eps^2 \cdot (t_1 - t_0) \right\} \cup \{ \tau \leq t_1 \} \;\Big|\; \mathfrak{F}^{t_0}, \max_{i \in [n]} \left| y_i^{t_0} \right| \leq \C{poly_n_gap} n \log n} \\
 & \qquad \geq \Pro{ \left\{ G_{t_0}^{t_1} \geq \eps^2 \cdot (t_1 - t_0) \right\} \cup \{ \tau \leq t_1 \} \;\Big|\; \mathfrak{F}^{t_0}, \Upsilon^{t_0} \leq T} \\
 & \qquad \geq 1 - T  \cdot e^{-\eps n} - n^{-\omega(1)} \geq 1 - n^{-\omega(1)}.
\end{align*}
By \cref{lem:initial_gap_nlogn}~$(i)$ and the union bound over $t_1 - t_0 = n^3 \log^4 n \leq n^4$ rounds, we have that \begin{align*}
& \Pro{\tau > t_1 \,\left|\, \mathfrak{F}^{t_0}, \max_{i \in [n]} \left| y_i^{t_0} \right| \leq \C{poly_n_gap} n \log n \right. } \\
  & \qquad \geq \Pro{\left. \max_{t \in [t_0,t_1]}\max_{i \in [n]} \left| y_i^{t} \right| \leq 2 \C{poly_n_gap} n \log n \,\right|\, \mathfrak{F}^{t_0}, \max_{i \in [n]} \left| y_i^{t_0} \right| \leq \C{poly_n_gap} n \log n } 
  \geq 1 - T \cdot n^{-12} \geq 1 - n^{-8}.
\end{align*}
Hence, taking the union bound for the above two inequalities, we get that
\begin{equation}\label{eq:recovery_many_good_quantiles_whp}
\Pro{ G_{t_0}^{t_1} \geq \eps^2 \cdot (t_1 - t_0) \;\left|\; \mathfrak{F}^{t_0}, \max_{i \in [n]} \left| y_i^{t_0} \right| \leq \C{poly_n_gap} n \log n \right.} \geq 1 - n^{-\omega(1)} - 2n^{-8}.
\end{equation}
We will now use the adjusted exponential potential $(\tilde{\Lambda}_{t_0}^{t})_{t \geq t_0}$ (defined in \eqref{eq:tilde_lambda_def}) to show that $\Lambda$ becomes small in at least one round in $[t_0, t_1]$. By \cref{lem:lambda_tilde_is_supermartingale}, it forms a super-martingale, so $\ex{\tilde{\Lambda}_{t_0}^{t_1 + 1}\mid \mathfrak{F}^{t_0}} \leq \tilde{\Lambda}_{t_0}^{t_0} =  \Lambda^{t_0}$. Hence, using Markov's inequality we get \begin{align}
& \Pro{\left. \tilde{\Lambda}_{t_0}^{t_1+1} >\Lambda^{t_0} \cdot n^8 \,\right|\, \mathfrak{F}^{t_0}, \max_{i \in [n]} \left| y_i^{t_0} \right| \leq \C{poly_n_gap} n \log n } \notag \\
& \qquad \leq \Pro{\left. \tilde{\Lambda}_{t_0}^{t_1+1} >\Lambda^{t_0} \cdot n^8\,\right|\, \mathfrak{F}^{t_0}, \Lambda^{t_0} \leq e^{2 \C{poly_n_gap} n \log n} } 
\leq  n^{-8}. \label{eq:recovery_supermartingale_markov}
\end{align}
Thus, negating and by the definition of $\tilde{\Lambda}_{t_0}^t$, we have 
\[
\Pro{\Lambda^{t_1+1} \cdot \mathbf{1}_{\mathcal{E}_{t_0}^{t_1}} \leq \Lambda^{t_0} \cdot n^8 \cdot \exp\left(  \frac{\C{bad_quantile_mult} \alpha^2}{n} \cdot B_{t_0}^{t_1}  -\frac{\C{good_quantile_mult} \alpha}{n} \cdot G_{t_0}^{t_1} \right) \, \Bigg| \, \mathfrak{F}^{t_0}, \Lambda^{t_0} \leq e^{2 \C{poly_n_gap} n \log n} } \geq  1 - n^{-8}.    
\]
Further, if in addition to the two events $\{\tilde{\Lambda}_{t_0}^{t_1+1} \leq \Lambda^{t_0} \cdot n^8\}$ and $\{\Lambda^{t_0} \leq e^{2 \C{poly_n_gap} n \log n} \}$, also the event $\{G_{t_0}^{t_1} \geq \eps^2 \cdot (t_1 - t_0)\}$ holds, then
\begin{align*}
\Lambda^{t_1+1} \cdot \mathbf{1}_{\mathcal{E}_{t_0}^{t_1}} & \leq \Lambda^{t_0} \cdot n^8 \cdot \exp\left( \frac{\C{bad_quantile_mult} \alpha^2}{n} \cdot B_{t_0}^{t_1} - \frac{\C{good_quantile_mult} \alpha}{n} \cdot G_{t_0}^{t_1} \right) \\
 & \leq e^{2 \C{poly_n_gap} n \log n} \cdot n^8 \cdot \exp\left( \frac{\C{bad_quantile_mult} \alpha^2}{n} \cdot (1 - \eps^2) \cdot (t_1 - t_0) - \frac{\C{good_quantile_mult} \alpha}{n} \cdot \eps^2 \cdot (t_1 - t_0) \right) \\
 & \stackrel{(a)}{\leq} e^{2 \C{poly_n_gap} n \log n} \cdot n^8 \cdot \exp\left( \frac{\C{good_quantile_mult} \alpha}{n} \cdot \frac{\eps^2}{2} \cdot (t_1 - t_0) - \frac{\C{good_quantile_mult} \alpha}{n} \cdot \eps^2 \cdot (t_1 - t_0) \right) \\
 & = e^{2 \C{poly_n_gap} n \log n} \cdot n^8 \cdot \exp\left( - \frac{\C{good_quantile_mult} \alpha}{n} \cdot \frac{\eps^2}{2} \cdot (t_1 - t_0) \right) \\
 & \stackrel{(b)}{=} e^{2 \C{poly_n_gap} n \log n} \cdot n^8 \cdot \exp\left( - \frac{\C{good_quantile_mult} \alpha}{n} \cdot \frac{\eps^2}{2} \cdot n^3 \log^4 n \right) \stackrel{(c)}{\leq} 1,
\end{align*}
using in $(a)$ that $\alpha \leq \frac{\C{good_quantile_mult} \eps^2}{2 \cdot \C{bad_quantile_mult} \cdot (1 - \eps^2)}$ and $(b)$ that $T := n^3 \log^4 n$ and in $(c)$ that $\alpha, \eps , \C{poly_n_gap} > 0$ are constants.
Observe that $\{ \Lambda^{t_1+1} \geq n \}$ holds deterministically, so we deduce from the above inequality that $\mathbf{1}_{\mathcal{E}_{t_0}^{t_1}}=0$, that is,
\[\Pro{ \neg \mathcal{E}_{t_0}^{t_1} \; \Bigg|\; \mathfrak{F}^{t_0},  \;\; \; \tilde{\Lambda}_{t_0}^{t_1+1} \leq \Lambda^{t_0} \cdot n^8, \;\;\; \Lambda^{t_0} \leq e^{2 \C{poly_n_gap} n \log n},  \;\;\; G_{t_0}^{t_1} \geq \eps^2 \cdot (t_1 - t_0)} \leq 1.\] 
Recalling the definition of $\mathcal{E}_{t_0}^{t_1} = \bigcap_{s \in [t_0, t_1]} \{ \Lambda^s > cn \}$, and taking the union bound over \eqref{eq:recovery_many_good_quantiles_whp} and \eqref{eq:recovery_supermartingale_markov} yields 
\[
\Pro{ \left. \bigcup_{s \in [t_0, t_1]} \{ \Lambda^s \leq \C{lambda_bound}n \} \; \right| \; \mathfrak{F}^{t_0}, \max_{i \in [n]} \left| y_i^{t_0} \right| \leq \C{poly_n_gap} n \log n }\geq 1 - 2n^{-8} - n^{-8} = 1 - 3n^{-8}.
\]
We conclude by combining with \eqref{eq:starting_bound_rec},
\[
\Pro{ \bigcup_{s \in [t_0, t_1]} \{ \Lambda^s \leq \C{lambda_bound} n \} } \geq \left(1 - 3n^{-8}\right) \cdot \left(1 - n^{-12}\right) \geq 1 - n^{-5}. \qedhere
\]
\end{proof}

\subsection{Stabilization Phase of the Process} \label{sec:stabilization}
 
 The next lemma establishes that a small value of the exponential potential function is preserved for some longer time period. We will refer to this property of the exponential potential function of being ``trapped'' in some region as \emph{stabilization}. More precisely, we prove in the lemma below that if for some round $t_0$ the exponential potential is not too small, i.e., at most $2cn$, then within the next $\Oh(n \log n)$ rounds, the exponential potential will \Whp~be smaller than $cn$ at least once.

\begin{lem}[Stabilization]\label{lem:stabilization}
Consider any $\PTwo \cap \WOne$-process or $\POne \cap \WTwo$-process and the potential $\Lambda := \Lambda(\alpha)$ for $\alpha > 0$ as defined in \eqref{eq:alpha_def}. Further, let $\eps>0$ be given by \eqref{eq:c_def}, $\C{lambda_bound}:=\C{lambda_bound}(\eps)>1$ be as defined in \cref{cor:change_for_large_lambda}, and $\C{stab_time} := \frac{22}{\C{good_quantile_mult} \eps^2 \alpha}$, for $\C{good_quantile_mult}:=\C{good_quantile_mult}(\eps) > 0$ as defined in \cref{lem:good_quantile_good_decrease}. Then, for any round $t_0 \geq 0$,
\[
 \Pro{\left.  \bigcup_{s \in [t_0,t_0 + \C{stab_time} n \log n]} \left\{\Lambda^{s} \leq \C{lambda_bound} n \right\} ~\right|~ \mathfrak{F}^{t_0}, \Lambda^{t_0} \in (\C{lambda_bound}n, 2\C{lambda_bound}n] } \geq 1 - n^{-9}.
\] 
\end{lem}

\begin{proof} 
Let $T := \C{stab_time} n \log n$ and  $t_1 :=t_0 + T$. By \cref{lem:many_good_quantiles_whp}, we have for  $\tau:=\inf\{ t \geq t_0 \colon \max_{i \in [n]} |y_i^{t}| > \sqrt{T} \cdot \log^{-2} n \}$ that
\[
\Pro{ \left\{ G_{t_0}^{t_1} \geq \eps^2 \cdot (t_1 - t_0) \right\} \cup \{ \tau \leq t_1 \} \;\Big|\; \mathfrak{F}^{t_0}, \Upsilon^{t_0} \leq T} \geq 1 - T  \cdot e^{-\eps n} - n^{-\omega(1)} \geq 1 - n^{-12}.
\] 
By \cref{lem:lambda_bound_implies_upsilon_bound}, the event $\big\{ \Lambda^{t_0} \in (\C{lambda_bound}n, 2\C{lambda_bound}n] \big\}$ implies that $\big\{ \Upsilon^{t_0} \leq T \big\}$ (since $T = \omega(n)$), and hence
\begin{align} \label{eq:stab_union_bound_1}
\Pro{ \left\{ G_{t_0}^{t_1} \geq \eps^2 \cdot (t_1 - t_0) \right\} \cup \{ \tau \leq t_1 \} \;\Big|\; \mathfrak{F}^{t_0}, \Lambda^{t_0} \in (\C{lambda_bound}n, 2\C{lambda_bound}n]} \geq 1 - n^{-12}.
\end{align}
By \cref{lem:small_change_for_linear_lambda} with $\kappa := \C{stab_time}$, as $\sqrt{T} \cdot \log^{-2} n \geq \log^2 n$ we have that \begin{align}
\Pro{\left. \tau > t_1 \,\right| \, \mathfrak{F}^{t_0},\; \Lambda^{t_0} \in (\C{lambda_bound}n, 2\C{lambda_bound}n] } 
  & \geq \Pro{\left. \max_{t \in [t_0,t_1]}\max_{i \in [n]} \left| y_i^{t} \right| \leq \log^2 n \,\right| \, \mathfrak{F}^{t_0},\; \Lambda^{t_0} \in (\C{lambda_bound}n, 2\C{lambda_bound}n] } \notag \\
  & \geq \Pro{\left. \max_{t \in [t_0,t_1]}\max_{i \in [n]} \left| y_i^{t} \right| \leq \log^2 n \,\right| \, \mathfrak{F}^{t_0},\; \Lambda^{t_0} \leq n^2 } \notag \\
  & \geq 1 - n^{-12}. \label{eq:stab_union_bound_2}
\end{align}
Hence, taking the union bound over \eqref{eq:stab_union_bound_1} and \eqref{eq:stab_union_bound_2}, we get that
\begin{equation}\label{eq:stabilization_many_good_quantiles_whp}
\Pro{ G_{t_0}^{t_1} \geq \eps^2 \cdot (t_1 - t_0) \;\Big|\; \mathfrak{F}^{t_0}, \Lambda^{t_0} \in (\C{lambda_bound}n, 2\C{lambda_bound}n]} \geq 1 - 2n^{-12}.
\end{equation}
We will now use the adjusted exponential potential $(\tilde{\Lambda}_{t_0}^{t})_{t \geq t_0}$ (defined in \eqref{eq:tilde_lambda_def}) to show that $\Lambda$ becomes small in at least one round in $[t_0, t_1]$. By \cref{lem:lambda_tilde_is_supermartingale}, it forms a super-martingale, so $\ex{\tilde{\Lambda}_{t_0}^{t_1+1}\mid \mathfrak{F}^{t_0}} \leq \tilde{\Lambda}_{t_0}^{t_0} =  \Lambda^{t_0}$. Hence, using Markov's inequality we get \begin{align}
\Pro{\left. \tilde{\Lambda}_{t_0}^{t_1+1} >\Lambda^{t_0} \cdot n^{10}\,\right|\, \mathfrak{F}^{t_0}, \Lambda^{t_0} \in (\C{lambda_bound}n, 2\C{lambda_bound}n] } 
\leq  n^{-10}. \label{eq:stabilization_supermartingale_markov}
\end{align}
Thus, by the definition of $\tilde{\Lambda}_{t_0}^t$, we have 
\begin{equation*}
\Pro{\Lambda^{t_1+1} \cdot \mathbf{1}_{\mathcal{E}_{t_0}^{t_1}} \leq \Lambda^{t_0} \cdot n^{10} \cdot \exp\left(  \frac{\C{bad_quantile_mult} \alpha^2 }{n} \cdot B_{t_0}^{t_1}  -\frac{\C{good_quantile_mult} \alpha}{n} \cdot G_{t_0}^{t_1} \right) \, \Bigg| \, \mathfrak{F}^{t_0}, \Lambda^{t_0} \in (\C{lambda_bound}n, 2\C{lambda_bound}n] } \geq  1 - n^{-10}.
\end{equation*}
Further, if in addition to the two events $\big\{\tilde{\Lambda}_{t_0}^{t_1+1} \leq \Lambda^{t_0} \cdot n^{10} \big\}$ and $\big\{\Lambda^{t_0} \in (\C{lambda_bound}n, 2\C{lambda_bound}n]\big\}$, also the event $\{G_{t_0}^{t_1} \geq \eps^2 \cdot (t_1 - t_0)\}$ holds, then
\begin{align*}
\Lambda^{t_1+1} \cdot \mathbf{1}_{\mathcal{E}_{t_0}^{t_1}} & \leq \Lambda^{t_0} \cdot n^{10} \cdot \exp\left( \frac{\C{bad_quantile_mult} \alpha^2}{n} \cdot B_{t_0}^{t_1} - \frac{\C{good_quantile_mult} \alpha}{n} \cdot G_{t_0}^{t_1} \right) \\
 & \leq 2cn^{11} \cdot \exp\left( \frac{\C{bad_quantile_mult} \alpha^2}{n} \cdot (1 - \eps^2) \cdot (t_1 - t_0) - \frac{\C{good_quantile_mult} \alpha}{n} \cdot \eps^2 \cdot (t_1 - t_0) \right) \\
 & \stackrel{(a)}{\leq} 2cn^{11} \cdot \exp\left( \frac{\C{good_quantile_mult} \alpha}{n} \cdot \frac{\eps^2}{2} \cdot (t_1 - t_0) - \frac{\C{good_quantile_mult} \alpha}{n} \cdot \eps^2 \cdot (t_1 - t_0) \right) \\
 & = 2cn^{11} \cdot \exp\left( - \frac{\C{good_quantile_mult} \alpha}{n} \cdot \frac{\eps^2}{2} \cdot (t_1 - t_0) \right) \\
 & \stackrel{(b)}{=} 2cn^{11} \cdot \exp\left( - \frac{\C{good_quantile_mult} \alpha}{n} \cdot \frac{\eps^2}{2} \cdot  \frac{22}{\C{good_quantile_mult} \eps^2 \alpha} \cdot n \log n \right)  = 2\C{lambda_bound},
\end{align*}
using in $(a)$ that $\alpha \leq \frac{\C{good_quantile_mult} \eps^2}{2 \cdot (1 - \eps^2) \cdot \C{bad_quantile_mult}}$ and $(b)$ that $T := \frac{22}{\C{good_quantile_mult} \eps^2 \alpha} \cdot n \log n$.
Observe that $\{ \Lambda^{t_1} \geq n \}$ holds deterministically, so we deduce from the above inequality that $\mathbf{1}_{\mathcal{E}_{t_0}^{t_1}}=0$, that is,
\[\Pro{ \neg \mathcal{E}_{t_0}^{t_1} \; \Bigg|\; \mathfrak{F}^{t_0},  \;\; \; \tilde{\Lambda}_{t_0}^{t_1+1} \leq \Lambda^{t_0} \cdot n^{10}, \;\;\; \Lambda^{t_0} \in (\C{lambda_bound}n, 2\C{lambda_bound}n],  \;\;\; G_{t_0}^{t_1} \geq \eps^2 \cdot (t_1 - t_0)} =1.\] 
Recalling the definition of $\mathcal{E}_{t_0}^{t_1} = \bigcap_{s \in [t_0, t_1]} \{ \Lambda^s \geq cn \}$, and taking the union bound over \eqref{eq:stabilization_many_good_quantiles_whp} and \eqref{eq:stabilization_supermartingale_markov} yields 
\[
\Pro{ \left. \bigcup_{s \in [t_0, t_0 + \C{stab_time} n \log n]} \{ \Lambda^s \leq cn \} \; \right| \; \mathfrak{F}^{t_0}, \Lambda^{t_0} \in (\C{lambda_bound}n, 2\C{lambda_bound}n]  }\geq 1 - 2n^{-12} - n^{-10} \geq 1 - n^{-9},
\]
as claimed. 
\end{proof} 

The next lemma shows that if there is a round with linear (exponential) potential then the gap is at most logarithmic for the next $n^4$ rounds. This follows by repeatedly applying \cref{lem:stabilization} to any contiguous sub-interval of length $\Theta(n\log n)$ and showing that it contains a round where $\Lambda$ is linear. The result then follows since a linear potential implies a logarithmic gap, and the load of a bin can change by at most $\Theta(\log n)$ in $\Theta(n\log n)$ rounds.

\begin{lem} \label{lem:good_gap_after_good_lambda}
Consider any $\PTwo \cap \WOne$-process or $\POne \cap \WTwo$-process, the potential $\Lambda := \Lambda(\alpha)$ for $\alpha > 0$ as defined in \eqref{eq:alpha_def}, $\C{lambda_bound} > 0$ as defined in \cref{cor:change_for_large_lambda} and $\C{stab_time} > 0$ as defined in \cref{lem:stabilization}. Then, for $\kappa := \frac{2}{\alpha} + w_- \cdot \C{stab_time}$ and for any rounds $t_0 \geq 0$ and $t_0 < t_1 \leq t_0 + n^4$,
\[
\Pro{\left. \max_{i \in [n] } \lvert y_i^{t_1} \rvert \leq \kappa \cdot \log n \, \right| \, \mathfrak{F}^{t_0}, \Lambda^{t_0} \leq \C{lambda_bound} n} \geq 1 - n^{-5}.
\]
\end{lem}

\begin{proof} 
We start by defining the event
\[ \mathcal{M}_{t_0}^{t_1} = \left\{\text{for all }t\in [t_0, t_1 ]\text{ there exists } s\in [t, t+\C{stab_time} n\log n  ] \text{ such that }\Lambda^s \leq \C{lambda_bound}n\right\},\]	
that is, if $\mathcal{M}_{t_0}^{t_1}$ holds then we have $\Lambda^s < \C{lambda_bound} n $ at least once every $\C{stab_time} n\log n$ rounds. 

Assume now that $\mathcal{M}_{t_0}^{t_1}$ holds. Choosing $t=t_1$, there exists $s \in [t_1,t_1+\C{stab_time} n \log n]$ such that $\Lambda^s \leq \C{lambda_bound}n$, which in turn implies by definition of $\Lambda$ that $\max_{i \in [n]} |y_i^s| \leq \frac{1}{\alpha} \cdot \log (\C{lambda_bound}n) < \frac{2}{\alpha} \cdot \log n$. Clearly, any $y_i^t$ can decrease by at most $w_-/n$ in each round, and from this it follows that if $\mathcal{M}_{t_0}^{t_1}$ holds, then 
$
\max_{i \in [n]} y_i^{t_1} \leq \max_{i \in [n]} |y_i^s| + w_-\cdot \C{stab_time} \log n \leq \kappa \cdot \log n, $
for $\kappa := \frac{2}{\alpha} +  w_- \cdot \C{stab_time}$.

We now turn to lower bounding the minimum load. If $t_1 \geq t_0 + \C{stab_time} n \log n$ and $\mathcal{M}_{t_0}^{t_1}$ holds, then choosing $t = t_1 - \C{stab_time} n \log n$, there exists $s \in [t_1 - \C{stab_time} n \log n, t_1]$ such that $\Lambda^s \leq \C{lambda_bound}n$. (In case $t_1 -  \C{stab_time} n \log n < t_0 $, then we arrive at the same conclusion by taking $s = t_0$ and using the condition $\Lambda^{t_0} \leq \C{lambda_bound} n$,). %
This in turn implies
$
 \max_{i \in [n]} | y_i^s  | < \frac{2}{\alpha} \cdot \log n.
$
Hence 
\[
\min_{i \in [n]} y_i^{t_1} \geq \min_{i \in [n]} y_i^{s} - w_-\cdot \C{stab_time} \log n 
\geq -\max_{i \in [n]} |y_i^{s}| -  w_-\cdot \C{stab_time} \log n 
\geq -\kappa \cdot \log n.
\]
 Hence, $\mathcal{M}_{t_0}^{t_1}$ (conditioned on $\Lambda^{t_0} \leq cn$) implies that $\max_{i \in [n]} \lvert y_i^{t_1} \rvert \leq \kappa \cdot \log n$. It remains to bound $\Pro{\neg \mathcal{M}_{t_0}^{t_1} \mid \mathfrak{F}^{t_0}, \Lambda^{t_0} \leq cn}$.

Note that since we start with $\Lambda^{t_0} \leq \C{lambda_bound} n$, if for some round $j \geq t_0$, $\Lambda^j> 2 \C{lambda_bound} n$,  there must exist some $s \in [t_0,j)$, such that $\Lambda^s \in (\C{lambda_bound} n, 2\C{lambda_bound}n]$, 
since for every $t \geq 0$ it holds $\Lambda^{t+1} \leq \Lambda^t \cdot e^{\alpha w_-} \leq 2 \Lambda^t$ and $\alpha \leq 1/(2w_-)$. Let $t_0 < \tau_1 < \tau_2 < \cdots$ and $t_0 =: s_0 < s_1 < \cdots$ be two interlaced sequences defined recursively for $i \geq 1$ by \[
  \tau_i := \inf\{ \tau >  s_{i-1}: \Lambda^\tau \in (\C{lambda_bound}n, 2\C{lambda_bound}n] \}
  \qquad\text{and}\qquad 
  s_i := \inf\{s > \tau_i : \Lambda^s \leq \C{lambda_bound} n\}. 
\] 
Thus we have
  \[
  t_0 = s_0 < \tau_1 < s_1 < \tau_2 <s_2 < \cdots, 
  \]
  and since $\tau_i > \tau_{i-1}$ we have $\tau_{t_1 - t_0}\geq t_1 - t_0$. Observe that if the event $\cap_{i=1}^{t_1 - t_0}\{s_i - \tau_i \leq \C{stab_time}n \log n \} $ holds, then also $ \mathcal{M}_{t_0}^{t_1}$ holds. 
  
  Recall that by \cref{lem:stabilization} we have for any $i=1,2,\ldots, t_1 - t_0$ and any $\tau = t_0 + 1, \ldots , t_1$ \[
 \Pro{ \left. \bigcup_{t \in [\tau_i,\tau_i + \C{stab_time} n \log n]} \left\{\Lambda^{t} \leq \C{lambda_bound} n \right\} ~\right|~ \mathfrak{F}^{\tau} , \; \Lambda^{\tau} \in (\C{lambda_bound}n, 2\C{lambda_bound}n], \tau_i = \tau } \geq  1 - n^{-9},
  \] and by negating and the definition of $s_i$,
  \[
  \Pro{s_i - \tau_i> \C{stab_time}n \log n \, \Big| \, \mathfrak{F}^{\tau}, \Lambda^{\tau} \in (\C{lambda_bound}n, 2\C{lambda_bound}n], \tau_i = \tau } \leq n^{-9}.
  \]
  Since the above bound holds for any $i$ and $\mathfrak{F}^{\tau}$, with $\tau_i = \tau$, it follows by the union bound over all $i=1,2,\ldots,t_1 - t_0$, as $t_1 - t_0 \leq n^4$,
  \[ %
  \Pro{\left. \neg \mathcal{M}_{t_0}^{t_1} \,\right|\, \mathfrak{F}^{t_0}, \Lambda^{t_0} \leq cn }\leq (t_1 - t_0)\cdot n^{-9} \leq n^{-5}. \qedhere
  \]
  \end{proof}

\subsection{Completing the Proof of Theorem~\ref{thm:main_technical}}

We are now ready to combine the recovery and stabilization phases to complete the proof the main theorem.%

{\renewcommand{\thethm}{\ref{thm:main_technical} }
\begin{thm}[Restated, page~\pageref{thm:main_technical}]
\maintechnical
	\end{thm}}
	\addtocounter{thm}{-1}

\begin{proof}
Consider an arbitrary round $m \geq 0$ and the potential $\Lambda := \Lambda(\alpha)$ with smoothing parameter $\alpha > 0$ as defined in \eqref{eq:alpha_def} and let $\C{lambda_bound} \geq 1$ be as defined in \cref{cor:change_for_large_lambda}. If $m < n^3 \log^4 n$, then the claim follows by \cref{lem:good_gap_after_good_lambda} with $t_0 := 0$ and $t_1 := m$, since $\Lambda^{t_0} = n \leq cn$.

Otherwise, let $t_0 := m - n^3 \log^4 n$. %
Firstly, by \cref{lem:recovery}, we get
\[
\Pro{\bigcup_{s \in [t_0, m] } \left\{ \Lambda^{s} \leq cn \right\} } \geq 1 - n^{-5}.
\]
Hence for the stopping time $\sigma:=\inf\{ s \geq t_0 \colon \Lambda^s \leq cn \}$ we have $\Pro{ \sigma \leq m } \geq 1-n^{-5}$.

Secondly, by \cref{lem:good_gap_after_good_lambda}, there exists a constant $\kappa := \kappa(\alpha)$ such that for any round $s \in [t_0, m]$
\[
\Pro{\left. \max_{i \in [n]} \left| y_i^m \right| \leq \kappa \cdot \log n ~\right|~ \mathfrak{F}^s , \Lambda^s \leq \C{lambda_bound} n} \geq 1 - n^{-5}.
\]
Combining the two inequalities from above, 
\begin{align*}
    \Pro { \max_{i \in [n]} \left| y_i^m \right| \leq \kappa \log n} &\geq \sum_{s=t_0}^{m} \Pro{ \left. \max_{i \in [n]} \left| y_i^m \right| \leq \kappa \log n  ~\right|~ \sigma = s } \cdot \Pro{ \sigma = s} \\
    &\geq \sum_{s=t_0}^{m} \Pro{ \left. \max_{i \in [n]} \left| y_i^m \right| \leq \kappa \log n  ~\right|~ \mathfrak{F}^{s}, \Lambda^{s} \leq \C{lambda_bound}n } \cdot \Pro{ \sigma = s} \\
    &\geq \left(1 - n^{-5} \right) \cdot \Pro{ \sigma \leq m} \\
    &\geq \left(1 - n^{-5} \right) \cdot \left(1 - n^{-5} \right) \geq 1 - n^{-4}.\qedhere
\end{align*}

\end{proof}

\section{Lower Bounds}\label{sec:lower_bounds}

In this section, we shall prove lower bounds for \MeanBiased processes satisfying condition \PThree. We now recall the definition of condition \PThree:
 
\begin{itemize}
  \item \textbf{Condition \PThree}: for any $\eps>0$ there exists a $\C{p4k} := \C{p4k}(\eps) \in (0, 1]$ such that for each round $t \geq 0$ with $\delta^t \in [\eps, 1 - \eps]$, the probability allocation vector $p^t$ satisfies, 
    \[ \min_{i \in [n]} p_i^t \geq \frac{\C{p4k}}{n}.\] 
\end{itemize}
This condition is satisfied for many natural processes, including \Twinning, \MeanThinning, and the $(1+\beta)$-process with constant $\beta \in (0,1)$. See Table \ref{tab:overview} for a concise overview of old and new lower bounds.
One interesting aspect is that our lower bound for \MeanBiased processes makes use of the quantile stabilization result from the previous section (\cref{lem:many_good_quantiles_whp}).

Essentially this condition \PThree implies that in any round $t$ where the mean quantile $\delta^t$ is in $[\eps, 1 - \eps]$, there is at least an $\Omega(1/n)$-probability of allocating to each bin.

We shall now observe that 
\begin{itemize}
	\item The \MeanThinning process has $p_i^t \geq \frac{\delta^t}{n}$, satisfying \PThree with $\C{p4k} := \eps$. 
	\item The \Twinning process has $p_i^t = \frac{1}{n}$, satisfying \PThree with $\C{p4k} := 1$. 
	\item The $(1+\beta)$-process has $ p_i^t = (1-\beta) \cdot \frac{1}{n}+ \beta \cdot \frac{2i-1}{n^2}> \frac{1-\beta}{n}$, satisfying \PThree with $\C{p4k} := 1-\beta$. 
\end{itemize}

Recall that $G_0^m:=G_0^m(\eps)$ is the number of rounds $t \in [0, m]$ with $\delta^t \in [\eps, 1 - \eps]$. The next lemma shows that \Whp~a constant fraction of the first $m = \poly(n)$ rounds are good. 
	\begin{lem}
		\label{lem:many_good_quantiles_in_lightly_loaded}
		Consider any $\PTwo \cap \WOne$-process or $\POne \cap \WTwo$-process. Let $\eps \in \big( 0, \frac{1}{2} \big)$ be given by \eqref{eq:c_def},   and $m \in [\C{stab_time} n \log n, n^3 \log^4 n]$ for $\C{stab_time}:= \C{stab_time}(\eps)> 0$ as defined in \cref{lem:stabilization}. Then,  
		\[
		\Pro{G_0^m(\eps) > \eps^2 \cdot m} \geq 1 - n^{-1/2}.
		\]
	\end{lem}
	\begin{proof}
		The proof is analogous to the first steps in the proof of the stabilization lemma (\cref{lem:stabilization}). 
		At round $0$, we have that\[
		\Lambda^{0} = n, \qquad \text{and} \qquad \Upsilon^{0} = 0.
		\] Let $t_0 := 0$, $t_1 := m$, $T := m$, and $\tau:=\inf\{ t \geq 0 \colon \max_{i \in [n]} |y_i^{t}| > \sqrt{T} \cdot \log^{-2} n \}$. Then,  
		\begin{align} \label{eq:lb_union_bound_1}
			\Pro{ \left\{ G_{0}^{m}> \eps^2 \cdot m \right\} \cup \{ \tau \leq m \} } \geq 1 - T  \cdot e^{-\eps n} - n^{-\omega(1)} \geq 1 - n^{-12},
		\end{align}by \cref{lem:many_good_quantiles_whp}. By \cref{thm:main_technical} and the union bound we have 
	\begin{equation}\label{eq:lb_union_bound_2}	\Pro{\tau > m} 
		\geq \Pro{ \max_{t \in [0,m]}\max_{i \in [n]} \left| y_i^{t} \right| \leq \log^2 n  } \geq 1 - n^{-4} \cdot m \geq  1 -   n^{-4/5}.\end{equation}
		Hence, taking the union bound over \eqref{eq:lb_union_bound_1} and \eqref{eq:lb_union_bound_2}, we get that\[
		\Pro{ G_{0}^{m}> \eps^2 \cdot m } \geq 1 - n^{-12} -   n^{-4/5} \geq 1 - n^{-1/2}. \qedhere
		\] 
	\end{proof}
\noindent We will use the following lower bound on the \OneChoice process (see also~\cite[Section 4]{PTW15}). 

\begin{lem}[{\cite[Lemma A.2]{LS23RBB}}]
\label{lem:one_choice_cnlogn}
Consider the \OneChoice process for $m = \lambda n \log n$ balls where $\lambda \geq 1/\log n$.
Then,
\[
\Pro{\max_{i \in [n]} x_i^m \geq \left(\lambda + \frac{\sqrt{\lambda}}{10} \right)  \cdot \log n} \geq 1 - n^{-2}.
\]
\end{lem}

We are now ready to restate and prove the main result of this section.
{\renewcommand{\thethm}{\ref{thm:main_lower} }
	\begin{thm}[Restated, page~\pageref{thm:main_lower}]
		\mainlower
\end{thm}}
\addtocounter{thm}{-1}

	\begin{proof} We start by defining the constant $\C{p4k} := \C{p4k}(\eps)  \in (0,1]$  for $\eps \in \big( 0, \frac{1}{2} \big)$ given by \eqref{eq:c_def}. Further, set $\kappa := \frac{\C{p4k}\eps^2}{800w_-^2} >0 $, and let $m := \kappa n\log n$. Define the event  $\mathcal{E}_1 := \{ |\{t\in [0,m] :  \delta^t\in[\eps, 1 - \eps] \}|\geq \eps^2 m\}$ and note that   $\Pro{\neg\mathcal{E}_1} <  n^{-2/3}  $ by \cref{lem:many_good_quantiles_in_lightly_loaded}. 
 
		We will now define a coupling between the allocations of the process and a \OneChoice process for each round $t\in [0,m]$:
	 \begin{quote} If $\delta^t\in[\eps, 1 - \eps]$:  Sample an independent Bernoulli random variable $X^t \sim \mathsf{Ber}(\C{p4k})$. %
		\begin{itemize}
			\item If $X^t = 1$, then we allocate the ball to a bin sampled uniformly at random.
			\item Otherwise, allocate to the $i$-th heaviest bin with probability $\displaystyle{\frac{p_i^t - \frac{\C{p4k}}{n}}{1 - \C{p4k}}}$.
		\end{itemize} 
	 \noindent If $\delta^t\notin[\eps, 1 - \eps]$: allocate to bin $i\in [n]$ with probability $p_i^t$.
		\end{quote} 
    Observe that this has the correct marginal distribution, since for any step $t$ with $\delta^t \in [\eps,1 - \eps]$,  \[
    k_4 \cdot \frac{1}{n} + 
    (1-k_4) \cdot \frac{p_i^t - \frac{k_4}{n}}{1-k_4} = p_i^t.
  \]
  
	Let  $(Y^i)_{i\geq 0}$, where $Y^i\sim \mathsf{Ber}(\C{p4k})$, be a sequence of i.i.d.\  Bernoulli random variables used in the coupling. That is, we pre-sample $(Y^i)_{i\geq 0}$ and then at round $t$ of the coupling if we have $\delta^t \in [\eps, 1 - \eps]$ for the $i$-th time then we take $X^t=Y^i$.  Define the event $\mathcal{E}_2:=\{\sum_{i=1}^{\eps^2 m} Y^i > \C{p4k}\eps^2m/2\}$. Then conditional on $\mathcal{E}_1\cap \mathcal{E}_2$, at least $(\C{p4k}\eps^2\kappa/2) \cdot n \log n$ balls are allocated according to the \OneChoice process during the $m$ steps of the coupling. By a Chernoff bound, 
	\begin{equation*} \Pro{\neg\mathcal{E}_2} = \Pro{\mathsf{Bin}(\eps^2m,\C{p4k})  \leq \frac{1}{2}\cdot \C{p4k}\eps^2m }\leq  \exp\left( -\frac{\C{p4k}\eps^2m }{8}\right) < n^{-2}.  \end{equation*}

We then sample a \OneChoice process, with load vector denoted by $\tilde{x}^t$, on $h:=\lambda n\log n $ rounds where $\lambda:=\frac{\C{p4k}\eps^2\kappa}{2}$, independently from all other random bits used in the coupling. Define the event $\mathcal{E}_3 = \{\max_{i \in [n]} \tilde{x}_i^{h} \geq \left(c + \frac{\sqrt{c}}{10} \right)  \cdot \log n   \} $ and observe that $\Pro{\neg\mathcal{E}_3} < n^{-2}$ by \cref{lem:one_choice_cnlogn}.  

Now, as $\kappa := \C{p4k}\eps^2/(800w_-^2)$, and each ball has weight in $[1,w_-]$, conditional on $\mathcal{E}_1\cap \mathcal{E}_2\cap \mathcal{E}_3$,   
	 
		\[ \frac{\Gap(m)}{\log n}  \geq \frac{\C{p4k}\eps^2\kappa}{2} +  \frac{1}{10}\sqrt{\frac{\C{p4k}\eps^2\kappa}{2}} - w_- \kappa \geq   \frac{1}{10}\sqrt{\frac{\C{p4k}^2\eps^4}{1600 w_-^2}} - w_- \kappa = w_- \kappa \geq \kappa.\]
	The result then follows since $\Pro{\neg\mathcal{E}_1\cup \neg\mathcal{E}_2\cup \neg\mathcal{E}_3} \leq n^{-2/3}  + n^{-2}+ n^{-2} <n^{-1/2} $.\end{proof}

 Hence, we can deduce from the lemma above by recalling that $\MeanThinning$, $\Twinning$ and the $(1+\beta)$-process all satisfy \PThree and the conditions \POne and \WTwo, or, \WOne and \PTwo:
\begin{cor}
Consider any of the \MeanThinning, \Twinning or the $(1+\beta)$-process with any constant $\beta \in (0, 1)$. Then, there exists a  constant $\kappa>0$ (different for each process) such that 
\[
\Pro{\Gap(\kappa n\log n ) \geq \kappa \log n} \geq 1 - n^{-1}.
\]
\end{cor}

Finally, we prove a tight lower bound for any $\RelativeThreshold(f(n))$ process with $f(n) \geq \log n$.  
 
{\renewcommand{\thethm}{\ref{lem:rellower} }
	\begin{lem}[Restated, page~\pageref{lem:rellower}]
		\rellower
\end{lem}}
\addtocounter{thm}{-1}
\begin{proof} Note that in any round $t \geq 0$ where $\max_{i \in [n]} y_i^t < f(n)$, the $\RelativeThreshold(f(n))$ process will always accept the first sample, and hence the allocation of this process can be coupled with the allocation of \OneChoice. 

In the following, fix an arbitrary bin $i \in [n]$. Let $(Y^t)_{1 \leq t \leq m}$ be independent $\mathsf{Ber}(1/n)$ random variables, and for any $1 \leq t \leq m$, define $Z^t := \sum_{s=1}^t \left( Y^s - \frac{1}{n} \right)$,  Then $\Ex{Z^t}=0$, and
\[
 \Var  {Z^{t+1} \, \mid \, Z^t } = \frac{1}{n} \cdot \left(1 - \frac{1}{n} \right) \leq \frac{1}{n} =: \sigma^2.
\]
Also $|Z^{t+1}-Z^{t}| \leq 1 =: M$ holds deterministically. Further, define $\tau:=\min \{ t \in \N \colon y_i^t \geq f(n) \}$. Further, since $Z^{t \wedge \tau}$ is a martingale, it follows by \cref{lem:cl06_thm_6_1} %
with $\lambda := f(n)$, 
\begin{align*}
 \Pro{ |Z^{m \wedge \tau} | \geq f(n) } &\leq 2\exp\left(- \frac{ f(n)^2}{2 (m \cdot \frac{1}{n} + f(n)/3) } \right) \\
 &= 2\exp\left(- \frac{f(n)}{2 \cdot  ( f(n) / (24 \log(n)) +1/3) }\right) \\
 &\leq 2n^{-4/3},
\end{align*}
using that $m:=\frac{1}{24} \cdot \frac{n \cdot (f(n))^2}{\log n}$ and $f(n) \geq \log n$. If the event $|Z^{m \wedge \tau}| < f(n)$ holds, then this implies that $\max_{1 \leq t \leq m} Z^t < f(n)$, and thanks to the coupling, $\max_{1 \leq t \leq m} y_i^t < f(n)$. Taking the union bound over all $n$ bins $i \in [n]$, \[
\Pro{ \max_{i \in [n]} \max_{1 \leq t \leq m}  y_i^t < f(n)} \geq 1-2n^{-1/3}.
\]
Conditional on this, the allocations of $\RelativeThreshold(f(n))$ and \OneChoice can be coupled until round $m$. Now, for \OneChoice, \cref{lem:one_choice_cnlogn}  with $\lambda :=( 1/24 )\cdot f(n)^2 / (\log n)^2$ implies,
\[
 \Pro{ \max_{i \in [n]} y_i^m \geq \frac{\sqrt{\lambda}}{10} \cdot \log n } \geq \Pro{ \max_{i \in [n]} y_i^m \geq \frac{f(n)}{50}} \geq 1-n^{-2}.
\]
Hence by the union bound, with probability at least $1-2n^{-1/3}-n^{-2} \geq 1-n^{-1/4}$ the gap of $\RelativeThreshold(f(n))$ at round $m$ is at least $f(n)/50$.
\end{proof}

\section{Sample Efficiency} \label{sec:sample_efficiency}

Recall that the \textit{sample efficiency} $\eta^m = W^m/S^m$ of a process, as the ratio of the weight $W^m$ of the balls allocated to the number of bin samples $S^m$ taken in the first $m$ rounds. The \OneChoice process has $\eta^m = 1$ and the \TwoChoice process has $\eta^m = \frac{1}{2}$.

We begin by showing that \MeanThinning is more sample-efficient than \TwoChoice.

{\renewcommand{\thelem}{\ref{lem:mean_thinning_efficiency}}
	\begin{lem}[Restated, page~\pageref{lem:mean_thinning_efficiency}]
\MeanThinningEfficiency
	\end{lem} }
	\addtocounter{lem}{-1}

\begin{proof}
Consider any $m \geq c_s n \log n$ and any interval $[t_0, t_1]$ with $t_1 := t_0 + n^3 \log^4 n$. By \eqref{eq:recovery_many_good_quantiles_whp} in the recovery lemma (\cref{lem:recovery}), there exists a constant $\eps \in \big(0, \frac{1}{2}\big)$ such that,
\[
\Pro{ G_{t_0}^{t_1} > \eps^2 \cdot (t_1 - t_0) \;\left|\; \mathfrak{F}^{t_0}, \max_{i \in [n]} \left| y_i^{t_0} \right| \leq \C{poly_n_gap} n \log n \right.} \geq 1 - 2n^{-12}.
\]
Hence, using \cref{lem:initial_gap_nlogn},
 \[
\Pro{ G_{t_0}^{t_1} > \eps^2 \cdot (t_1 - t_0)} \geq \left(1 - 2n^{-12} \right) \cdot \left(1 - n^{-12}\right) \geq  1 - 3n^{-12}.
\]
Let $S_{t_0}^{t_1}$ be the number of samples in the interval $(t_0, t_1]$. Note that \MeanThinning takes only one sample in round $t$ if the first sample is an underloaded bin, i.e., with probability $1 - \delta^t$, and when $\delta^t \in [\eps, 1 - \eps]$, it follows that this probability is at least $\geq \eps$.

Consider the following coupling between the \MeanThinning process and the random variables $Z_1, \ldots, Z_{N} \sim \mathsf{Ber}(\eps)$ for $N := \eps^2 \cdot (t_1 - t_0)$. For any $t \in [t_0, t_1)$, having initially $j^{t_0} = 1$:
\begin{quote}
If $\delta^t \in [\eps, 1 - \eps]$ and $j^t \leq N$:
\begin{itemize}
  \item If $Z^{j^t} = 1$, then allocate to an underloaded bin $i \in B_-^t$ with probability $\frac{1}{n \cdot (1 - \delta^t)}$, and set $j^{t+1} := j^t + 1$.
  \item Otherwise: Allocate to bin $i \in [n]$ with probability 
  \[
    \begin{cases}
        \frac{1}{1 - \eps} \cdot \frac{\delta^t}{n} & \text{if }i \in B_+^t, \\
        \frac{1}{1 - \eps} \cdot \left( \frac{1 + \delta^t}{n} - \frac{\eps}{n \cdot (1 - \delta^t)} \right) & \text{if }i \in B_-^t,
    \end{cases}
  \]
  and set $j^{t+1} := j^t$.
\end{itemize}

Otherwise: Allocate to bin $i \in [n]$ with probability $p_i^t$.
\end{quote}
The marginal distribution of this coupling is correct, since any bin $i \in B_{+}^t$ is allocated to with probability
$
  (1 - \eps) \cdot \frac{1}{1 - \eps} \cdot \frac{\delta^t}{n} = \frac{\delta^t}{n},
$
whereas any bin $i \in B_{-}^t$ is allocated to with probability
\[
 \eps \cdot \frac{1}{n \cdot (1-\delta^t)} + (1 - \eps) \cdot \frac{1}{1 - \eps} \cdot \left( \frac{1+\delta^t}{n} - \frac{\eps}{n \cdot (1-\delta^t)} \right)
 = \frac{1 + \delta^t}{n}.
\]
Next observe that, under the above coupling, $Z^t = 1$ 
implies that the first sample of \MeanThinning is an underloaded bin, and thus only one sample is used. This implies,
\[
S_{t_0}^{t_1} \leq \sum_{t = 1}^{j^{t_1} - 1} Z^t + 2 \cdot \left(t_1 - t_0 - \sum_{t = 1}^{j^{t_1} - 1} Z^t \right).
\]
By a Chernoff bound, and since $N=\Omega(n \log n)$, we have that \[
  \Pro{\sum_{t = 1}^{N} Z^t \geq \frac{1}{2} N \cdot \eps } \geq 1 - n^{-12}.
\]
By taking the union bound we have that,
\[
\Pro{\left\{ G_{t_0}^{t_1} > \eps^2 \cdot (t_1 - t_0) \right\} \cap \left\{ \sum_{t = 1}^{N} Z^t \geq \frac{1}{2} N \cdot \eps \right\}} \geq 1 - 4n^{-12}.
\]
Note that $\left\{ G_{t_0}^{t_1} > \eps^2 \cdot (t_1 - t_0) \right\}$ implies that $\{ j^{t_1} = N + 1 \}$.  Therefore, with probability at least $1 - 4n^{-12}$ we have that
\begin{align*}
S_{t_0}^{t_1} 
  \leq 1 \cdot \left( \sum_{t = 1}^{N} Z^t \right) + 2 \cdot \left( t_1 - t_0 - \sum_{t = 1}^{N} Z^t \right) \leq (t_1 - t_0) \cdot \left(2 - \frac{\eps^3}{2} \right).
\end{align*}
For the first $t_0 := \mathrm{rem}(m, n^3 \log^4 n)$ rounds, we consider two cases. If $t_0 \leq c_s n \log n$, then set $\xi := 1$ and we shall use the trivial estimate $S_0^{t_0} \leq 2t_0$. Otherwise, for  $t_0 \in [c_s n \log n, n^3 \log^4 n]$, set $\xi := 0$ and by \cref{lem:many_good_quantiles_in_lightly_loaded}, we obtain,
\[
\Pro{G_{0}^{t_0} > \eps^2 \cdot t_0} \geq 1 - n^{-1/2},
\]
and so using the coupling above
\[
 \Pro{S_{0}^{t_0} 
  \leq t_0 \cdot \left(2 - \frac{\eps^3}{2} \right)} \geq 1 - 2n^{-1/2}.
\]
Let $k := \lfloor \frac{m}{n^3 \log^4 n } \rfloor$ and define
\[
L^0 := \mathbf{1}_{S_{0}^{t_0} > t_0 \cdot \left(2 - \frac{\eps^3}{2} \right)},
\quad 
\text{ and }
\quad 
L^j := \mathbf{1}_{S_{t_0 + j \cdot n^3 \log^4 n }^{t_0 + (j+1) \cdot n^3 \log^4 n} > n^3 \log^4 n \cdot \left(2 - \frac{\eps^3}{2} \right)},
\]
where $0 < j < k$. Then for $L := \sum_{j = \xi}^{k-1} L^j$, we have that $\Ex{L} \leq 2n^{-1/2} \cdot k$, and hence by Markov's inequality,
\[
  \Pro{L \leq 2n^{-1/4} \cdot k} \geq 1 - n^{-1/4}.
\]
When $\{ L \leq 2n^{-1/4} \cdot k \}$ holds, we have that
\[
S^m = S_{0}^{m} 
 \leq 2 \cdot ( 2n^{-1/4} + \xi \cdot n^{-2}) \cdot m + \left( 2 - \frac{\eps^3}{2} \right) \cdot (1 - 2n^{-1/4} - \xi \cdot n^{-2}) \cdot m
 \leq \left( 2 - \frac{\eps^3}{4} \right) \cdot m,
\]
using that $t_0/m \leq n^{-2}$ when $\xi := 1$. Further, we obtain \[
  \eta^m = \frac{W^m}{S^m} \geq  \frac{m}{\left( 2 - \frac{\eps^3}{4} \right) \cdot m} \geq \frac{1}{2} + \frac{\eps^3}{8},
\]
using that $\frac{1}{1-z} \geq 1 + z$ for any $z \in (0, 1)$. Finally,
\begin{align*}
 \Ex{\eta^m} 
   & \geq \Ex{ \eta^m \, \left| \, S_0^{m} \leq \left( 2 - \frac{\eps^3}{4} \right) \cdot m \right.} \cdot \Pro{S_0^{m} \leq \left( 2 - \frac{\eps^3}{4} \right) \cdot m} \\
   & \qquad + \frac{1}{2} \cdot \Pro{S_0^{m} > \left( 2 - \frac{\eps^3}{4} \right) \cdot m} \geq \frac{1}{2} + \frac{\eps^3}{16}.
\end{align*}
The statement follows by choosing $\rho := \min\big\{ \frac{\eps^3}{16}, \frac{1}{c_s}\big\}$.
\end{proof}

Next, we show that \Twinning is more sample-efficient than \OneChoice.

{\renewcommand{\thelem}{\ref{lem:twinning_sample_efficient}}
	\begin{lem}[Restated, page~\pageref{lem:twinning_sample_efficient}]
\TwinningSampleEfficiency
	\end{lem} }
	\addtocounter{lem}{-1}
 
\begin{proof}
The proof proceeds similarly to the proof of \cref{lem:mean_thinning_efficiency}, but here we lower bound $W^m$ instead of upper bounding $S^m$. In round $t$, \Twinning allocates two balls with probability $1 - \delta^t$, which is $\geq \eps$ when $\delta^t \in [\eps, 1 - \eps]$.

Similarly to \cref{lem:mean_thinning_efficiency}, we define the following coupling between \Twinning and random variables $Z_1, \ldots, Z_{N} \sim \mathsf{Ber}(\eps)$ for $N := \eps^2 \cdot (t_1 - t_0)$. For any $t \in [t_0, t_1)$, having initially $j^{t_0} = 1$:
\begin{quote}
If $\delta^t \in [\eps, 1 - \eps]$ and $j^t \leq N$:
\begin{itemize}
  \item If $Z^{j^t} = 1$, then allocate to an underloaded bin $i \in B_-^t$ with probability $\frac{1}{n \cdot (1 - \delta^t)}$, and set $j^{t+1} := j^t + 1$.
  \item Otherwise: Allocate to bin $i \in [n]$ with probability 
  \[
    \begin{cases}
        \frac{1}{1 - \eps} \cdot \frac{1}{n} & \text{if }i \in B_+^t, \\
        \frac{1}{1 - \eps} \cdot \left( \frac{1}{n} - \frac{\eps}{n \cdot (1 - \delta^t)} \right) & \text{if }i \in B_-^t,
    \end{cases}
  \]
  and set $j^{t+1} := j^t$.
\end{itemize}

Otherwise: Allocate to bin $i \in [n]$ with probability $\frac{1}{n}$.
\end{quote}
The marginal distribution of this coupling is correct, since any bin $i \in B_{+}^t$ is allocated to with probability
$
  (1 - \eps) \cdot \frac{1}{1 - \eps} \cdot \frac{1}{n} = \frac{1}{n},
$
whereas any bin $i \in B_{-}^t$ is allocated to with probability
\[
 \eps \cdot \frac{1}{n} + (1 - \eps) \cdot \frac{1}{1 - \eps} \cdot \left( \frac{1}{n} - \frac{\eps}{n} \right)
 = \frac{1}{n}.
\]
Next observe that, under the above coupling, $Z^t = 1$ implies that the first sample of \MeanThinning is an underloaded bin, and thus only one sample is used. This implies,
\[
W_{t_0}^{t_1} \geq 2 \cdot \sum_{t = 1}^{j^{t_1} - 1} Z^t + 1 \cdot \left(t_1 - t_0 - \sum_{t = 1}^{j^{t_1} - 1} Z^t \right).
\]
Proceeding as in \cref{lem:mean_thinning_efficiency} (with the main difference being that we do not need to apply the inequality $\frac{1}{1 - z} \geq 1 + z$), we obtain
\[
 \Pro{\eta^m \geq 1 + \frac{\eps^3}{4}} \geq 1 - n^{-1/4},
\]
and thus also $  \Ex{\eta^m} \geq 1 +  \eps^3/8. $
\end{proof}

\section{Experimental Results} \label{sec:experiemental_results}

In this section, we present some experimental results for the \MeanThinning and \Twinning processes (\cref{tab:gap_distribution} and \cref{fig:gap_vs_bins_alt}) and compare their gap with that of a $(1 + \beta)$-process with $\beta = 1/2$, a $\Quantile(1/2)$ process, and that of the \TwoChoice process. 

\colorlet{GA}{black!40!white}
\colorlet{GB}{black!70!white}
\colorlet{GC}{black}
\newcommand{\CA}[1]{\textcolor{GA}{#1}} %
\newcommand{\CB}[1]{\textcolor{GB}{#1}} %
\newcommand{\CC}[1]{\textcolor{GC}{#1}} %
\newcommand{\CI}[2]{\FPeval{\result}{min(30 + 3 * #1, 100)} \colorlet{tmpC}{black!\result!white} {\textcolor{tmpC}{#2}}}

\begin{figure}[H]
    \centering
    \includegraphics[height=4.7cm]{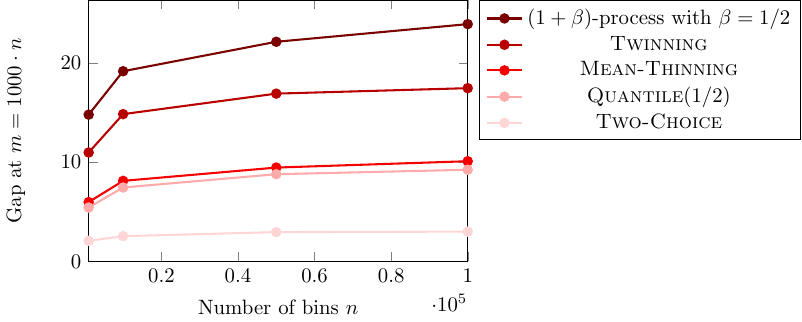}
    \caption{Average Gap vs. number of bins~$n \in \{ 10^3, 10^4, 5 \cdot 10^4, 10^5\}$ for the experimental setup of \cref{tab:gap_distribution}.}
    \label{fig:gap_vs_bins_alt}
\end{figure}

\begin{table}[H]
	\centering
	\scriptsize{
		\begin{tabular}{|c|c|c|c|c|c|}
			\hline
			$n$ & \makecell{$(1+\beta)$-process\\ with $\beta = 1/2$} & \Twinning & \MeanThinning & $\Quantile(1/2)$ & \TwoChoice \\ \hline
			$10^3$ &
            \makecell{
                \CI{5}{\textbf{12} : \ 5\% } \\
                \CI{15}{\textbf{13} : 15\% } \\
                \CI{31}{\textbf{14} : 31\% } \\
                \CI{21}{\textbf{15} : 21\% } \\
                \CI{15}{\textbf{16} : 15\% } \\
                \CI{5}{\textbf{17} : \ 5\% } \\
                \CI{4}{\textbf{18} : \ 4\% } \\
                \CI{2}{\textbf{19} : \ 2\% } \\
                \CI{1}{\textbf{20} : \ 1\% } \\
                \CI{1}{\textbf{21} : \ 1\% } } &
			\makecell{
				\CI{3}{\textbf{\ 8} : \ 3\%} \\
				\CI{21}{\textbf{\ 9} : 21\%} \\
				\CI{25}{\textbf{10} : 25\%} \\
				\CI{18}{\textbf{11} : 18\%} \\
				\CI{13}{\textbf{12} : 13\%} \\
				\CI{8}{\textbf{13} : \ 8\%} \\
				\CI{9}{\textbf{14} : \ 9\%} \\
				\CI{3}{\textbf{17} : \ 3\%}} &
			\makecell{
				\CI{2}{\textbf{\ 4} : \ 2\%} \\
				\CI{38}{\textbf{\ 5} : 38\%} \\ 
				\CI{35}{\textbf{\ 6} : 35\%} \\
				\CI{15}{\textbf{\ 7} : 15\%} \\
				\CI{8}{\textbf{\ 8} : \ 8\%} \\
				\CI{1}{\textbf{\ 9} : \ 1\%} \\
				\CI{1}{\textbf{12} : \ 1\%}} &
            \makecell{
                \CI{1}{\textbf{\ 3} : \ 1\% } \\ 
                \CI{11}{\textbf{\ 4} : 11\% } \\ 
                \CI{46}{\textbf{\ 5} : 46\% } \\
                \CI{33}{\textbf{\ 6} : 33\% } \\
                \CI{6}{\textbf{\ 7} : \ 6\% } \\
                \CI{2}{\textbf{\ 8} : \ 2\% } \\
                \CI{1}{\textbf{10} : \ 1\% } } &
            \makecell{
                \CI{93}{\textbf{2} : 93\% } \\
                \CI{7}{\textbf{3} : \ 7\% }} \\ \hline
			$10^4$ &
            \makecell{
                \CI{3}{\textbf{16} : \ 3\% } \\
                \CI{21}{\textbf{17} : 21\% } \\
                \CI{19}{\textbf{18} : 19\% } \\
                \CI{10}{\textbf{19} : 10\% } \\
                \CI{23}{\textbf{20} : 23\% } \\
                \CI{11}{\textbf{21} : 11\% } \\
                \CI{10}{\textbf{22} : 10\% } \\
                \CI{2}{\textbf{23} : \ 2\% } \\
                \CI{1}{\textbf{24} : \ 1\% }} &
			\makecell{
				\CI{1}{\textbf{11} : \ 1\%} \\
				\CI{9}{\textbf{12} : \ 9\%} \\
				\CI{24}{\textbf{13} : 24\%} \\
				\CI{22}{\textbf{14} : 22\%} \\
				\CI{13}{\textbf{15} : 13\%} \\
				\CI{9}{\textbf{16} : \ 9\%} \\
				\CI{8}{\textbf{17} : \ 8\%} \\
				\CI{5}{\textbf{18} : \ 5\%} \\
				\CI{6}{\textbf{19} : \ 6\%} \\
				\CI{1}{\textbf{20} : \ 1\%} \\
				\CI{1}{\textbf{21} : \ 1\%} \\
				\CI{1}{\textbf{26} : \ 1\%} } &
			\makecell{
				\CI{2}{\textbf{\ 6} : \ 2\%} \\ 
				\CI{30}{\textbf{\ 7} : 30\%} \\
				\CI{38}{\textbf{\ 8} : 38\%} \\
				\CI{19}{\textbf{\ 9} : 19\%} \\
				\CI{9}{\textbf{10} : \ 9\%} \\
				\CI{1}{\textbf{11} : \ 1\%} \\
				\CI{1}{\textbf{14} : \ 1\%}} &
            \makecell{
                \CI{14}{\textbf{\ 6} : 14\% } \\
                \CI{42}{\textbf{\ 7} : 42\% } \\
                \CI{25}{\textbf{\ 8} : 25\% } \\
                \CI{15}{\textbf{\ 9} : 15\% } \\
                \CI{2}{\textbf{10} : \ 2\% } \\
                \CI{1}{\textbf{11} : \ 1\% } \\
                \CI{1}{\textbf{12} : \ 1\% } } &
            \makecell{
                \CI{46}{\textbf{2} : 46\% } \\
                \CI{54}{\textbf{3} : 54\% }} \\ \hline
			$10^5$ & 
            \makecell{
                \CI{2}{\textbf{20} : \ 2\% } \\ 
                \CI{7}{\textbf{21} : \ 7\% } \\
                \CI{9}{\textbf{22} : \ 9\% } \\
                \CI{26}{\textbf{23} : 26\% } \\
                \CI{27}{\textbf{24} : 27\% } \\
                \CI{14}{\textbf{25} : 14\% } \\
                \CI{6}{\textbf{26} : \ 6\% } \\
                \CI{3}{\textbf{27} : \ 3\% } \\
                \CI{4}{\textbf{28} : \ 4\% } \\
                \CI{1}{\textbf{29} : \ 1\% } \\
                \CI{1}{\textbf{34} : \ 1\% }} &
			\makecell{
				\CI{2}{\textbf{14} : \ 2\%} \\
				\CI{5}{\textbf{15} : \ 5\%} \\
				\CI{25}{\textbf{16} : 25\%} \\
				\CI{28}{\textbf{17} : 28\%} \\
				\CI{17}{\textbf{18} : 17\%} \\
				\CI{10}{\textbf{19} : 10\%} \\
				\CI{8}{\textbf{20} : \ 8\%} \\
				\CI{1}{\textbf{21} : \ 1\%} \\
				\CI{1}{\textbf{22} : \ 1\%} \\
				\CI{3}{\textbf{23} : \ 3\%} } &
			\makecell{
				\CI{3}{\textbf{\ 8} : \ 3\%} \\ 
				\CI{32}{\textbf{\ 9} : 32\%} \\
				\CI{38}{\textbf{10} : 38\%} \\
				\CI{15}{\textbf{11} : 15\%} \\
				\CI{6}{\textbf{12} : \ 6\%} \\
				\CI{3}{\textbf{13} : \ 3\%} \\
				\CI{3}{\textbf{14} : \ 3\%} } &
            \makecell{
                \CI{28}{\textbf{\ 8} : 28\% } \\
                \CI{42}{\textbf{\ 9} : 42\% } \\
                \CI{18}{\textbf{10} : 18\% } \\
                \CI{7}{\textbf{11} : \ 7\% } \\
                \CI{3}{\textbf{12} : \ 3\% } \\
                \CI{1}{\textbf{14} : \ 1\% } \\
                \CI{1}{\textbf{15} : \ 1\% }} &
            \makecell{
                \CI{100}{\textbf{3} : 100\% }} \\ \hline
	\end{tabular}}
	\caption{Empirical gap distribution  for $n \in \{ 10^3, 10^4, 10^5\}$ bins and $m=1000 \cdot n$ balls, for $100$ repetitions. The observed gap values are in bold and next to that is the $\%$ of runs where a particular gap value was observed.}
	\label{tab:gap_distribution}
\end{table}

\section{Conclusions} \label{sec:conclusions}

In this work  we introduced a framework for analyzing \MeanBiased processes in the heavily loaded case. These include processes that bias allocations (of single balls) towards underloaded bins as opposed to overloaded bins, the prototypical process being \MeanThinning. These also include processes that weight-bias allocations to underloaded bins as opposed to overloaded, the prototypical process being \Twinning. Our analysis which uses a complex interplay between three different types of potential functions implies an $\Oh(\log n)$ upper bound on the gap for any process in the framework. For \MeanThinning and \Twinning we proved a matching lower bound on the gap and show that the processes are more sample-efficient than \TwoChoice and \OneChoice respectively. By a coupling argument, we also obtain tight bounds for $\RelativeThreshold(f(n))$ processes with any $f(n) \geq \log n$.

There are several directions for future work. One avenue is to derive sub-logarithmic gap bounds. 
More specifically, one could investigate \KRelativeThreshold processes either with $k \geq 2$, or with $k = 1$ and $f(n) \in (0, \log n)$; we conjecture that for a suitable choice of $f(n)$ (being poly-logarithmic in $n$), a gap bound of $\Oh(\frac{\log n}{\log \log n})$ holds. A second direction is to extend the \MeanBiased framework to noisy settings, where load information is outdated~\cite{LS22Queries}, or dynamic settings, where in addition  balls could also be removed~\cite{CFMMRSU98,BK22}.

\section*{Acknowledgments} 
We thank David Croydon and Martin Krejca for some helpful discussions.

\setlength{\bibsep}{0pt plus 0.3ex}

\addcontentsline{toc}{section}{Bibliography}
\renewcommand{\bibsection}{\section*{Bibliography}}
\bibliographystyle{ACM-Reference-Format-CAM}
\bibliography{bibliography}

%%% -*-BibTeX-*-
%%% Do NOT edit. File created by BibTeX with style
%%% ACM-Reference-Format-Journals [18-Jan-2012].

\begin{thebibliography}{46}

%%% ====================================================================
%%% NOTE TO THE USER: you can override these defaults by providing
%%% customized versions of any of these macros before the \bibliography
%%% command.  Each of them MUST provide its own final punctuation,
%%% except for \shownote{}, \showDOI{}, and \showURL{}.  The latter two
%%% do not use final punctuation, in order to avoid confusing it with
%%% the Web address.
%%%
%%% To suppress output of a particular field, define its macro to expand
%%% to an empty string, or better, \unskip, like this:
%%%
%%% \newcommand{\showDOI}[1]{\unskip}   % LaTeX syntax
%%%
%%% \def \showDOI #1{\unskip}           % plain TeX syntax
%%%
%%% ====================================================================

\ifx \showCODEN    \undefined \def \showCODEN     #1{\unskip}     \fi
\ifx \showDOI      \undefined \def \showDOI       #1{#1}\fi
\ifx \showISBNx    \undefined \def \showISBNx     #1{\unskip}     \fi
\ifx \showISBNxiii \undefined \def \showISBNxiii  #1{\unskip}     \fi
\ifx \showISSN     \undefined \def \showISSN      #1{\unskip}     \fi
\ifx \showLCCN     \undefined \def \showLCCN      #1{\unskip}     \fi
\ifx \shownote     \undefined \def \shownote      #1{#1}          \fi
\ifx \showarticletitle \undefined \def \showarticletitle #1{#1}   \fi
\ifx \showURL      \undefined \def \showURL       {\relax}        \fi
% The following commands are used for tagged output and should be
% invisible to TeX
\providecommand\bibfield[2]{#2}
\providecommand\bibinfo[2]{#2}
\providecommand\natexlab[1]{#1}
\providecommand\showeprint[2][]{arXiv:#2}

\bibitem[Alistarh et~al\mbox{.}(2018a)]%
        {AAG18}
\bibfield{author}{\bibinfo{person}{Dan Alistarh}, \bibinfo{person}{James
  Aspnes}, {and} \bibinfo{person}{Rati Gelashvili}.}
\newblock \showarticletitle{Space-Optimal Majority in Population Protocols}. In
  \bibinfo{booktitle}{\emph{\SODA{29}{18}}}. \bibinfo{publisher}{{SIAM}},
  \bibinfo{pages}{2221--2239}.
\newblock
\href{https://doi.org/10.1137/1.9781611975031.144}{\texttt{doi}}


\bibitem[Alistarh et~al\mbox{.}(2018b)]%
        {ABKLN18}
\bibfield{author}{\bibinfo{person}{Dan Alistarh}, \bibinfo{person}{Trevor
  Brown}, \bibinfo{person}{Justin Kopinsky}, \bibinfo{person}{Jerry~Zheng Li},
  {and} \bibinfo{person}{Giorgi Nadiradze}.}
\newblock \showarticletitle{Distributionally Linearizable Data Structures}. In
  \bibinfo{booktitle}{\emph{\SPAA{30}{18}}}. \bibinfo{publisher}{{ACM}},
  \bibinfo{pages}{133--142}.
\newblock
\href{https://doi.org/10.1145/3210377.3210411}{\texttt{doi}}


\bibitem[Alistarh et~al\mbox{.}(2021)]%
        {AGR21}
\bibfield{author}{\bibinfo{person}{Dan Alistarh}, \bibinfo{person}{Rati
  Gelashvili}, {and} \bibinfo{person}{Joel Rybicki}.}
\newblock \showarticletitle{Fast Graphical Population Protocols}. In
  \bibinfo{booktitle}{\emph{\OPODIS{25}{21}}}, Vol.~\bibinfo{volume}{217}.
  \bibinfo{publisher}{Schloss Dagstuhl - Leibniz-Zentrum f{\"{u}}r Informatik},
  \bibinfo{pages}{14:1--14:18}.
\newblock
\href{https://doi.org/10.4230/LIPIcs.OPODIS.2021.14}{\texttt{doi}}


\bibitem[Alistarh et~al\mbox{.}(2017)]%
        {AK0N17}
\bibfield{author}{\bibinfo{person}{Dan Alistarh}, \bibinfo{person}{Justin
  Kopinsky}, \bibinfo{person}{Jerry Li}, {and} \bibinfo{person}{Giorgi
  Nadiradze}.}
\newblock \showarticletitle{The Power of Choice in Priority Scheduling}. In
  \bibinfo{booktitle}{\emph{\PODC{36}{17}}}. \bibinfo{publisher}{{ACM}},
  \bibinfo{pages}{283--292}.
\newblock
\href{https://doi.org/10.1145/3087801.3087810}{\texttt{doi}}


\bibitem[Alistarh et~al\mbox{.}(2022)]%
        {ANS22}
\bibfield{author}{\bibinfo{person}{Dan Alistarh}, \bibinfo{person}{Giorgi
  Nadiradze}, {and} \bibinfo{person}{Amirmojtaba Sabour}.}
\newblock \showarticletitle{Dynamic Averaging Load Balancing on Cycles}.
\newblock \bibinfo{journal}{\emph{Algorithmica}} \bibinfo{volume}{84},
  \bibinfo{number}{4} (\bibinfo{year}{2022}), \bibinfo{pages}{1007--1029}.
\newblock
\href{https://doi.org/10.1007/s00453-021-00905-9}{\texttt{doi}}


\bibitem[Azar et~al\mbox{.}(2020)]%
        {ACM20}
\bibfield{author}{\bibinfo{person}{Yossi Azar}, \bibinfo{person}{Andrei~Z.
  Broder}, \bibinfo{person}{Anna~R. Karlin}, \bibinfo{person}{Michael
  Mitzenmacher}, {and} \bibinfo{person}{Eli Upfal}.}
  \bibinfo{year}{2020}\natexlab{}.
\newblock \bibinfo{title}{{The ACM Paris Kanellakis Theory and Practice
  Award}}.
\newblock
\newblock
\newblock
\shownote{\url{https://www.acm.org/media-center/2021/may/technical-awards-2020}}.


\bibitem[Azar et~al\mbox{.}(1999)]%
        {ABKU99}
\bibfield{author}{\bibinfo{person}{Yossi Azar}, \bibinfo{person}{Andrei~Z.
  Broder}, \bibinfo{person}{Anna~R. Karlin}, {and} \bibinfo{person}{Eli
  Upfal}.}
\newblock \showarticletitle{Balanced allocations}.
\newblock \bibinfo{journal}{\emph{SIAM J. Comput.}} \bibinfo{volume}{29},
  \bibinfo{number}{1} (\bibinfo{year}{1999}), \bibinfo{pages}{180--200}.
\newblock
\showISSN{0097-5397}
\href{https://doi.org/10.1137/S0097539795288490}{\texttt{doi}}


\bibitem[Bansal and Feldheim(2022)]%
        {BF22}
\bibfield{author}{\bibinfo{person}{Nikhil Bansal} {and}
  \bibinfo{person}{Ohad~N. Feldheim}.}
\newblock \showarticletitle{The power of two choices in graphical allocation}.
  In \bibinfo{booktitle}{\emph{\STOC{54}{22}}}. \bibinfo{publisher}{{ACM}},
  \bibinfo{pages}{52--63}.
\newblock
\href{https://doi.org/10.1145/3519935.3519995}{\texttt{doi}}


\bibitem[Bansal and Kuszmaul(2022)]%
        {BK22}
\bibfield{author}{\bibinfo{person}{Nikhil Bansal} {and}
  \bibinfo{person}{William Kuszmaul}.}
\newblock \showarticletitle{Balanced Allocations: The Heavily Loaded Case with
  Deletions}. In \bibinfo{booktitle}{\emph{\FOCS{63}{22}}}.
  \bibinfo{publisher}{{IEEE}}, \bibinfo{pages}{801--812}.
\newblock
\href{https://doi.org/10.1109/FOCS54457.2022.00081}{\texttt{doi}}


\bibitem[Berenbrink et~al\mbox{.}(2006)]%
        {BCSV06}
\bibfield{author}{\bibinfo{person}{Petra Berenbrink}, \bibinfo{person}{Artur
  Czumaj}, \bibinfo{person}{Angelika Steger}, {and} \bibinfo{person}{Berthold
  V\"{o}cking}.}
\newblock \showarticletitle{Balanced allocations: the heavily loaded case}.
\newblock \bibinfo{journal}{\emph{SIAM J. Comput.}} \bibinfo{volume}{35},
  \bibinfo{number}{6} (\bibinfo{year}{2006}), \bibinfo{pages}{1350--1385}.
\newblock
\showISSN{0097-5397}
\href{https://doi.org/10.1137/S009753970444435X}{\texttt{doi}}


\bibitem[Berenbrink et~al\mbox{.}(2008)]%
        {BFHM08}
\bibfield{author}{\bibinfo{person}{Petra Berenbrink}, \bibinfo{person}{Tom
  Friedetzky}, \bibinfo{person}{Zengjian Hu}, {and} \bibinfo{person}{Russell
  Martin}.}
\newblock \showarticletitle{On weighted balls-into-bins games}.
\newblock \bibinfo{journal}{\emph{Theoret. Comput. Sci.}}
  \bibinfo{volume}{409}, \bibinfo{number}{3} (\bibinfo{year}{2008}),
  \bibinfo{pages}{511--520}.
\newblock
\showISSN{0304-3975}
\href{https://doi.org/10.1016/j.tcs.2008.09.023}{\texttt{doi}}


\bibitem[Berenbrink et~al\mbox{.}(2013)]%
        {BKSS13}
\bibfield{author}{\bibinfo{person}{Petra Berenbrink}, \bibinfo{person}{Kamyar
  Khodamoradi}, \bibinfo{person}{Thomas Sauerwald}, {and}
  \bibinfo{person}{Alexandre Stauffer}.}
\newblock \showarticletitle{Balls-into-bins with nearly optimal load
  distribution}. In \bibinfo{booktitle}{\emph{\SPAA{25}{13}}}.
  \bibinfo{publisher}{{ACM}}, \bibinfo{pages}{326--335}.
\newblock
\href{https://doi.org/10.1145/2486159.2486191}{\texttt{doi}}


\bibitem[Celis et~al\mbox{.}(2011)]%
        {CRSW11}
\bibfield{author}{\bibinfo{person}{L.~Elisa Celis}, \bibinfo{person}{Omer
  Reingold}, \bibinfo{person}{Gil Segev}, {and} \bibinfo{person}{Udi Wieder}.}
\newblock \showarticletitle{Balls and Bins: Smaller Hash Families and Faster
  Evaluation}. In \bibinfo{booktitle}{\emph{\FOCS{52}{11}}}.
  \bibinfo{publisher}{{IEEE}}, \bibinfo{pages}{599--608}.
\newblock


\bibitem[Chung and Lu(2006)]%
        {CL06}
\bibfield{author}{\bibinfo{person}{Fan Chung} {and} \bibinfo{person}{Linyuan
  Lu}.}
\newblock \showarticletitle{Concentration inequalities and martingale
  inequalities: a survey}.
\newblock \bibinfo{journal}{\emph{Internet Math.}} \bibinfo{volume}{3},
  \bibinfo{number}{1} (\bibinfo{year}{2006}), \bibinfo{pages}{79--127}.
\newblock
\showISSN{1542-7951}
\urldef\tempurl%
\url{http://projecteuclid.org/euclid.im/1175266369}
\showURL{%
\tempurl}


\bibitem[Cole et~al\mbox{.}(1998)]%
        {CFMMRSU98}
\bibfield{author}{\bibinfo{person}{Richard Cole}, \bibinfo{person}{Alan
  Frieze}, \bibinfo{person}{Bruce~M. Maggs}, \bibinfo{person}{Michael
  Mitzenmacher}, \bibinfo{person}{Andr\'{e}a~W. Richa}, \bibinfo{person}{Ramesh
  Sitaraman}, {and} \bibinfo{person}{Eli Upfal}.}
\newblock \showarticletitle{On balls and bins with deletions}. In
  \bibinfo{booktitle}{\emph{\RANDOM{2}{98}}}, Vol.~\bibinfo{volume}{1518}.
  \bibinfo{publisher}{Springer, Berlin}, \bibinfo{pages}{145--158}.
\newblock
\href{https://doi.org/10.1007/3-540-49543-6\_12}{\texttt{doi}}


\bibitem[Czumaj and Stemann(2001)]%
        {CS01}
\bibfield{author}{\bibinfo{person}{Artur Czumaj} {and} \bibinfo{person}{Volker
  Stemann}.}
\newblock \showarticletitle{Randomized allocation processes}.
\newblock \bibinfo{journal}{\emph{Random Structures \& Algorithms}}
  \bibinfo{volume}{18}, \bibinfo{number}{4} (\bibinfo{year}{2001}),
  \bibinfo{pages}{297--331}.
\newblock
\showISSN{1042-9832}
\href{https://doi.org/10.1002/rsa.1011}{\texttt{doi}}


\bibitem[Dubhashi and Panconesi(2009)]%
        {DubPan}
\bibfield{author}{\bibinfo{person}{Devdatt~P. Dubhashi} {and}
  \bibinfo{person}{Alessandro Panconesi}.} \bibinfo{year}{2009}\natexlab{}.
\newblock \bibinfo{booktitle}{\emph{Concentration of Measure for the Analysis
  of Randomized Algorithms}}.
\newblock \bibinfo{publisher}{Cambridge University Press},
  \bibinfo{address}{Cambridge}.
\newblock
\showISBNx{978-0-521-88427-3}


\bibitem[Dwivedi et~al\mbox{.}(2019)]%
        {DFG19}
\bibfield{author}{\bibinfo{person}{Raaz Dwivedi}, \bibinfo{person}{Ohad~N.
  Feldheim}, \bibinfo{person}{Ori Gurel-Gurevich}, {and}
  \bibinfo{person}{Aaditya Ramdas}.}
\newblock \showarticletitle{The power of online thinning in reducing
  discrepancy}.
\newblock \bibinfo{journal}{\emph{Probab. Theory Related Fields}}
  \bibinfo{volume}{174}, \bibinfo{number}{1-2} (\bibinfo{year}{2019}),
  \bibinfo{pages}{103--131}.
\newblock
\showISSN{0178-8051}
\href{https://doi.org/10.1007/s00440-018-0860-y}{\texttt{doi}}


\bibitem[{Eager} et~al\mbox{.}(1986)]%
        {ELZ86}
\bibfield{author}{\bibinfo{person}{Derek~L. {Eager}},
  \bibinfo{person}{Edward~D. {Lazowska}}, {and} \bibinfo{person}{John
  {Zahorjan}}.}
\newblock \showarticletitle{Adaptive load sharing in homogeneous distributed
  systems}.
\newblock \bibinfo{journal}{\emph{IEEE Transactions on Software Engineering}}
  \bibinfo{volume}{SE-12}, \bibinfo{number}{5} (\bibinfo{year}{1986}),
  \bibinfo{pages}{662--675}.
\newblock
\href{https://doi.org/10.1109/TSE.1986.6312961}{\texttt{doi}}


\bibitem[Feldheim and Gurel-Gurevich(2021)]%
        {FG18}
\bibfield{author}{\bibinfo{person}{Ohad~N. Feldheim} {and} \bibinfo{person}{Ori
  Gurel-Gurevich}.}
\newblock \showarticletitle{The power of thinning in balanced allocation}.
\newblock \bibinfo{journal}{\emph{Electron. Commun. Probab.}}
  \bibinfo{volume}{26} (\bibinfo{year}{2021}), \bibinfo{pages}{Paper No. 34,
  8}.
\newblock
\href{https://doi.org/10.1214/21-ecp400}{\texttt{doi}}


\bibitem[Feldheim et~al\mbox{.}(pear)]%
        {FGL21}
\bibfield{author}{\bibinfo{person}{Ohad~N. Feldheim}, \bibinfo{person}{Ori
  Gurel-Gurevich}, {and} \bibinfo{person}{Jiange Li}.}
\newblock \showarticletitle{Long-term balanced allocation via thinning}.
\newblock \bibinfo{journal}{\emph{The Annals of Applied Probability}}
  (\bibinfo{year}{\textit{to appear}}), \bibinfo{pages}{arXiv:2110.05009}.
\newblock


\bibitem[Feldheim and Li(2020)]%
        {FL20}
\bibfield{author}{\bibinfo{person}{Ohad~N. Feldheim} {and}
  \bibinfo{person}{Jiange Li}.}
\newblock \showarticletitle{Load balancing under {$d$}-thinning}.
\newblock \bibinfo{journal}{\emph{Electron. Commun. Probab.}}
  \bibinfo{volume}{25} (\bibinfo{year}{2020}), \bibinfo{pages}{Paper No. 1,
  13}.
\newblock
\href{https://doi.org/10.1214/19-ecp282}{\texttt{doi}}


\bibitem[Friedrich et~al\mbox{.}(2012)]%
        {FGS12}
\bibfield{author}{\bibinfo{person}{Tobias Friedrich}, \bibinfo{person}{Martin
  Gairing}, {and} \bibinfo{person}{Thomas Sauerwald}.}
\newblock \showarticletitle{Quasirandom Load Balancing}.
\newblock \bibinfo{journal}{\emph{{SIAM} J. Comput.}} \bibinfo{volume}{41},
  \bibinfo{number}{4} (\bibinfo{year}{2012}), \bibinfo{pages}{747--771}.
\newblock
\href{https://doi.org/10.1137/100799216}{\texttt{doi}}


\bibitem[Gibbens et~al\mbox{.}(1988)]%
        {GKK88}
\bibfield{author}{\bibinfo{person}{Richard~J. Gibbens},
  \bibinfo{person}{Frank~P. Kelly}, {and} \bibinfo{person}{Peter~B. Key}.}
\newblock \showarticletitle{Dynamic alternative routing -- modelling and
  behavior}. In \bibinfo{booktitle}{\emph{12th International Teletraffic
  Congress (ITC'88)}}. \bibinfo{publisher}{Elsevier, Amsterdam},
  \bibinfo{pages}{1019--1025}.
\newblock


\bibitem[Godfrey(2008)]%
        {G08}
\bibfield{author}{\bibinfo{person}{P.~Brighten Godfrey}.}
\newblock \showarticletitle{Balls and bins with structure: balanced allocations
  on hypergraphs}. In \bibinfo{booktitle}{\emph{\SODA{19}{08}}}.
  \bibinfo{publisher}{ACM}, \bibinfo{pages}{511--517}.
\newblock


\bibitem[Gupta et~al\mbox{.}(2020)]%
        {GK0020}
\bibfield{author}{\bibinfo{person}{Anupam Gupta}, \bibinfo{person}{Ravishankar
  Krishnaswamy}, \bibinfo{person}{Amit Kumar}, {and} \bibinfo{person}{Sahil
  Singla}.}
\newblock \showarticletitle{Online Carpooling Using Expander Decompositions}.
  In \bibinfo{booktitle}{\emph{\FSTTCS{40}{20}}}, Vol.~\bibinfo{volume}{182}.
  \bibinfo{publisher}{Schloss Dagstuhl - Leibniz-Zentrum f{\"{u}}r Informatik},
  \bibinfo{pages}{23:1--23:14}.
\newblock
\href{https://doi.org/10.4230/LIPIcs.FSTTCS.2020.23}{\texttt{doi}}


\bibitem[Iwama and Kawachi(2005)]%
        {IK04}
\bibfield{author}{\bibinfo{person}{Kazuo Iwama} {and} \bibinfo{person}{Akinori
  Kawachi}.}
\newblock \showarticletitle{Approximated Two Choices in Randomized Load
  Balancing}. In \bibinfo{booktitle}{\emph{Algorithms and Computation}}.
  \bibinfo{publisher}{Springer Berlin Heidelberg}, \bibinfo{pages}{545--557}.
\newblock
\showISBNx{978-3-540-30551-4}


\bibitem[Karp et~al\mbox{.}(1996)]%
        {KLM96}
\bibfield{author}{\bibinfo{person}{Richard~M. Karp}, \bibinfo{person}{Michael
  Luby}, {and} \bibinfo{person}{Friedhelm Meyer auf~der Heide}.}
\newblock \showarticletitle{Efficient {PRAM} simulation on a distributed memory
  machine}.
\newblock \bibinfo{journal}{\emph{Algorithmica}} \bibinfo{volume}{16},
  \bibinfo{number}{4-5} (\bibinfo{year}{1996}), \bibinfo{pages}{517--542}.
\newblock
\showISSN{0178-4617}
\href{https://doi.org/10.1007/s004539900063}{\texttt{doi}}


\bibitem[Los and Sauerwald(2022)]%
        {LS22Queries}
\bibfield{author}{\bibinfo{person}{Dimitrios Los} {and} \bibinfo{person}{Thomas
  Sauerwald}.}
\newblock \showarticletitle{Balanced Allocations with Incomplete Information:
  The Power of Two Queries}. In \bibinfo{booktitle}{\emph{\ITCS{13}{22}}},
  Vol.~\bibinfo{volume}{215}. \bibinfo{publisher}{Schloss Dagstuhl -
  Leibniz-Zentrum f{\"{u}}r Informatik}, \bibinfo{pages}{103:1--103:23}.
\newblock
\href{https://doi.org/10.4230/LIPIcs.ITCS.2022.103}{\texttt{doi}}


\bibitem[Los and Sauerwald(2023a)]%
        {LS23HerdPhenomenon}
\bibfield{author}{\bibinfo{person}{Dimitrios Los} {and} \bibinfo{person}{Thomas
  Sauerwald}.}
\newblock \showarticletitle{Balanced Allocations in Batches: The Tower of Two
  Choices}. In \bibinfo{booktitle}{\emph{\SPAA{35}{23}}}.
  \bibinfo{publisher}{{ACM}}, \bibinfo{pages}{51--61}.
\newblock
\href{https://doi.org/10.1145/3558481.3591088}{\texttt{doi}}


\bibitem[Los and Sauerwald(2023b)]%
        {LS23Noise}
\bibfield{author}{\bibinfo{person}{Dimitrios Los} {and} \bibinfo{person}{Thomas
  Sauerwald}.}
\newblock \showarticletitle{Balanced Allocations with the Choice of Noise}.
\newblock \bibinfo{journal}{\emph{J. ACM}} \bibinfo{volume}{70},
  \bibinfo{number}{6}, Article \bibinfo{articleno}{37} (\bibinfo{date}{nov}
  \bibinfo{year}{2023}), \bibinfo{numpages}{84}~pages.
\newblock
\showISSN{0004-5411}
\href{https://doi.org/10.1145/3625386}{\texttt{doi}}


\bibitem[Los and Sauerwald(2023c)]%
        {LS23RBB}
\bibfield{author}{\bibinfo{person}{Dimitrios Los} {and} \bibinfo{person}{Thomas
  Sauerwald}.}
\newblock \showarticletitle{{Tight Bounds for Repeated Balls-Into-Bins}}. In
  \bibinfo{booktitle}{\emph{\STACS{40}{23}}}, Vol.~\bibinfo{volume}{254}.
  \bibinfo{publisher}{Schloss Dagstuhl -- Leibniz-Zentrum f{\"u}r Informatik},
  \bibinfo{pages}{45:1--45:22}.
\newblock
\showISBNx{978-3-95977-266-2}
\showISSN{1868-8969}
\href{https://doi.org/10.4230/LIPIcs.STACS.2023.45}{\texttt{doi}}


\bibitem[Los et~al\mbox{.}(2022)]%
        {OurSODA}
\bibfield{author}{\bibinfo{person}{Dimitrios Los}, \bibinfo{person}{Thomas
  Sauerwald}, {and} \bibinfo{person}{John Sylvester}.}
\newblock \showarticletitle{Balanced Allocations: Caching and Packing, Twinning
  and Thinning}. In \bibinfo{booktitle}{\emph{\SODA{33}{22}}}.
  \bibinfo{publisher}{{SIAM}}, \bibinfo{pages}{1847--1874}.
\newblock
\href{https://doi.org/10.1137/1.9781611977073.74}{\texttt{doi}}


\bibitem[Los et~al\mbox{.}(pear)]%
        {LSS22Filling}
\bibfield{author}{\bibinfo{person}{Dimitrios Los}, \bibinfo{person}{Thomas
  Sauerwald}, {and} \bibinfo{person}{John Sylvester}.}
  \bibinfo{year}{\textit{to appear}}\natexlab{}.
\newblock \bibinfo{title}{The Power of Filling in Balanced Allocations}.
\newblock , \bibinfo{numpages}{arXiv:2110.05009}~pages.
\newblock


\bibitem[Lu et~al\mbox{.}(2011)]%
        {LXKGLG11}
\bibfield{author}{\bibinfo{person}{Yi Lu}, \bibinfo{person}{Qiaomin Xie},
  \bibinfo{person}{Gabriel Kliot}, \bibinfo{person}{Alan Geller},
  \bibinfo{person}{James~R. Larus}, {and} \bibinfo{person}{Albert~G.
  Greenberg}.}
\newblock \showarticletitle{Join-Idle-Queue: {A} novel load balancing algorithm
  for dynamically scalable web services}.
\newblock \bibinfo{journal}{\emph{Perform. Evaluation}} \bibinfo{volume}{68},
  \bibinfo{number}{11} (\bibinfo{year}{2011}), \bibinfo{pages}{1056--1071}.
\newblock
\href{https://doi.org/10.1016/j.peva.2011.07.015}{\texttt{doi}}


\bibitem[Mitzenmacher(1999)]%
        {M99}
\bibfield{author}{\bibinfo{person}{Michael Mitzenmacher}.}
\newblock \showarticletitle{On the Analysis of Randomized Load Balancing
  Schemes}.
\newblock \bibinfo{journal}{\emph{Theory Comput. Syst.}} \bibinfo{volume}{32},
  \bibinfo{number}{3} (\bibinfo{year}{1999}), \bibinfo{pages}{361--386}.
\newblock
\href{https://doi.org/10.1007/s002240000122}{\texttt{doi}}


\bibitem[Mitzenmacher et~al\mbox{.}(2002)]%
        {MPS02}
\bibfield{author}{\bibinfo{person}{Michael Mitzenmacher},
  \bibinfo{person}{Balaji Prabhakar}, {and} \bibinfo{person}{Devavrat Shah}.}
\newblock \showarticletitle{Load Balancing with Memory}. In
  \bibinfo{booktitle}{\emph{\FOCS{43}{02}}}. \bibinfo{publisher}{{IEEE}},
  \bibinfo{pages}{799--808}.
\newblock
\href{https://doi.org/10.1109/SFCS.2002.1182005}{\texttt{doi}}


\bibitem[Mitzenmacher et~al\mbox{.}(2001)]%
        {MRS01}
\bibfield{author}{\bibinfo{person}{Michael Mitzenmacher},
  \bibinfo{person}{Andr\'{e}a~W. Richa}, {and} \bibinfo{person}{Ramesh
  Sitaraman}.} \bibinfo{year}{2001}\natexlab{}.
\newblock \showarticletitle{The power of two random choices: a survey of
  techniques and results}.
\newblock In \bibinfo{booktitle}{\emph{Handbook of randomized computing, {V}ol.
  {I}, {II}}}. \bibinfo{series}{Comb. Optim.}, Vol.~\bibinfo{volume}{9}.
  \bibinfo{publisher}{Kluwer Acad. Publ.}, \bibinfo{address}{Dordrecht},
  \bibinfo{pages}{255--312}.
\newblock
\href{https://doi.org/10.1007/978-1-4615-0013-1\_9}{\texttt{doi}}


\bibitem[Ousterhout et~al\mbox{.}(2013)]%
        {OWZS13}
\bibfield{author}{\bibinfo{person}{Kay Ousterhout}, \bibinfo{person}{Patrick
  Wendell}, \bibinfo{person}{Matei Zaharia}, {and} \bibinfo{person}{Ion
  Stoica}.}
\newblock \showarticletitle{Sparrow: distributed, low latency scheduling}. In
  \bibinfo{booktitle}{\emph{{ACM} {SIGOPS} 24th Symposium on Operating Systems
  Principles, {SOSP} '13, Farmington, PA, USA, November 3-6, 2013}}.
  \bibinfo{publisher}{{ACM}}, \bibinfo{pages}{69--84}.
\newblock
\href{https://doi.org/10.1145/2517349.2522716}{\texttt{doi}}


\bibitem[Park(2017)]%
        {G17}
\bibfield{author}{\bibinfo{person}{Gahyun Park}.}
\newblock \showarticletitle{A generalization of multiple choice
  balls-into-bins: tight bounds}.
\newblock \bibinfo{journal}{\emph{Algorithmica}} \bibinfo{volume}{77},
  \bibinfo{number}{4} (\bibinfo{year}{2017}), \bibinfo{pages}{1159--1193}.
\newblock
\showISSN{0178-4617}
\href{https://doi.org/10.1007/s00453-016-0141-z}{\texttt{doi}}


\bibitem[Peres et~al\mbox{.}(2015)]%
        {PTW15}
\bibfield{author}{\bibinfo{person}{Yuval Peres}, \bibinfo{person}{Kunal
  Talwar}, {and} \bibinfo{person}{Udi Wieder}.}
\newblock \showarticletitle{Graphical balanced allocations and the
  {$(1+\beta)$}-choice process}.
\newblock \bibinfo{journal}{\emph{Random Structures \& Algorithms}}
  \bibinfo{volume}{47}, \bibinfo{number}{4} (\bibinfo{year}{2015}),
  \bibinfo{pages}{760--775}.
\newblock
\showISSN{1042-9832}
\href{https://doi.org/10.1002/rsa.20558}{\texttt{doi}}


\bibitem[Talwar and Wieder(2007)]%
        {TW07}
\bibfield{author}{\bibinfo{person}{Kunal Talwar} {and} \bibinfo{person}{Udi
  Wieder}.}
\newblock \showarticletitle{Balanced allocations: the weighted case}. In
  \bibinfo{booktitle}{\emph{\STOC{39}{07}}}. \bibinfo{publisher}{{ACM}},
  \bibinfo{pages}{256--265}.
\newblock
\href{https://doi.org/10.1145/1250790.1250829}{\texttt{doi}}


\bibitem[V{\"{o}}cking(1999)]%
        {V99}
\bibfield{author}{\bibinfo{person}{Berthold V{\"{o}}cking}.}
\newblock \showarticletitle{How Asymmetry Helps Load Balancing}. In
  \bibinfo{booktitle}{\emph{\FOCS{40}{99}}}. \bibinfo{publisher}{{IEEE}},
  \bibinfo{pages}{131--141}.
\newblock
\href{https://doi.org/10.1109/SFFCS.1999.814585}{\texttt{doi}}


\bibitem[Wieder(2007)]%
        {W07}
\bibfield{author}{\bibinfo{person}{Udi Wieder}.}
\newblock \showarticletitle{Balanced allocations with heterogenous bins}. In
  \bibinfo{booktitle}{\emph{\SPAA{19}{07}}}. \bibinfo{publisher}{{ACM}},
  \bibinfo{pages}{188--193}.
\newblock
\href{https://doi.org/10.1145/1248377.1248407}{\texttt{doi}}


\bibitem[Wieder(2017)]%
        {W17}
\bibfield{author}{\bibinfo{person}{Udi Wieder}.}
\newblock \showarticletitle{Hashing, Load Balancing and Multiple Choice}.
\newblock \bibinfo{journal}{\emph{Found. Trends Theor. Comput. Sci.}}
  \bibinfo{volume}{12}, \bibinfo{number}{3-4} (\bibinfo{year}{2017}),
  \bibinfo{pages}{275--379}.
\newblock
\href{https://doi.org/10.1561/0400000070}{\texttt{doi}}


\bibitem[Wydrowski et~al\mbox{.}(2024)]%
        {WKKA23}
\bibfield{author}{\bibinfo{person}{Bartek Wydrowski}, \bibinfo{person}{Robert
  Kleinberg}, \bibinfo{person}{Stephen~M. Rumble}, {and} \bibinfo{person}{Aaron
  Archer}.}
\newblock \showarticletitle{Load is not what you should balance: Introducing
  Prequal}. In \bibinfo{booktitle}{\emph{21st {USENIX} Symposium on Networked
  Systems Design and Implementation, ({NSDI'24})}}.
  \bibinfo{publisher}{{USENIX} Association}.
\newblock


\end{thebibliography}

 \clearpage 
\appendix

\section{Auxiliary Inequalities}
In this section we collect some useful inequalities used throughout the paper.
\paragraph{Deterministic Inequalities} The following lemma is similar to~\cite[Lemma A.1]{FGS12}.

\begin{lem}\label{lem:quasilem}Let the sequences $(a_k)_{k=1}^n , (b_k)_{k=1}^n $ be non-negative and the sequence $(c_k)_{k=1}^n$ be non-negative and non-increasing. If $\sum_{k=1}^i a_k \leq \sum_{k=1}^i b_k$ holds for all $i \in [n]$ then, \begin{equation} \label{eq:toprove}\sum_{k=1}^n a_k\cdot c_k \leq \sum_{k=1}^n b_k\cdot c_k.\end{equation}
\end{lem}
\begin{proof}
	We shall prove \eqref{eq:toprove} holds by induction on $n\geq 1$. The base case $n=1$ follows immediately from the fact that $a_1\leq b_1$ and $c_1\geq 0$. Thus we assume $\sum_{k=1}^{n-1} a_k\cdot c_k \leq \sum_{k=1}^{n-1} b_k\cdot c_k$ holds for all sequences $(a_k)_{k=1}^{n-1}, (b_k)_{k=1}^{n-1}$ and $ (c_k)_{k=1}^{n-1}$ satisfying the conditions of the lemma.
	
	For the inductive step, suppose we are given sequences $(a_k)_{k=1}^{n}$,  $(b_k)_{k=1}^{n}$ and $(c_k)_{k=1}^{n}$ satisfying the conditions of the lemma. If $c_2=0$ then, since $(c_k)_{k=1}^n$ is non-negative and non-increasing, $c_k=0$ for all $k\geq 2$. Thus as $a_1\leq b_1$ and $c_1\geq 0$ by the precondition of the lemma, we conclude 
\[ \sum_{k=1}^n a_k\cdot c_k = a_1 \cdot c_1 \leq b_1\cdot c_1 =\sum_{k=1}^n b_k\cdot c_k. \] We now treat the case $c_2>0 $. Define the non-negative sequences $(a_k')_{k=1}^{n-1}$ and $(b_k')_{k=1}^{n-1}$ 
	as follows:
	\begin{itemize}
		\item $a'_{1} = \frac{c_1}{c_{2}} \cdot a_{1} + a_2$ and $a'_{k} = a_{k+1}$ for $2\leq k \leq n-1$ ,
		 \item $b'_{1} = \frac{c_1}{c_{2}} \cdot b_{1} + b_2$ and $b'_{k} = b_{k+1}$ for $2\leq k \leq n-1$,
	\end{itemize} Then as the inequalities $c_1\geq c_2 $, $a_1 \leq b_1$ and $\sum_{i=1}^n a_k \leq \sum_{i=1}^n b_k$ hold by assumption, we have  
\[ \sum_{k=1}^{n-1}a_k' = \left(\frac{c_{1}}{c_{2}} - 1  \right)a_1 + \sum_{k=1}^{n}a_k  \leq \left(\frac{c_{1}}{c_{2}} - 1  \right)b_1 +  \sum_{k=1}^{n}b_k   = \sum_{k=1}^{n-1}b_k'. \] Thus if we also let $(c_k')_{k=1}^{n-1} = (c_{k+1})_{k=1}^{n-1} $, which is positive and non-increasing, then 
\[\sum_{k=1}^{n-1} a_k'\cdot c_k'  \leq \sum_{k=1}^{n-1} b_k'\cdot c_k',\]by the inductive hypothesis. However \[\sum_{k=1}^{n-1} a_k'\cdot c_k' = \left(\frac{c_1}{c_{2}} \cdot a_{1} + a_2 \right)c_2 +  \sum_{k=2}^{n-1} a_{k+1}\cdot c_{k+1} = \sum_{k=1}^{n} a_k\cdot c_k, \] and likewise $\sum_{k=1}^{n-1} b_k'\cdot c_k' = \sum_{k=1}^{n} b_k\cdot c_k$. The result follows.\end{proof}

\begin{clm} \label{clm:eps_ineq}
For any $\eps \in \big(0,\frac{1}{2}\big)$, we have $(1 - \frac{\eps}{2}) /(1 -\eps) \geq 1 + \frac{\eps}{2}$.
\end{clm}
\begin{proof}
For $\eps \in \big(0,\frac{1}{2}\big)$ the following chain of double implications holds
\[
\frac{1 - \frac{\eps}{2}}{1 -\eps} \geq 1 + \frac{\eps}{2} \quad \Leftrightarrow \quad  1 - \frac{\eps}{2} \geq (1 -\eps) \cdot \Big(1 + \frac{\eps}{2}\Big) = 1 - \frac{\eps}{2} + \frac{\eps^2}{2} \quad \Leftrightarrow \quad \frac{\eps^2}{2} \geq 0.
\qedhere \]
\end{proof}

\paragraph{Probabilistic Inequalities} We begin by stating two well-known concentration inequalities. 

 \begin{lem}[Method of Bounded Independent Differences {\cite[Corollary 5.2]{DubPan}}]\label{mobd} Let $f$ be a function of $N$ independent random variables $X^1 ,\dots , X^N$, where each $X^i$ takes values in a set $\Omega^i$. Assume that for each $i \in [N]$ there exists a $c_i \geq 0 $ such that
	\[ \left|f(x_1 ,\dots,x_{i-1}, x_{i},x_{i+1},\dots, x_N) - f(x_1 ,\dots,x_{i-1}, x_{i}',x_{i+1},\dots, x_N) \right| \leq  c_i,\] for any $x_1\in \Omega_1,\dots,x_{i-1}\in \Omega_{i-1}, x_{i},x_{i}'\in \Omega_{i} ,x_{i+1}\in \Omega_{i+1},\dots, x_N \in \Omega_N$. Then, for any $\lambda>0$,
	\[\Pro{f < \Ex{f} - \lambda }  \leq \exp\left(- \frac{\lambda^2}{2 \cdot \sum_{i=1}^N c_i^2} \right).\]\end{lem}

\begin{lem}[Azuma's Inequality for Super-Martingales {\cite[Problem 6.5]{DubPan}}] \label{lem:azuma}
Let $X^0, \ldots, X^N$ be a super-martingale with filtration $\mathfrak{F}^0, \ldots, \mathfrak{F}^N$. Assume that for each $i \in [N]$ there exists $c_i \geq 0$ such that 
\[
\left( \left|X^{i} - X^{i-1}\right| ~\big|~ \mathfrak{F}^{i-1} \right)  \leq c_i.
\]
Then, for any $\lambda > 0$,
\[
\Pro{X^N \geq X^0 + \lambda} \leq \exp\left(- \frac{\lambda^2}{2 \cdot \sum_{i=1}^N c_i^2} \right).
\]
\end{lem}

\begin{lem}[{\cite[Theorems~6.1 \& 6.5]{CL06}}] \label{lem:cl06_thm_6_1}
Let $X^0, \ldots, X^N$ be a martingale with filtration $\mathfrak{F}^0, \ldots, \mathfrak{F}^N$. Assume that for each $i \in [n]$,
\[
\left( \left|X^{i} - X^{i-1}\right| ~\big|~ \mathfrak{F}^{i-1} \right) \leq M,
\]
and additionally,
\[
\Var{X^i \mid \mathfrak{F}^{i-1}} \leq \sigma_i^2.
\]
Then, for any $\lambda > 0$,
\[
\Pro{\left| X^N - \Ex{X^N}\right| \geq \lambda} \leq 2 \cdot \exp\left(- \frac{\lambda^2}{2 \cdot (\sum_{i=1}^N \sigma_i^2 + M\lambda/3 )} \right).
\]
\end{lem}

Next, we state the following well-known inequality for a sequence of random variables, whose expectations are related through a recurrence inequality.

\begin{lem} \label{lem:geometric_arithmetic}
Consider any sequence of random variables $(X^i)_{i \in \mathbb{N}}$ for which there exist $a > 0$ and $b > 0$, such that every $i \geq 1$,
\[
\Ex{X^i \mid X^{i-1}} \leq X^{i-1} \cdot a + b.
\]
Then, $(i)$~for every $i \geq 0$, 
\[
\Ex{X^i \mid X^0} \leq X^0 \cdot a^i + b \cdot \sum_{j = 0}^{i-1} a^j.
\]
Further, $(ii)$~for $a \in (0, 1)$ and for every $i \geq 0$,
\[
\Ex{X^i \mid X^0}
\leq X^0 \cdot a^i + \frac{b}{1 - a}.
\]
Finally, $(iii)$~for $a \in (0, 1)$ and if $X^0 \leq \frac{b}{1-a}$ holds, then for every $i \geq 0$,
\[
\Ex{X^i} \leq \frac{b}{1 - a}.
\]
\end{lem}
\begin{proof}
\textit{First statement.} We will prove this claim by induction.
For $i = 0$, $\Ex{X^0 \mid X^0} \leq X^0$. Assuming the induction hypothesis holds for some $i \geq 0$, then since $a > 0$,
\begin{align*}
\Ex{X^{i+1} \mid X^0} & = \Ex{\Ex{X^{i+1} \mid X^i}\mid X^0} \leq \Ex{X^{i}\mid X^0} \cdot a + b \\
 & \leq \left(X^0 \cdot a^i + b \cdot \sum_{j = 0}^{i-1} a^j \right) \cdot a + b \\
 & = X^0 \cdot a^{i+1} +b \cdot \sum_{j = 0}^i a^j.
\end{align*}
\textit{Second statement.} The claims follows using that $\sum_{j = 0}^i a^j \leq \sum_{j=0}^{\infty} a^j = \frac{1}{1-a}$, for any $a \in (0,1)$.

\noindent \textit{Third statement.} We will prove this claim by induction. For $i = 0$, it follows by the assumption. Then, assuming that $\Ex{X^i} \leq \frac{b}{1-a}$ holds for $i \geq 0$, then for $i+1$
\begin{align*}
\Ex{X^{i+1}}
  & = \Ex{\Ex{X^{i+1} \mid X^{i}}}  \leq \Ex{X^{i} } \cdot a + b \leq \frac{b}{1-a} \cdot a + b = \frac{b}{1-a}. \qedhere 
\end{align*}
\end{proof}

\section{Counterexample for the Exponential Potential Function}

In this section, we present a load vector for which the potential function $\Lambda := \Lambda(\alpha)$ for any constant $\alpha \in (0, 1)$ may increase in expectation over one round, if we do not condition on any ``good event''.

In contrast to processes with constant probability bias in every round, such as those studied in~\cite{PTW15}, for the \MeanThinning process, there exists configurations where the potential $\Lambda$ for constant $\alpha$ increases in expectation over a single round, even when it is $\omega(n)$. This is because in the worst-case, the \MeanThinning process may have a very small $\Theta(1/n^2)$ probability bias to allocate away from overloaded bins, corresponding to the \textsc{Towards-Min} process in~\cite{PTW15}.

\begin{clm} \label{clm:bad_configuration_lambda}
For any constant $\alpha > 0$ and for sufficiently large $n := n(\alpha) > 0 $ consider the (normalized) load configuration at some round $t \geq 0$,
\[
y^t = \left(n^2, n, n, \ldots, n , - \frac{n\cdot (2n-3)}{2}, - \frac{n \cdot (2n-3)}{2} \right).
\]
Then for the \MeanThinning process, $(i)$~the exponential potential $\Lambda := \Lambda(\alpha)$ increases in expectation over the next round, i.e.,
\[
\Ex{\left. \Lambda^{t+1} \,\right|\, \mathfrak{F}^t} \geq \Lambda^t \cdot \left( 1 + 0.2 \cdot \frac{\alpha^2}{n}\right),
\]
and $(ii)$~the mean quantile satisfies $\delta^s \geq 1 - 2/n$ for all rounds $s \in [t, t + n^2)$.
\end{clm}
\begin{proof}
\textit{First statement.} Consider a labeling of the bins so that the normalized loads are sorted in non-increasing order. We will show that just the expected increase of bin $i = 1$ with maximum load $y_1^t = n^2$ implies an expected increase for the entire potential.  The probability of allocating to that bin is $\frac{1 - \frac{2}{n}}{n}$, so using the Taylor estimate $e^z \geq 1 + z + 0.3 z^2$ for $z \geq -1.5$,
\begin{align*}
\Ex{\left. \Lambda_1^{t+1} \,\right|\, \mathfrak{F}^t} 
 & = e^{\alpha n^2} \cdot e^{-\alpha / n} \cdot \left(1 + (e^{\alpha} -1) \cdot \frac{1 - \frac{2}{n}}{n} \right) \\
 & \geq e^{\alpha n^2} \cdot \left(1 -\frac{\alpha}{n} + 0.3 \cdot  \frac{\alpha^2}{n^2}\right) \cdot \left(1 + (\alpha + 0.3 \alpha^2) \cdot \frac{1 - \frac{2}{n}}{n} \right) \\
 & = e^{\alpha n^2} \cdot \Big(1 -\frac{\alpha}{n} + (\alpha + 0.3\alpha^2) \cdot \frac{1}{n} + o(n^{-1}) \Big) \\
 & = e^{\alpha n^2} \cdot \left(1 +0.4 \cdot \frac{\alpha^2}{n} \right).
\end{align*}
To relate this change to the expected change of the rest of the bins, we note that for sufficiently large $n$ and since $\alpha$ is constant,
\begin{align}
0.2 \cdot \frac{\alpha^2}{n} \cdot e^{\alpha n^2} 
 & \geq \left(1 + 0.2 \cdot \frac{\alpha^2}{n} \right) \cdot n \cdot e^{\alpha n^2} \cdot e^{-\alpha \cdot 3n/2} 
= \left(1 + 0.2 \cdot \frac{\alpha^2}{n} \right) \cdot n \cdot e^{\alpha n(2n-3)/2} \notag \\
 & \geq \left(1 + 0.2 \cdot \frac{\alpha^2}{n} \right) \cdot \left( \sum_{i > 1} \Lambda_i^t \right),  \label{eq:first_term_dominates}
\end{align}
since $\Lambda_i^t \leq e^{\alpha n(2n-3)/2}$ for $i > 1$. Hence,
\begin{align*}
\Ex{\left. \Lambda^{t+1} \,\right|\, \mathfrak{F}^t} 
 & \geq \Ex{\left. \Lambda_1^{t+1} \,\right|\, \mathfrak{F}^t} \geq e^{\alpha n^2} \cdot \left(1 + 0.4 \cdot \frac{\alpha^2}{n} \right)
 \\ &= e^{\alpha n^2} \cdot \left(1 + 0.2 \cdot \frac{\alpha^2}{n} \right) + 0.2 \cdot \frac{\alpha^2}{n} \cdot e^{\alpha n^2} \\
 & \!\!\stackrel{(\ref{eq:first_term_dominates})}{\geq} \Lambda_1^t \cdot \left(1 + 0.2 \cdot \frac{\alpha^2}{n} \right)+ \left( \sum_{i > 1} \Lambda_i^t \right) \cdot \left(1 + 0.2 \cdot \frac{\alpha^2}{n} \right)\\
 &= \Lambda^t \cdot \left(1 + 0.2 \cdot \frac{\alpha^2}{n} \right).
\end{align*}
\textit{Second statement.} Since there are $n-2$ overloaded bins with overload at least $n$ and these can decrease by at most $1/n$ in each round, we have that $\delta^s \geq 1 - 2/n$ for any $s \in [t, t + n^2)$.
\end{proof}

\end{document}